\documentclass[amsthm]{elsart}

\usepackage{yjsco}
\usepackage{amsfonts}
\usepackage{amssymb}
\usepackage{amsmath}
\usepackage{graphicx}
\usepackage{lscape}
\usepackage{mathtools} 

\usepackage[numbers,square,sort,comma]{natbib}

\newtheorem{example}{Example}
\newtheorem{defi}{Definition}
\newtheorem{lemma}{Lemma}
\newtheorem{theorem}{Theorem}
\newtheorem{proposition}{Proposition}
\newtheorem{corollary}{Corollary}

\makeatletter
\def\bbordermatrix#1{\begingroup \m@th
  \@tempdima 4.75\p@
  \setbox\z@\vbox{%
    \def\cr{\crcr\noalign{\kern2\p@\global\let\cr\endline}}%
    \ialign{$##$\hfil\kern2\p@\kern\@tempdima&\thinspace\hfil$##$\hfil
      &&\quad\hfil$##$\hfil\crcr
      \omit\strut\hfil\crcr\noalign{\kern-\baselineskip}%
      #1\crcr\omit\strut\cr}}%
  \setbox\tw@\vbox{\unvcopy\z@\global\setbox\@ne\lastbox}%
  \setbox\tw@\hbox{\unhbox\@ne\unskip\global\setbox\@ne\lastbox}%
  \setbox\tw@\hbox{$\kern\wd\@ne\kern-\@tempdima\left[\kern-\wd\@ne
    \global\setbox\@ne\vbox{\box\@ne\kern2\p@}%
    \vcenter{\kern-\ht\@ne\unvbox\z@\kern-\baselineskip}\,\right]$}%
  \null\;\vbox{\kern\ht\@ne\box\tw@}\endgroup}
\makeatother

\makeatletter
\newcommand{\superimpose}[2]{%
  {\ooalign{$#1\@firstoftwo#2$\cr\hfil$#1\@secondoftwo#2$\hfil\cr}}}
\makeatother

\newcommand{\cgtop}{\mathpalette\superimpose{{\bigcirc}{t}}}
\newcommand{\cgbot}{\mathpalette\superimpose{{\bigcirc}{b}}}
\newcommand{\cgequal}{\mathpalette\superimpose{{\bigcirc}{=}}}

\newcommand{\GL}{\operatorname{GL}}

\begin{document}
\begin{frontmatter}

\title{\Large{Polynomial-time Solvable \#CSP Problems \\ via Algebraic Models and Pfaffian Circuits}}

\vspace{-25pt}

\author{S. Margulies}
\address{Department of Mathematics,
United States Naval Academy,
Annapolis, MD
}
\ead{margulie@usna.edu}

\author{J. Morton}
\address{Department of Mathematics,
Pennsylvania State University,
State College, PA
}
\ead{morton@math.psu.edu}

\begin{abstract}
A \emph{Pfaffian circuit} is a tensor contraction network where the edges are labeled with changes of bases in such a way that a very specific set of combinatorial properties are satisfied. By modeling the permissible changes of bases as systems of polynomial equations, and then solving via computation, we are able to identify classes of 0/1 planar \#CSP problems solvable in polynomial-time via the Pfaffian circuit evaluation theorem (a variant of L. Valiant's Holant Theorem). We present two different models of 0/1 variables, one that is possible under a homogeneous change of basis, and one that is possible under a heterogeneous change of basis only. We enumerate a series of 1,2,3, and 4-arity gates/cogates that represent constraints, and define a class of constraints that is possible under the assumption of a ``bridge" between two particular changes of bases. We discuss the issue of planarity of Pfaffian circuits, and demonstrate possible directions in algebraic computation for designing a Pfaffian 
tensor contraction network fragment that can simulate a \textsc{swap} gate/cogate. We conclude by developing the notion of a \emph{decomposable} gate/cogate, and discuss the computational benefits of this definition.
\end{abstract}


\end{frontmatter}

\section{Introduction} \label{sec_intro} 

A solution to a \emph{constraint satisfaction problem} (CSP) is an assignment of values to a set of variables such that certain constraints on the combinations of values are satisfied. A solution to a \emph{counting constraint satisfaction problem} (\#CSP) is the number of solutions to a given CSP. For example, the classic NP-complete problem 3-SAT is a CSP problem, but counting the number of satisfying assignments is a \#CSP problem. CSP and \#CSP problems are ubiquitous in computer science, optimization and mathematics: they can model problems in fields as diverse as Boolean logic, graph theory, database query evaluation, type inference, scheduling and artificial intelligence. This paper uses computational commutative algebra to study the classification of \#CSP problems solvable in polynomial-time by \emph{Pfaffian circuits}.

In a seminal paper \cite{valiant_ha}, Valiant used ``matchgates" to demonstrate polynomial-time algorithms for a number of \#CSP problems, where none had been previously known before. 
Pfaffian circuits \cite{ lands_hol_wo_mg, morton_pfaff, morton2013generalized} are a simplified and extensible reformulation of Valiant's notion of a holographic algorithm, which builds on J.Y. Cai and V. Choudhary's work in expressing holographic algorithms in terms of tensor contraction networks \cite{cai_mg_ten}. Valiant's Holant Theorem (\cite{valiant_qc, valiant_ha}) is an equation where the left-hand side has both an exponential number of terms and an exponential number of term cancellations, whereas the right-hand side is computable in polynomial-time.  Extensions to holographic algorithms made possible by Pfaffian circuits include swapping out this equation for another combinatorial identity \cite{morton2013generalized}, viewing this equation as an equivalence of monoidal categories \cite{morton2013generalized}, or, as is done here, using heterogeneous changes of bases with the aid of computational commutative algebra.

In a series of  papers (\cite{bulatov_dalmau_maltsev, bulatov_dalmau_towdic, bulatov_grohe}) culminating in \cite{bulatov_dic,bulatov2013complexity}, A. Bulatov explored the problem of counting the number of homomorphisms between two finite relational structures (an equivalent statement of the general \#CSP problem). In \cite{bulatov_dic,bulatov2013complexity}, Bulatov demonstrates a complete dichotomy theorem for \#CSP problems. In other words, Bulatov demonstrates that a \#CSP problem is either in FP (solvable in polynomial time), or it is \#P-complete. However, not only does his paper rely on an thorough knowledge of universal algebras, but it relies on the notion of a \emph{congruence singular} finite relational structure, and the complexity of determining if a given structure has this property is unknown (perhaps even undecidable). However, in 2010, Dyer and Richerby \cite{dyer2010effective} offered a simplified proof of Bulatov's result, and furthermore established a new criterion (known as \emph{
strong balance}) for the \#CSP dichotomy that is not only decidable, but is in fact in NP.

In light of these elegant and conclusive dichotomy theorems, research on \#CSP problems focused in a different direction. The dichotomy theorems  were  specialized to categorize both 0/1 and finite alphabets  \cite{bulatov_fa, cai2012complexity}, and also specialized for restricted input cases such as planar instances, symmetric signatures, or homogeneous change of basis \cite{cai_lu_xia:dichotomy}. This paper begins the process of developing a dichotomy theorem for  Pfaffian circuits under a non-homogeneous (heterogeneous) change of basis by identifying classes of symmetric and asymmetric planar 0/1 polynomial-time solvable instances via algebraic methods and computation.  It is the first systematic exploration of the heterogeneous basis case.

For clarity, we list four main reasons for this approach. First, since the time complexity of determining the Pfaffian of a matrix is equivalent to that of finding the determinant, our approach is not only polynomial-time, but $O(n^{\omega_p})$ (where $1.19 \leq \omega_p \leq 3$ and $n$ is the total number of variable inclusions in the clauses). We observe that Valiant's matchgate approach has a similar time complexity; however, the tensor contraction network representation of the underlying \#CSP instance is often a more compact representation than that of the matchgate encoding. Second, our approach is based on algebraic computational methods, and thus, any independent innovations to Gr\"obner basis algorithms (or determinant algorithms, for that matter) will make it easier to classify problems. Third, by approaching Pfaffian circuits via algebraic computation, we are able to consider non-homogeneous (heterogeneous) changes of bases, an option which has not yet been explored. The use of a heterogeneous 
changes of bases brings us to our final reason for considering Pfaffian circuits: the gates/cogates may be decomposable into smaller gates/cogates in such a way that lower degree polynomials are used, which exponentially enhances the performance of this method.

We begin Sec. \ref{sec_back} with a detailed introduction to tensor contraction networks, which can be skipped by those already familiar with these ideas. We next recall the definition of Pfaffian gates/cogates, and their connection to classical logic gates (OR, NAND, etc.).  In Sec. \ref{sec_back_pfaff_eval}, we define Pfaffian circuits, and give a step-by-step example of evaluating a Pfaffian circuit via the polynomial-time Pfaffian circuit evaluation theorem \cite{morton_pfaff}. In Sec. \ref{sec_alg_meth}, we present the new algebraic and computational aspect of this project: given a gate/cogate, we describe a system of polynomial equations such that  the solutions (if any) are in bijection to the changes of bases (not necessarily homogeneous) where the gate/cogate is ``Pfaffian". By ``linking" the ideals associated with these gates/cogates together, and then solving using software such as \textsc{Singular} \cite{singular}, we begin the process of characterizing the building blocks of Pfaffian circuits.

In Sec. \ref{sec_bool_ten}, we present the first results of our computational exploration. We demonstrate two different ways of simulating 0/1 variables as planar, Pfaffian, tensor contraction network fragments. The first uses a homogeneous change of a basis, and the second (known as a Boolean tree) is possible under a heterogeneous change of basis only. We use the Boolean trees to develop two classes of compatible constraints, the first of which identifies gates/cogates that are Pfaffian under a homogeneous change of basis, and the second of which utilizes the two existing changes of bases ($A$ and $B$), and then posits the existence of a third change of basis $C$, to identify 24 additional Pfaffian gates/cogates. Thus, Sec. \ref{sec_bool_ten} begins the process of characterizing (via algebraic computation) planar, Pfaffian, 0/1 \#CSP problems that are solvable in polynomial-time.

In Sec. \ref{sec_swap}, we investigate the question of planarity. Within the Pfaffian circuit framework, the addition of a \textsc{swap} gate/cogate is \emph{not} equivalent to lifting the planarity restriction, since the compatible gates/cogates representing solvable \#CSP problems are not automatically identifiable. Thus, there are no inherent complexity-theoretic stumbling blocks to investigating a Pfaffian \textsc{swap} gate/cogate, and indeed the hope is that such an investigation would eventually yield a new sub-class of non-planar poloynomial-time solvable \#CSP problems. However, in this paper, we only demonstrate that specific gates which can be used as building blocks for a \textsc{swap} gate (such as \textsc{CNOT}) are indeed Pfaffian (under a heterogeneous change of basis only). We then describe several attempts to construct a Pfaffian \textsc{swap} gate, indicating precisely where the attempts fail, and conclude by presenting a \emph{partial} \textsc{swap} gate. This result suggests a specific 
direction (in both algebraic computation and combinatorial structure), for constructing a Pfaffian \textsc{swap} gate. We conclude in Sec. \ref{sec_decomp} by introducing the notion of a \emph{decomposable} gate/cogate, and discussing the computational advantages of gate decompositions.

To summarize, this project models Pfaffian circuits as systems of polynomial equations, and then solves the systems using \textsc{Singular} \cite{singular}, with the goal of identifying classes of planar 0/1 \#CSP problems that are solvable in polynomial-time. 

%

\section{Background and Definitions of Pfaffian Circuits} \label{sec_back} In this section, we develop the necessary background for modeling \#CSP problems as Pfaffian circuits, and then solving them via the Pfaffian circuit evaluation theorem \cite{morton_pfaff}. We begin with tensors and the convenience of the Dirac (bra/ket) notation from quantum mechanics, and then express Boolean predicates (which we call gates/cogates) as elements of a tensor product space. We then describe the process of applying a change of basis to these gates/cogates such that they are expressible as tensors with coefficients that are the sub-Pfaffians of some skew-symmetric matrix. We conclude with a step-by-step example of solving a particular Pfaffian circuit with the Pfaffian circuit evaluation theorem.

\subsection{Tensors and Dirac Notation} \label{sec_back_ten} 
Let $U$ and $V$ be two-dimensional complex vector spaces equipped with bases, with  $v \in V$ and $u \in U$. In the induced basis we can express the tensor product $v \otimes u$ as the Kronecker product, the vector $w \in \mathbb{C}^{4}$ ($w =[w_0,w_1,w_2,w_3]^T$) with $w_{2i + j} = v_iu_j$ for $0\leq i,j\leq 1$. For example,
{\footnotesize{
\[\arraycolsep=1.4pt\def\arraystretch{1.0}
\left[\begin{array}{c}
3\\
1/3 + i
\end{array}\right] \otimes \left[\begin{array}{c}
1 + 2i\\
1/2
\end{array}\right] = \left[\begin{array}{c}
\hspace{13pt}3\hspace{15pt}\left[\begin{array}{c}
1 + 2i\\
1/2
\end{array}\right]\\
(1/3 + i)\left[\begin{array}{c}
1 + 2i\\
1/2
\end{array}\right]\\
\end{array}\right] = \left[\begin{array}{c}
3+6i\\
3/2\\
-5/3 + 5/3i\\
1/6 + i/2
\end{array}\right]~.
\]}}
Now let $V_1,\ldots, V_n$ each be isomorphic to $\mathbb{C}^{2}$, and let $\{v^0_i, v^1_i\}$ be a basis of $V_i$. Any induced basis vector in the tensor product space $V_1 \otimes V_2 \otimes \cdots \otimes V_n$ can be concisely written as a \emph{ket}, where $v^0$ is denoted by $|0\rangle$ and $v^1$ is denoted by $|1\rangle$. For example, a basis element of  $V_1 \otimes V_2 \otimes \cdots \otimes V_n$ can be written as
\begin{align*}
\underbrace{|1101\cdots 01\rangle}_{\text{ket}}~ &= v^{\mathbf{1}}_1 \otimes v^{\mathbf{1}}_2 \otimes v^{\mathbf{0}}_3 \otimes v^{\mathbf{1}}_4 \otimes \cdots \otimes  v^{\mathbf{0}}_{n-1} \otimes v^{\mathbf{1}}_n~,
\end{align*}
and an arbitrary vector $w \in V_1 \otimes V_2 \otimes \cdots \otimes V_n$ can be written as
\begin{align*}
w &= \alpha_1\underbrace{|00\cdots 0\rangle}_{\text{ket}} + \alpha_2\underbrace{|10\cdots 0\rangle}_{\text{ket}} + \cdots + \alpha_{2^n}\underbrace{|11\cdots 1\rangle}_{\text{ket}}~, \quad \text{with~} \alpha_1,\cdots,\alpha_{2^n} \in \mathbb{C}~.
\end{align*}


Given a vector space $V$, the \emph{dual vector space} $V^{*}$ is the vector space of all linear functions on $V$. Given  $V_1,\ldots, V_n$, let $V^{*}_1, \ldots, V^{*}_n$ be their duals, with $\{\nu^{0*}_i, \nu^{1*}_i\}$ the dual basis of $V^{*}_i$. We write a linear function in the tensor product space $V^{*}_1 \otimes \cdots \otimes V^{*}_n$ as a \emph{bra}, with $\nu^{0*}$ denoted by $\langle 0|$ and $\nu^{1*}$ denoted by $\langle 1|$. For example, a basis element can be written as
\begin{align*}
\underbrace{\langle 1101\cdots 01|}_{\text{bra}} &= \nu^{\mathbf{1}*}_1 \otimes \nu^{\mathbf{1}*}_2 \otimes \nu^{\mathbf{0}*}_3 \otimes \nu^{\mathbf{1}*}_4 \otimes \cdots \otimes  \nu^{\mathbf{0}*}_{n-1} \otimes \nu^{\mathbf{1}*}_n~,
\end{align*}
and an arbitrary linear function $w \in V^{*}_1 \otimes V^{*}_2 \otimes \cdots \otimes V^{*}_n$ can be written as
\begin{align*}
w &= \beta_1\underbrace{\langle 00\cdots 0|}_{\text{bra}} + \beta_2\underbrace{\langle 10\cdots 0|}_{\text{bra}} + \cdots + \beta_{2^n}\underbrace{\langle 11\cdots 1|}_{\text{bra}}~, \quad \text{with~} \beta_1,\cdots,\beta_{2^n} \in \mathbb{C}~.
\end{align*}
We may use subscripts to identify sub-tensor-products: for example, $|0_61_90_{11} \rangle$ denotes the induced basis element $v^0_6 \otimes v^1_9 \otimes v^0_{11}$ in the tensor product space $V_6 \otimes V_9 \otimes V_{11}$ (and similarly for the bras). 

The dual basis element $\nu^{j*}_i$ with $j \in \{0,1\}$ yields a Kronecker delta function:
\[
\nu^{j*}_i(v^k_i) = \begin{cases}
1 & \text{if $j=k$~,}\\
0 & \text{otherwise~.}
\end{cases}
\]
In the bra-ket notation, this is expressed as a \emph{contraction}; e.g.\  $\langle 0 | 0 \rangle = 1$ and $\langle 0 | 1 \rangle = 0$. 

In classical computing, a bit takes on the value of 0 or 1. In quantum computing, given an orthonormal basis $\big\{|0\rangle, |1\rangle\big\}$, a pure state of a qubit is given by the superposition of states, denoted $\alpha |0\rangle + \beta |1\rangle$, with $\alpha^2+ \beta^2=1$ and $\alpha, \beta \in \mathbb{C}$. 

\subsection{Boolean Predicates as Tensor Products} \label{sec_back_gcg} 

 A \emph{Boolean predicate} is a 0/1-valued function where the true/false output is dependent on the true/false input assignments of the variables. 
Here we see the 2-input Boolean predicate OR represented as both a bra and a ket.

\begin{minipage}{.25\linewidth}
{\footnotesize
\begin{center}
\text{OR} = \begin{tabular}{c|c||c}
0 & 0 & 0\\ [-1.5ex]
1 & 0 & 1\\ [-1.5ex]
0 & 1 & 1\\ [-1.5ex]
1 & 1 & 1
\end{tabular}
\end{center}}
\end{minipage}
\begin{minipage}{.75\linewidth}
\vspace{-10pt}
\begin{align*}
\text{OR (as a ket)} &= 0\cdot |00\rangle + 1\cdot|10\rangle + 1\cdot|01\rangle + 1\cdot|11\rangle~\\
 &= | 10\rangle +  | 01\rangle +  | 11\rangle~, \quad \text{and}\\
\text{OR (as a bra)}&= \langle 10| +  \langle 01| +  \langle 11|~.
\end{align*}
\end{minipage}

\noindent A Boolean predicate represented as a ket is a \emph{gate}, and a Boolean predicate represented as a bra is a \emph{cogate}. Just as Boolean predicates  can be connected to describe a (counting) constraint satisfaction problems, these gates and cogates can be connected to describe a {\em tensor contraction network}.

\subsection{Tensor Contraction Networks} \label{sec_back_tcn} 
A bipartite graph $G = \{X,Y,E\}$ is a graph partitioned into two disjoint vertex sets, $X$ and $Y$, such that every edge in the graph is incident on a vertex in both $X$ and $Y$. A \emph{bipartite tensor contraction network} $\Gamma$ is a bipartite graph partitioned into gates and cogates. If $\Gamma$ contains $m$ edges, we consider the vector spaces $V_1,\ldots, V_m$ (and vector space duals $V^{*}_1,\ldots, V^{*}_m$) , and every vertex in $\Gamma$ is labeled with either a gate or cogate. Consider a vertex of degree $d$ which is incident on edges $e_1,\ldots,e_d$. Then, the gate (or cogate) associated with that vertex is an element of the tensor product space $V_{e_1}\otimes \cdots \otimes V_{e_d}$ (or $V^{*}_{e_1}\otimes \cdots \otimes V^{*}_{e_d}$, respectively). We denote the tensor contraction network $\Gamma$ as the 3-tuple $\{G, C, E\}$ consisting of gates, cogates and edges. Every tensor contraction network considered in this paper is bipartite.
\begin{example} \label{ex_tcn}
\emph{Consider the following tensor contraction network $\Gamma$ with arbitrary gates/cogates:}

\vspace{-8pt}
\begin{minipage}{.20\linewidth}
\begin{center}
\includegraphics[scale=.25,clip=true,trim=75 500 220 50]{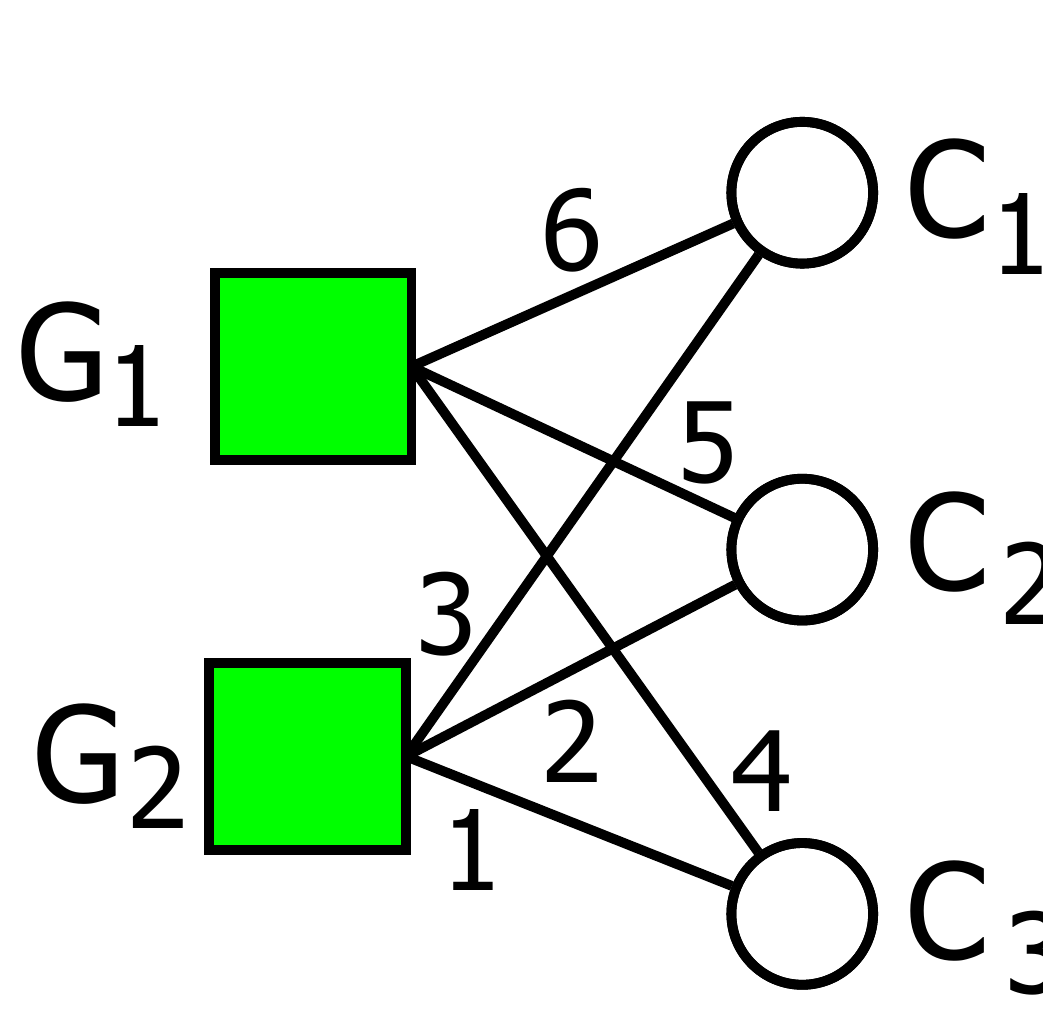}
\end{center}
\end{minipage}
\begin{minipage}{.55\linewidth}
\begin{align*}
 \quad \Gamma &= \begin{cases}
G = \begin{cases}
G_1 = |0_40_50_6 \rangle + |1_41_51_6 \rangle~,\\
 G_2 = |1_10_21_3\rangle + |0_11_21_3\rangle + |1_11_21_3\rangle~,
\end{cases}\\
C = \begin{cases}
C_1 = \langle 0_31_6| + \langle 1_30_6|~,\\
C_2 = \langle 0_21_5| + \langle 1_20_5|~,\\
C_3 = \langle 1_11_4|~,
\end{cases}\\
E = \{1,\ldots, 6\}~.
\end{cases}
\end{align*}
\end{minipage}

\emph{
\noindent Gates are denoted by boxes ($\Box$) and cogates by circles ($\circ$). Note that gate $G_2$ is incident on edges $1,2$ and $3$, and is an element of the tensor product space $V_{1}\otimes V_2\otimes V_{3}$. Similarly, cogate $C_3$ is incident on edges $1$ and $4$, and $C_3 \in V^{*}_1\otimes V^{*}_4$. \hfill $\Box$}
\end{example}

Given a tensor contraction network $\Gamma$, the \emph{value of the tensor contraction network}, denoted $val(\Gamma)$, is the contraction of all the tensors in $\Gamma$. The value of any \emph{bipartite} tensor contraction network is a scalar.  We observe that the value of the tensor contraction network $\Gamma$ in Ex. \ref{ex_tcn} is $0$ (see Sec \ref{sec_back_pfaff_eval} for a contraction example).

\subsection{Pfaffian Gates, Cogates and Circuits} \label{sec_back_pfaff} 
In the previous section, we defined gates/cogates as the fundamental building blocks of tensor contraction networks. In this section, we describe how to find the value of a tensor contraction network in polynomial-time when the network satisfies certain combinatorial and algebraic conditions. In particular, we describe what it means for an individual gate/cogate to be \emph{Pfaffian}.

An $n \times n$ \emph{skew-symmetric} matrix $A$ has $a_{ij} = -a_{ji}$, and thus $a_{ii} = 0$.  
For $n$ odd, the determinant of a skew-symmetrix matrix $A$ is zero, and for $n$ even, $\det(A)$ can be written as the square of a polynomial known as the \emph{Pfaffian} of $A$. Given an even integer $n$, let $S^{\text{Pf}}_n \subseteq S_n$ be the set of permutations $\sigma \in S_n$ such that $\sigma(1) < \sigma(2), \sigma(3) < \sigma(4), \ldots, \sigma({n-1}) < \sigma(n)$, and $\sigma(1) < \sigma(3) < \sigma(5) < \cdots < \sigma({n-1})$. Then, the \emph{Pfaffian} of an $n \times n$ skew-symmetric matrix $A$, denoted $\text{Pf}(A)$, is given by
\vspace{-10pt}
\begin{displaymath}
\text{Pf}(A) = \begin{cases}
0~, & \text{for $n$ odd}~,\\
1~, & \text{for $n = 0$}~,
\end{cases}~,  \quad \text{and} \quad 
\text{Pf}(A) = \overbrace{\sum\limits_{\sigma \in S^{\text{Pf}}_n} \text{sgn}(\sigma)a_{\sigma(1),\sigma(2)}a_{\sigma(3),\sigma(4)}\cdots a_{\sigma({n-1}),\sigma(n)}}^{\text{for $n$ even}}~,
\end{displaymath}
and $\big(\text{Pf}(A)\big)^2 = \det(A)$. For example, if $A$ is a $2 \times 2$ matrix, then $\text{Pf}(A) = a_{12}$. If $A$ is a $4 \times 4$ matrix, then $\text{Pf}(A) = a_{12}a_{34} - a_{13}a_{24} + a_{14}a_{23}$. 
Laplace expansion can also be used to compute $\text{Pf}(A)$. 
For example, if $A$ is a $6 \times 6$ matrix, then $\text{Pf}(A) =a_{12}\text{Pf}(A|_{3456}) - a_{13}\text{Pf}(A|_{2456}) + a_{14}\text{Pf}(A|_{2356}) - a_{15}\text{Pf}(A|_{2346}) + a_{16}\text{Pf}(A|_{2345})$, where $A|_{2356}$ is the submatrix of $A$ consisting only of the rows and columns indexed by $\{2,3,5,6\}$.

Let $[n] = \{1,\ldots,n\}$, and $I \subseteq [n]$. Then $|I\rangle$ is the ket corresponding to subset $I$. For example, given $I = \{2,5,7\} \subseteq [8]$, then $|I\rangle = |01001010\rangle$. Additionally, given $J = \{1,2,4,6,8\} \subseteq [8]$, then $J^C$ is the set of integers \emph{not} present in $J$, or $J^C = \{3,5,7\}$.
\begin{defi}\cite{morton_pfaff,lands_hol_wo_mg} \label{def_sPf}
Given an $n \times n$ skew-symmetric matrix $\Xi$, we define the \emph{subPfaffian} of $\Xi$ and the \emph{subPfaffian dual} of $\Xi$, denoted by $\text{\emph{sPf}}(\Xi)$  and $\text{\emph{sPf}}^{*}(\Xi)$ respectively, as
\begin{align*}
\text{\emph{sPf}}(\Xi) &= \sum_{I \subseteq [n]}\text{\emph{Pf}}(\Xi|_{I}) |I\rangle~, \quad \quad
\text{\emph{sPf}}^{*}(\Xi) = \sum_{J \subseteq [n]}\text{\emph{Pf}}(\Xi|_{J^C}) \langle J|~,
\end{align*}
where $\Xi|_{I}$ is the submatrix of $\Xi$ with rows/columns indexed by $I$ (and similarly for $\Xi|_{J^C}$).
\end{defi}
\begin{example} \label{ex_subpfaff} \emph{Given the $4 \times 4$ skew-symmetric matrix $\Xi$, we calculate $\text{\emph{sPf}}(\Xi), \text{\emph{sPf}}^{*}(\Xi)$:}

\hspace{-30pt}\begin{minipage}{.30\linewidth}
\vspace{-21pt}
\[\arraycolsep=2.4pt\def\arraystretch{1.0}
\Xi = \left[ \begin{array}{cccc}
0 & i & 0 & 2\\
-i & 0 & -1 & 0\\
0 & 1 & 0 & 3\\
-2 & 0 & -3 & 0
\end{array} \right],
\]
\end{minipage}
\begin{minipage}{.20\linewidth}
{\footnotesize{
\begin{align*}
\text{\emph{sPf}}(\Xi) &= \text{\emph{Pf}}(\Xi|_{\emptyset})|0000\rangle + \text{\emph{Pf}}(\Xi|_{1})|1000\rangle+ \cdots +  \text{\emph{Pf}}(\Xi|_{1234})|1111\rangle~\\
&= |0000\rangle + i|1100\rangle + 2|1001\rangle - |0110\rangle + 3|0011\rangle + (-2 + 3i)|1111\rangle~,\\
\text{\emph{sPf}}^{*}(\Xi) &=\text{\emph{Pf}}(\Xi|_{1234})\langle 0000|+ \text{\emph{Pf}}(\Xi|_{234})\langle1000| + \cdots +  \text{\emph{Pf}}(\Xi|_{\emptyset})\langle 1111|~\\
&=  (-2 + 3i)\langle 0000| + 3\langle 1100| - \langle 1001| +2 \langle 0110| + i \langle 0011| + \langle 1111|~.
\end{align*}}} \normalsize
\end{minipage}

\emph{We observe that, while $\text{\emph{Pf}}(\Xi) = -2 + 3i$ is a scalar, $\text{\emph{sPf}}(\Xi)$ is an element of the tensor product space $V_1 \otimes \cdots \otimes V_4$ and $\text{\emph{sPf}}^{*}(\Xi)$ is an element of $V^{*}_1 \otimes \cdots \otimes V^{*}_4$.} \hfill $\Box$
\end{example}

The value of a closed tensor network is invariant under the action of $\times_i \GL V_i$, with each $\GL V_i $ acting on the corresponding wire.  To see this explicitly, we now apply a change of basis $A$ to an edge in a tensor contraction network.

\begin{minipage}{.10\linewidth}
\begin{center}
Consider
\end{center}
\end{minipage}
\begin{minipage}{.25\linewidth}
\begin{center}
\includegraphics[scale=.20,clip=true,trim=85 625 90 90]{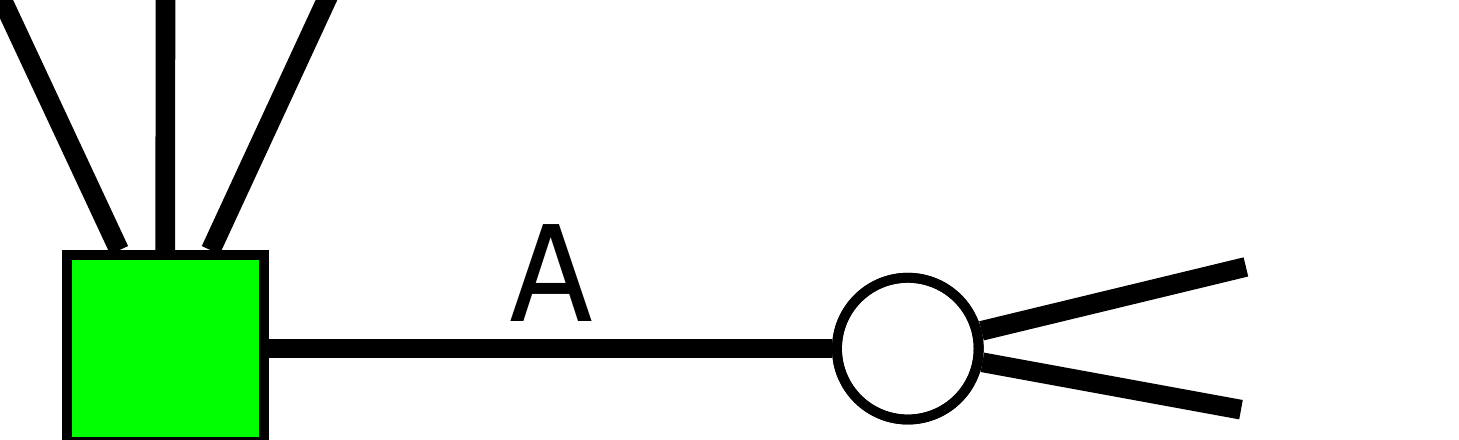}
\end{center}
\end{minipage}
\hspace{-15pt}\begin{minipage}{.25\linewidth}
\begin{align*}
\text{, where } A = \left[ \begin{array}{cc}
a_{00} & a_{01}\\
a_{10} & a_{11}
\end{array} 
\right], \text{ and } A^{-1} = \frac{1}{\det(A)} \left[ \begin{array}{cc}
a_{11} & -a_{01}\\
-a_{10} & a_{00}
\end{array} 
\right]~.
\end{align*}
\end{minipage}

The choice of indexing the rows and columns from zero is a notational convenience which will become clear in Sec. \ref{sec_alg_meth}. When the change of basis $A$ is applied to an edge, the contraction property $\langle 0|1\rangle = \langle 1|0\rangle = 0$ and $\langle 0|0\rangle = \langle 1|1\rangle = 1$ must be preserved, since applying a change of basis must not affect $val(\Gamma)$. Since every edge in a bipartite tensor contraction network is incident on both a gate and cogate, when we apply the change of basis to the gate, we find $A |0\rangle = a_{00}|0\rangle + a_{10}|1\rangle$, and $A|1\rangle = a_{01}|0\rangle + a_{11}|1\rangle$. When we apply the change of basis to the cogate, we find $A^{-1}\langle 0| = \det(A)^{-1}\big(a_{11} \langle 0| - a_{01}\langle 1|\big)$, and $A^{-1}\langle 1| = \det(A)^{-1}\big(-a_{10} \langle 0| + a_{00}\langle 1|\big)$.  Note that
\begin{align*}
1 = \langle 0 | 0 \rangle &= \Big\langle A^{-1}\langle 0|,~A|0\rangle \Big\rangle = \Big\langle \det(A)^{-1}\big(a_{11} \langle 0| - a_{01}\langle 1|\big),~a_{00}|0\rangle 
+ a_{10}|1\rangle \Big\rangle~\\
&=\frac{a_{11} a_{00}\langle 0|0\rangle + a_{11}a_{10} \langle 0|1\rangle - a_{01}a_{00}\langle 1|0\rangle  - a_{01}a_{10}\langle 1|1\rangle~}{\det(A)}
= \frac{a_{11} a_{00} - a_{01}a_{10}}{\det(A)} = 1~.
\end{align*}
Therefore, when the change of basis $A$ is applied to the gate, the change of basis $A^{-1}$ is applied to the cogate. For convenience, this is denoted as follows: 

\begin{minipage}{.10\linewidth}
\begin{center}
\phantom{xx}
\end{center}
\end{minipage}
\begin{minipage}{.30\linewidth}
\begin{center}
\includegraphics[scale=.20,clip=true,trim=85 600 80 90]{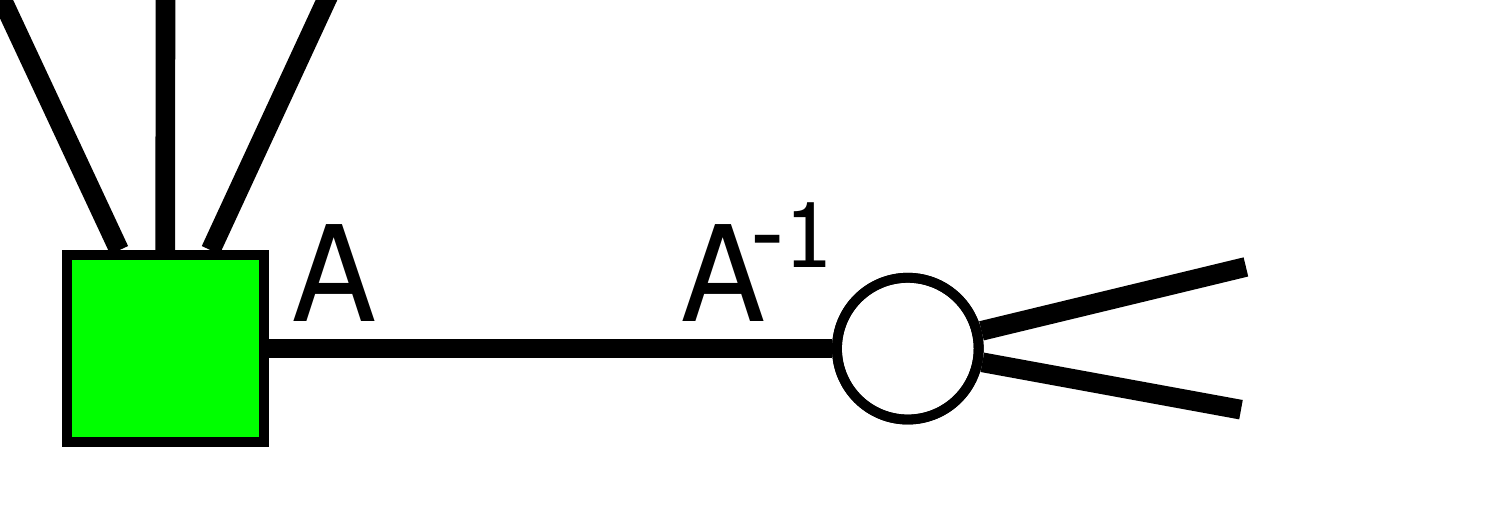}
\end{center}
\end{minipage}
\hspace{-13pt}\begin{minipage}{.20\linewidth}
\begin{center}
is denoted as
\end{center}
\end{minipage}
\begin{minipage}{.30\linewidth}
\begin{center}
\includegraphics[scale=.20,clip=true,trim=85 600 90 90]{gcg_cob_A.pdf}
\end{center}
\end{minipage}

We now present the notion of a \emph{Pfaffian} gate/cogate. Recall that in a tensor contraction network, the number of edges incident on the gate/cogate is the \emph{arity} of the gate/cogate. For example, in Ex. \ref{ex_tcn}, cogate $C_2 = \langle 0_21_5| + \langle 1_20_5|$ is incident on edges 2 and 5, and thus has \emph{arity} two.

\begin{defi}\label{def_pfaff_gcg}
An arity-$n$ gate $G$ is \emph{Pfaffian after a change of basis} if there exists an $n \times n$ skew-symmetric matrix $\Xi$, an $\alpha \in \mathbb{C}$, and matrices $A_1,\ldots,A_n \in \mathbb{C}^{2 \times 2}$ such that
\begin{align*}
(A_1 \otimes \cdots \otimes A_n) G &= \alpha\hspace{1pt}\text{sPf}(\Xi)~.
\end{align*}
An arity-$n$ cogate $C$ is \emph{Pfaffian after a change of basis} if there exists an $n \times n$ skew-symmetric matrix $\Theta$, a $\beta \in \mathbb{C}$, and matrices $A_1,\ldots,A_n \in \mathbb{C}^{2 \times 2}$ such that
\begin{align*}
(A^{-1}_1 \otimes \cdots \otimes A^{-1}_n) C &= \beta\hspace{1pt}\text{sPf}^{*}(\Theta)~.
\end{align*}
\end{defi}
When $A_1 = \cdots = A_n$, we say that $G$ is Pfaffian under a \emph{homogeneous} change of basis.  Otherwise, we say that $G$ is Pfaffian under a \emph{heterogeneous} change of basis. When no change of basis is needed, we simply say that $G$ is {\em Pfaffian} (and similarly for cogates).    
\begin{example}[Pfaffian Gates and Cogates] \label{ex_pfaff_gcg}
Consider the change of basis matrix $A$:
\[
A = \left[ \begin{array}{cc}
 \frac{1}{10}\big( -5^{3/4} + 5^{5/4} ) & 5^{-1/4}  \\
 \frac{1}{10} \big( -5^{3/4} - 5^{5/4} \big)  & 5^{-1/4}
\end{array} \right]~.
\]
Consider the OR gate $|10\rangle + |01\rangle + |11\rangle$, and observe that
\begin{align*}
(A\otimes A)\big(|10\rangle + |01\rangle + |11\rangle \big )&=|00 \rangle - |11\rangle~ =
 \alpha\hspace{2pt} \text{\emph{sPf}}(\Xi) = \text{\emph{sPf}}\left( \left[\begin{array}{cc}
0 & -1\\[-1.5ex]
1 & 0
\end{array}\right] \right)~.
\end{align*}
Additionally, consider the cogate $\langle 00| +  \langle 11|$, and observe that
\begin{align*}
\big(A^{-1} \otimes A^{-1} \big) \big ( \langle 00| +  \langle 11| \big) &= \beta\hspace{2pt}\text{\emph{sPf}}^{*}(\Theta)~ =  \underbrace{ \Big(\frac{ \sqrt{5}}{2} - \frac{1}{2} \Big)}_{\beta}\text{\emph{sPf}}^{*}
\left( \left[\begin{array}{cc}
0 &  \frac{\sqrt{5}}{2} + \frac{3}{2}\\
-\big(  \frac{\sqrt{5}}{2} + \frac{3}{2} \big) & 0
\end{array}\right] \right)~.
\end{align*}
Therefore the gate $|10\rangle + |01\rangle + |11\rangle$ and the cogate $\langle 00| +  \langle 11|$ are both Pfaffian under the homogeneous change of basis $A$. We observe that the change of basis matrix $A$ was found computationally via the algebraic method described in Sec. \ref{sec_alg_meth}. \hfill $\Box$
\end{example}

A \emph{Pfaffian circuit} \cite{morton_pfaff} is a tensor contraction network where some change of basis (possibly the identity) has been applied to every edge such that every gate/cogate in the network is Pfaffian. In the next section, we explain the importance of Pfaffian circuits.

\subsection{Pfaffian Circuits and the Pfaffian Circuit Evaluation Theorem} \label{sec_back_pfaff_eval} In this section, we explain the Pfaffian circuit evaluation theorem. Just as Valiant's Holant Theorem involves an identity where the left-hand side seems to require exponential time, while  the right-hand side can be evaluated in polynomial time,  the Pfaffian circuit evaluation theorem is a similar equation. In this section, we follow a particular example, and evaluate both the left and right-hand sides of the identity, in exponential and polynomial-time, respectively.

Consider the following Boolean formula (an example of 2-SAT) in variables $x_1,x_2,x_3$:
\begin{align}
(x_1 \vee x_2) \wedge (x_2 \vee x_3) \wedge (x_1 \vee x_3)~. \label{eq_bool_form}
\end{align}

\noindent The tensor contraction network $\Gamma$ displayed in Fig. \ref{fig_tcn_ex2} (top) is a model for the above Boolean formula.
\begin{figure}[h!]
\begin{minipage}{.25\linewidth}
\begin{center}
\includegraphics[scale=.25,trim=35 360 0 92]{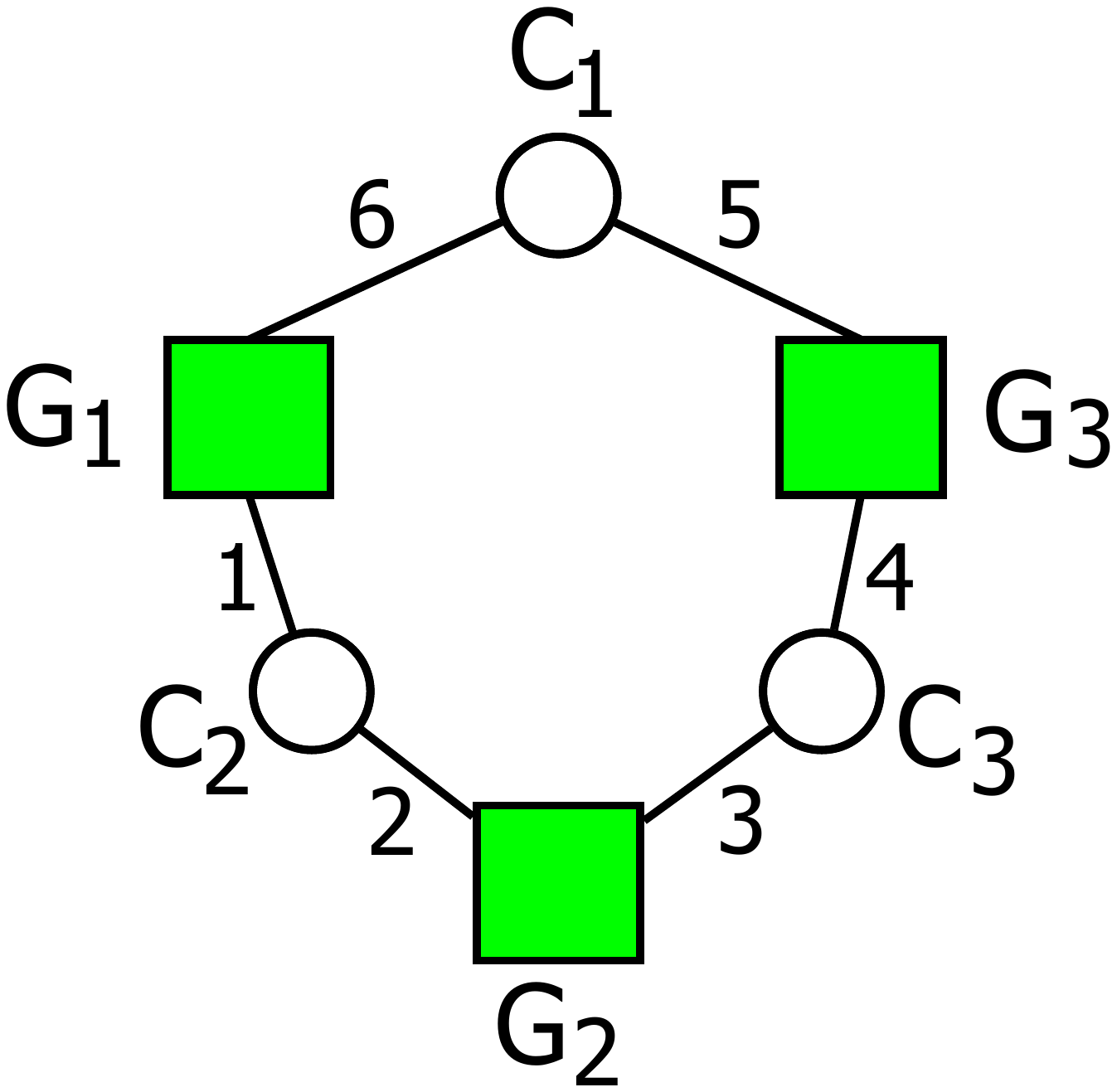}
\end{center}
\end{minipage}
\begin{minipage}{.75\linewidth}
\begin{align*}
, \quad \Gamma &= \begin{cases}
G = \begin{cases}
G_{1}= |1_10_6\rangle + |0_11_6\rangle + |1_11_6\rangle~, &\text{clause $(x_1 \vee x_2)$}\\
G_{2}= |1_20_3\rangle + |0_21_3\rangle + |1_21_3\rangle~, &\text{clause $(x_2 \vee x_3)$}\\
G_{3}= |1_40_5\rangle + |0_41_5\rangle + |1_41_5\rangle~, &\text{clause $(x_1 \vee x_3)$}
\end{cases}\\
C = \begin{cases}
{C}_{1}= \langle 0_50_6| + \langle 1_51_6|~,&\hspace{38pt}\text{variable $x_1$}\\
{C}_{2}= \langle 0_10_2| + \langle 1_11_2|~,&\hspace{38pt}\text{variable $x_2$}\\
{C}_{3}= \langle 0_30_4| + \langle 1_31_4|~,&\hspace{38pt}\text{variable $x_3$}
\end{cases}\\
E = \{1,\ldots, 6\}~.
\end{cases}
\end{align*}
\end{minipage}\\
\begin{minipage}{.25\linewidth}
\begin{center}
\includegraphics[scale=.25,trim=30 370 0 93]{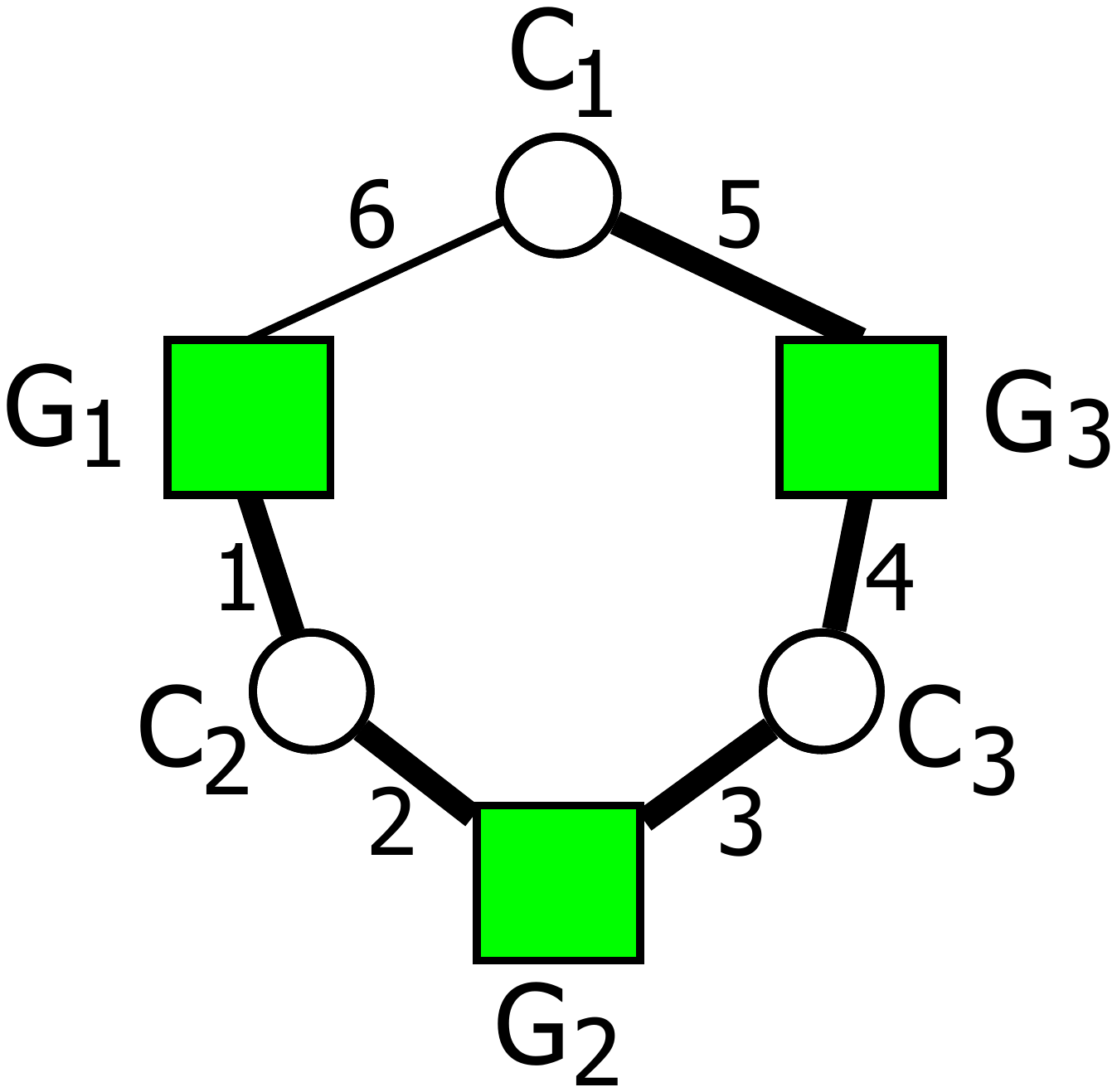}
\end{center}
\end{minipage}
\begin{minipage}{.35\linewidth}
\begin{align*}
, \quad \quad \quad \underbrace{G_2\Big(C_2\big(G_1\big),C_3\big(G_3(C_1)\big)\Big)}_{\text{Cayley/Dyck/Newick format}} &\Longrightarrow
\end{align*}
\begin{align*}
\quad \quad \quad \quad \quad \quad G_2 \otimes C_2 \otimes G_1 \otimes C_3 \otimes G_3 \otimes C_1 & = val(\Gamma)~.
\end{align*}
\end{minipage}
\caption{The tensor contraction network $\Gamma$ associated with boolean formula Eq. \ref{eq_bool_form} (top). A minimum spanning tree $T$ rooted at $G_2$, and the corresponding Newick format (bottom).} \label{fig_tcn_ex2}
\end{figure}

We claim that the $val(\Gamma) = 4$, which is the number of satisfying solutions to the Boolean formula in Eq. \ref{eq_bool_form}. We first examine the cogates $C_1,C_2,C_3$. Observe that the bra $\langle 00| + \langle 11|$ forces the state of the incident edges to be either all zero ($\langle 0|$) or all one ($\langle 1 |$). This is combinatorially analogous to the fact that if a variable is set to \textbf{false} in a clause, then it is \textbf{false} everywhere in the Boolean formula (and similarly for \textbf{true}). We now examine gates $G_1,G_2,G_3$. The only kets that exist in these gates are kets consistent with a standard OR operation. For example, $|10\rangle, |01 \rangle$ or $|11\rangle$. As we compute the tensor contraction of $\Gamma$, we will see that the contractions of bras and kets that are performed represent a brute-force iteration over \emph{every possible solution}, and the only contractions $\langle x | y\rangle$ that equal one (and are thus ``counted") correspond to legitimate solutions.

We first compute $val(\Gamma)$ via a standard exponential-time algorithm.  In Step 1, we arbitrarily choose a root of $\Gamma$, and compute a minimum spanning tree. In Step 2, we represent the minimum spanning tree in Cayley/Dyck/Newick format. In Step 3, we perform the contraction according to the order defined by the Newick tree representation.

In Fig. \ref{fig_tcn_ex2} (bottom), we demonstrate a minimum spanning tree of $\Gamma$ with $G_2$ as the root, and the corresponding nested Newick \cite{cite_newick} format which defines the order of the contraction. We now calculate $val(\Gamma)$.

{\small{
\begin{align*}
G_3  \otimes C_1 &=   \big( |1_{4}0_{5}\rangle + |0_{4}1_{5}\rangle + |1_{4}1_{5}\rangle \big)  \otimes \big(  \langle 0_5 0_6| + \langle 1_5 1_6| \big)
= \langle 0_6  |1_{4}\rangle + \langle 1_6  |0_{4}\rangle +  \langle 1_6 |1_{4}\rangle~.\\
{C}_{3} \otimes G_3  \otimes C_1 &= \big( \langle 0_3 0_4| + \langle 1_3 1_4| \big) \otimes  \big(   \langle 0_6  |1_{4}\rangle + \langle 1_6  |0_{4}\rangle +  \langle 1_6 |1_{4}\rangle  \big)~
=  \langle 0_3 1_6|  + \langle 1_3 0_6| + \langle 1_31_6|~.
\end{align*}
\vspace{-20pt}
\begin{align*}
G_2 \otimes C_2 \otimes G_1 \otimes C_3 \otimes G_3 \otimes C_1 &= \mathbf{4}~.
\end{align*}
}}\normalsize

By inspecting the intermediary steps, we see that the individual bra-ket contractions performed throughout the calculation are essentially a brute-force iteration over all the possible true-false assignments to the variables in the original Boolean formula. Thus, it is clear that processing the contraction in this way takes exponential time.

We will now present an alternative, polynomial-time computable method for calculating $val(\Gamma)$. First we recall the direct sum (denoted $\oplus$) of two \emph{labeled} matrices. In order to describe the labeling method of the rows and columns, recall that in a tensor contraction network, every gate or cogate is an element of a tensor product space as defined by the incident edges. For example, in Fig. \ref{fig_tcn_ex2}, observe that cogate $C_3$ is incident on edges 3 and 4, and $C_3 = \langle 0_30_4| + \langle 1_31_4| \in V^{*}_3 \otimes V^{*}_4$. Since cogate $C_3$ is Pfaffian (see Ex. \ref{ex_pfaff_gcg}), there exists a matrix $\Theta$ and a scalar $\beta$ such that $\beta\hspace{1pt}\text{sPf}^{*}(\Theta) =  \langle 0_30_4| + \langle 0_30_4|$. In this case, $\Theta$ is a $2 \times 2$ matrix (since $C_3$ is a 2-arity cogate), and the rows and columns of $\Theta$ are labeled with edges $3$ and $4$:
\vspace{-12pt}
{\small {
\[
\bbordermatrix{
   & 3& 4\cr
3 & 0 &  \frac{\sqrt{5}}{2} + \frac{3}{2}\cr
4 & -\big(  \frac{\sqrt{5}}{2} + \frac{3}{2} \big) & 0\cr
}~.\]}}
Now consider a tensor contraction network $\Gamma = \{G, C, E\}$ with all gates/cogates Pfaffian, and let $\{\Xi_1,\ldots, \Xi_{|G|}\}$ be the set of matrices associated with the Pfaffian gates $G = \{G_1, \ldots, G_{|G|}\}$  (note that no cogates are included). Since $\Gamma$ is a bipartite graph, the row and column labels associated with any two $\Xi, \Xi' \in \{\Xi_1,\ldots, \Xi_{|G|}\}$ are disjoint. Let $I$ be the set of row/column labels of $\Xi$ and $J$ be the set of row/column labels of $\Xi'$, and let $\sigma$ be an order on the edges of $\Gamma$. Then the direct sum $\Xi \oplus_{\sigma} \Xi'$ is the matrix $\Xi''$ with row/column label set $I \cup J$ \emph{ordered according to} $\sigma$ where $\xi''_{k\ell}=0$ if $k \in I$ and $\ell \in J$ or vice versa,  $\xi''_{k\ell}=\xi_{k\ell}$ if $k,\ell \in I$, and  $\xi''_{k\ell}=\xi'_{k\ell}$ if $k,\ell \in J$. For example,
{\scriptsize
\[
\underbrace{\bordermatrix{
&4& 5 &6\cr
4&\xi_{44}&\xi_{45}&\xi_{46}\cr
5&\xi_{54}&\xi_{55}&\xi_{56}\cr
6&\xi_{64}&\xi_{65}&\xi_{66}\cr
}}_{\Xi} \quad \oplus_{\sigma} \quad
\underbrace{\bordermatrix{
&1& 3 & 8 & 9\cr
1&\xi'_{11}&\xi'_{13}&\xi'_{18}&\xi'_{19}\cr
3&\xi'_{31}&\xi'_{33}&\xi'_{38}&\xi'_{39}\cr
8&\xi'_{81}&\xi'_{83}&\xi'_{88}&\xi'_{89}\cr
9&\xi'_{91}&\xi'_{93}&\xi'_{98}&\xi'_{99}
}}_{\Xi'} \quad = \quad 
\underbrace{\bordermatrix{
&1&3&4&5&6&8&9\cr
1&\xi'_{11} & \xi'_{13} & 0 & 0 & 0 & \xi'_{18} & \xi'_{19}\cr
3&\xi'_{31} & \xi'_{33} & 0 & 0 & 0 & \xi'_{38} & \xi'_{39}\cr
4& 0 & 0 & \xi_{44}&\xi_{45}&\xi_{46} & 0 & 0\cr
5& 0 & 0 & \xi_{54}&\xi_{55}&\xi_{56} & 0 & 0\cr
6& 0 & 0 & \xi_{64}&\xi_{65}&\xi_{66} & 0 & 0\cr
8&\xi'_{81} & \xi'_{83} & 0 & 0 & 0 & \xi'_{88} & \xi'_{89}\cr
9&\xi'_{91} & \xi'_{93} & 0 & 0 & 0 & \xi'_{98} & \xi'_{99}
}}_{\Xi \oplus_{\sigma} \Xi' = \Xi''}
\]} \normalsize
The Pfaffian circuit evaluation theorem relies on the \emph{planar spanning tree edge} order. 

\vspace{-10pt}
\begin{figure}[!h]
\begin{minipage}{.23\linewidth}
\begin{center}
\includegraphics[scale=.21,clip=true,trim=0 369 50 90]{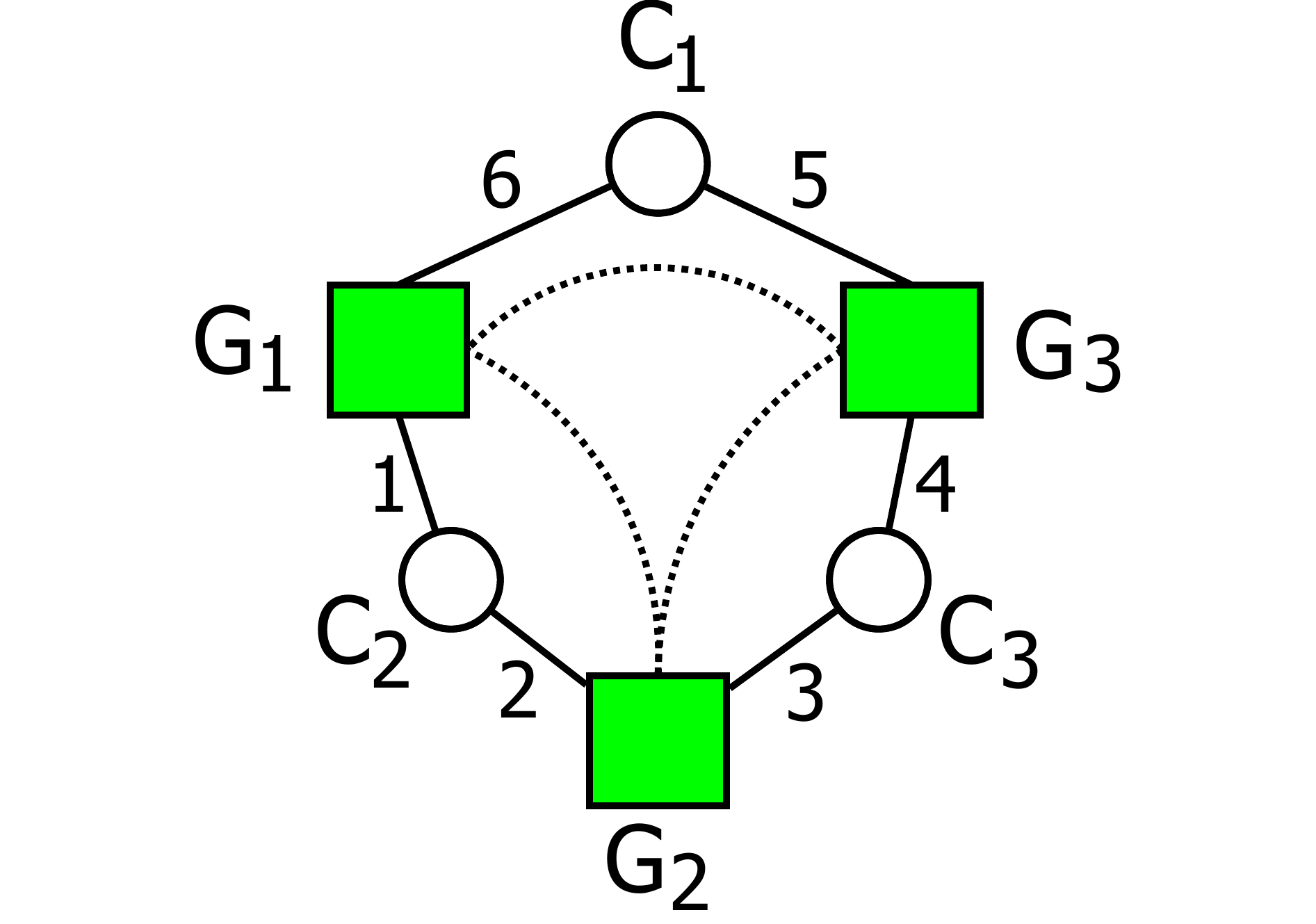}
\end{center}
\vspace{-1pt}
\hspace{51pt}(a)
\end{minipage}
\begin{minipage}{.23\linewidth}
\begin{center}
\includegraphics[scale=.21,clip=true,trim=0 369 45 90]{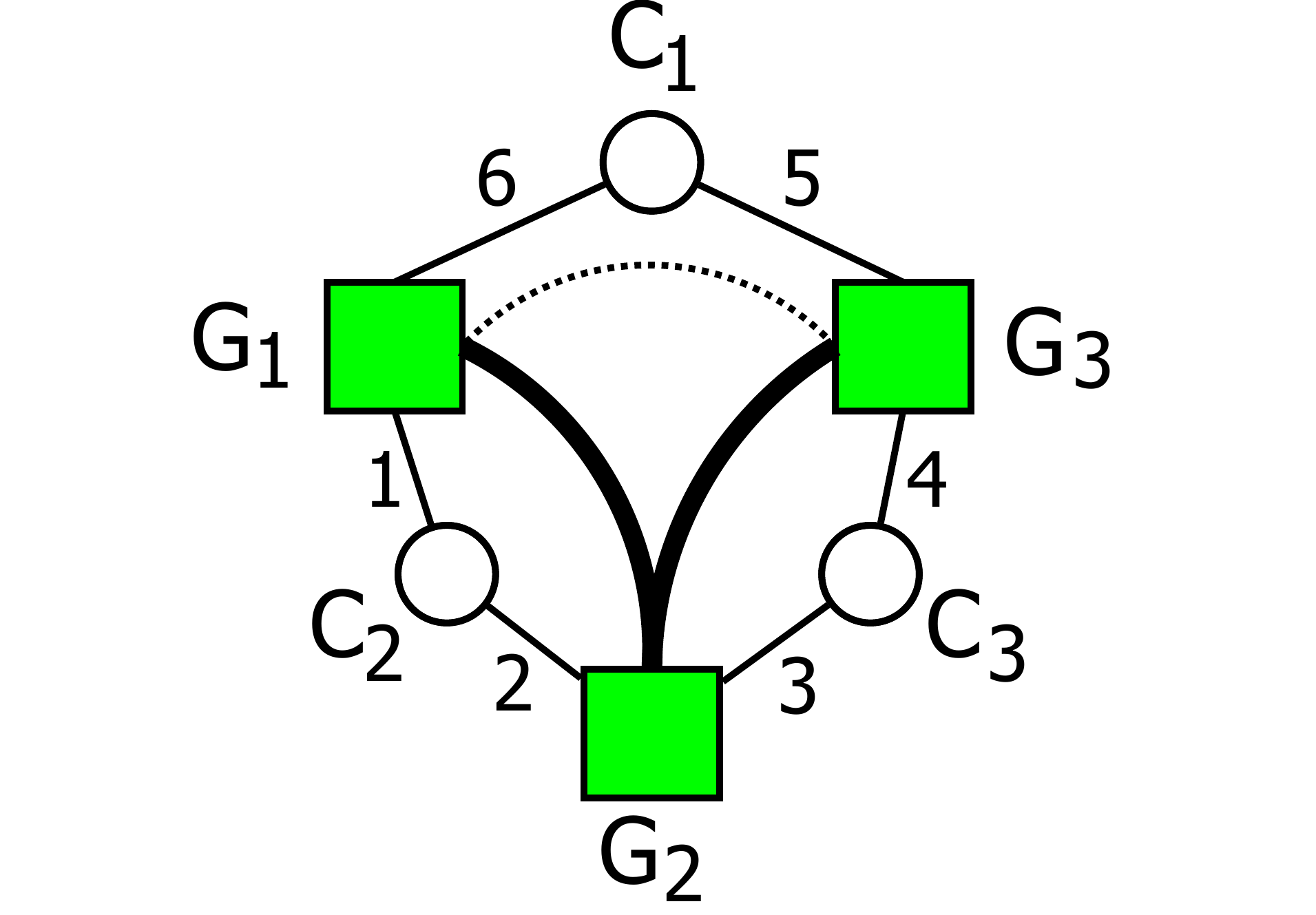}
\end{center}
\vspace{-1pt}
\hspace{51pt}(b)
\end{minipage}
\begin{minipage}{.23\linewidth}
\vspace{5pt}
\begin{center}
\includegraphics[scale=.21,clip=true,trim=0 362 50 90]{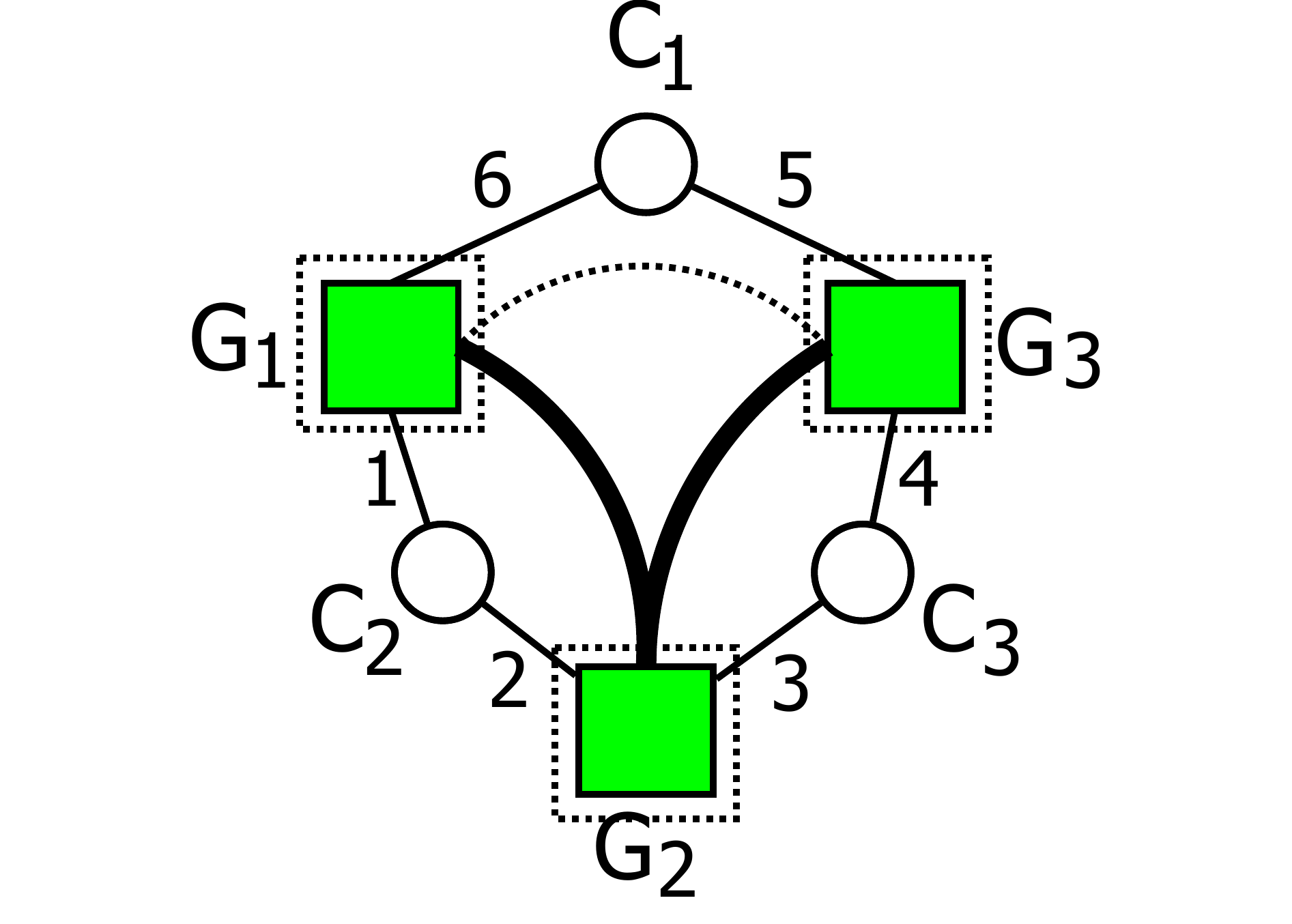}
\end{center}
\vspace{-1pt}
\hspace{51pt}(c)
\end{minipage}
\begin{minipage}{.23\linewidth}
\vspace{5pt}
\begin{center}
\includegraphics[scale=.21,clip=true,trim=0 362 45 90]{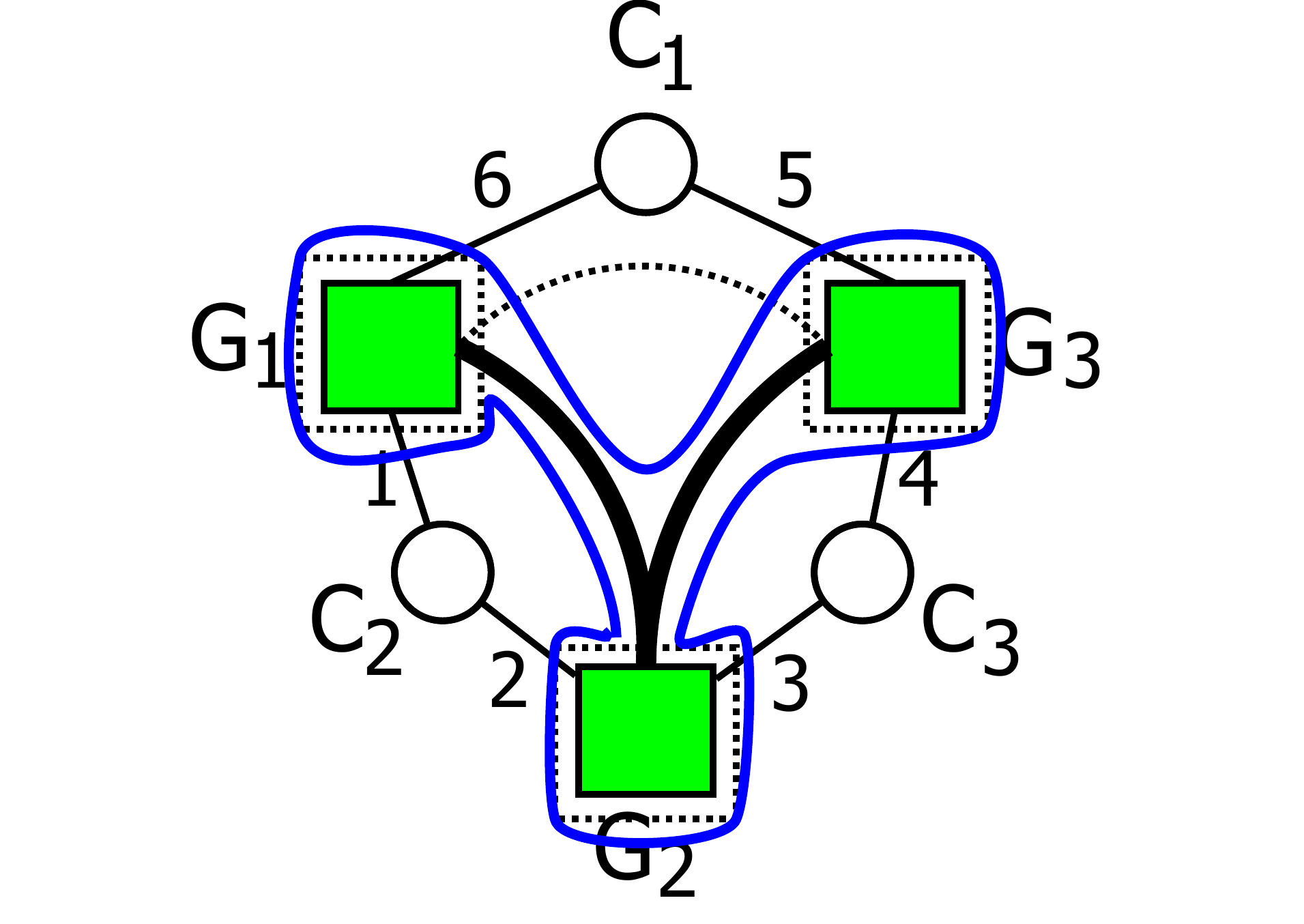}
\end{center}
\vspace{-1pt}
\hspace{51pt}(d)
\end{minipage}
\caption{Finding the planar spanning tree edge order.} \label{fig_pst_eord}
\end{figure}
\vspace{-6pt}
In order to the find the \emph{planar spanning tree edge order}, we follow the four steps demonstrated in Fig. \ref{fig_pst_eord}. Given a tensor contraction network $\Gamma$, for each interior face, we connect the \emph{gates} incident on that interior face in a cycle (Fig. \ref{fig_pst_eord}.a). Next, we take a spanning tree of the edges in those cycles (Fig. \ref{fig_pst_eord}.b), and draw ``boxes" around each of the gates (Fig. \ref{fig_pst_eord}.c). Finally, we trace the spanning tree from gate to gate, following the ``boxes" around the gates (Fig. \ref{fig_pst_eord}.d). This creates a closed curve that crosses every edge exactly once. Then, we trace the curve, beginning at any point, and tracing in any direction, ordering the edges according to when they are crossed by the curve. For example, in Fig. \ref{fig_pst_eord}.d, if we begin at $G_1$ where the curve crosses between gate $G_1$ and cogate $C_2$, and trace in counter-clockwise order. This yields the very convenient edge order $\{1,2,3,4,5,6\}$. 
The planar spanning tree edge order is discussed in detail in \cite{morton_pfaff}.

The following lemma is the first part of the Pfaffian circuit evaluation theorem.

\begin{lemma} [\cite{morton_pfaff}] \label{lem_g_tp_ds} Let $\Gamma$ be a \emph{planar} tensor contraction network of Pfaffian gates and cogates and let $\sigma$ be a planar spanning tree edge order. Let $G_1,\ldots, G_{|G|}$ be the Pfaffian gates with $G_i = \alpha_i\hspace{1pt} \text{sPf}(\Xi_i)$ for $1 \leq i \leq |G|$.  Then 
\begin{align}
G_{1} \otimes \cdots \otimes G_{|G|} = \big(\alpha_{1}\cdots \alpha_{|G|}\big)\hspace{2pt}\text{\emph{sPf}} \big( \Xi_{1} \oplus_{\sigma} \cdots \oplus_{\sigma} \Xi_{|G|} \big)~. \label{eq_tp_sp_exp}
\end{align}
The analogous result holds for cogates.
\end{lemma}

Observe that both sides of Eq. \ref{eq_tp_sp_exp} require exponential time to evaluate. However, the direct sum $\Xi_{1} \oplus_{\sigma} \cdots \oplus_{\sigma} \Xi_{|G|}$ can be calculated in polynomial-time. We now combine the gate and cogate versions of Lemma \ref{lem_g_tp_ds} to state the Pfaffian circuit evaluation theorem.
\begin{theorem} [Pfaffian circuit evaluation theorem, \cite{morton_pfaff}] \label{thm_pfaff_eval} Given a planar, Pfaffian tensor contraction network with an \emph{even} number of edges and an edge order $\sigma$, let $ \Xi_{1}, \ldots, \Xi_{|G|}$ (and $\Theta_{1}, \ldots, \Theta_{|C|}$, respectively) be the matrices from Lemma \ref{lem_g_tp_ds}. Furthermore, let $\Xi =  \Xi_{1} \oplus_{\sigma} \cdots \oplus_{\sigma} \Xi_{|G|}$, and $\Theta =  \Theta_{1} \oplus_{\sigma} \cdots \oplus_{\sigma} \Theta_{|C|}$.  Let $\widetilde{\Theta}$ be $\Theta$ with the signs flipped such that $\widetilde{\theta}_{ij} = (-1)^{i+j+1}\theta_{ij}$. Then
\begin{align}
\underbrace{\big\langle \beta\hspace{1pt} \text{\emph{sPf}}^{*}(\widetilde{\Theta})~|~\alpha \hspace{1pt}\text{\emph{sPf}}(\Xi) \big\rangle}_{\text{exponential-time}} &= \underbrace{\alpha\beta\hspace{1pt} \text{\emph{Pf}}( \widetilde{\Theta} + \Xi)}_{\text{polynomial-time}}~. \label{eq_pfaff_eval}
\end{align}
\end{theorem}

Observe that the left-hand side of Eq. \ref{eq_pfaff_eval} requires exponential-time to compute, since it involves the contraction of every bra/ket in the subPfaffian/subPfaffian dual. However, the right-hand side has the same complexity as the determinant of an $m \times m$ matrix (where $\Gamma$ has $m$ edges), and is thus computable in polynomial-time in the size of $\Gamma$. To conclude our example, let
\begin{align*}
k &= \frac{ 5^{1/2}}{2} + \frac{3}{2}~, \quad \quad \text{and let} \quad A = \left[ \begin{array}{cc}
 \frac{1}{10}\big( -5^{3/4} + 5^{5/4} ) & 5^{-1/4}  \\
 \frac{1}{10} \big( -5^{3/4} - 5^{5/4} \big)  & 5^{-1/4}
\end{array} \right]~.
\end{align*}
Observe that, via Ex. \ref{ex_pfaff_gcg} and Fig. \ref{fig_tcn_ex2}, $\Gamma$ is Pfaffian under the homogeneous change of basis $A$. Then, via Lemma \ref{lem_g_tp_ds} and Ex. \ref{ex_pfaff_gcg}, we see representations for $G_2 \otimes G_1 \otimes G_3$ and $C_1 \otimes C_2 \otimes C_3$:

\begin{minipage}{.20\linewidth}
\begin{center}
\includegraphics[scale=.22,clip=true,trim=0 362 50 90]{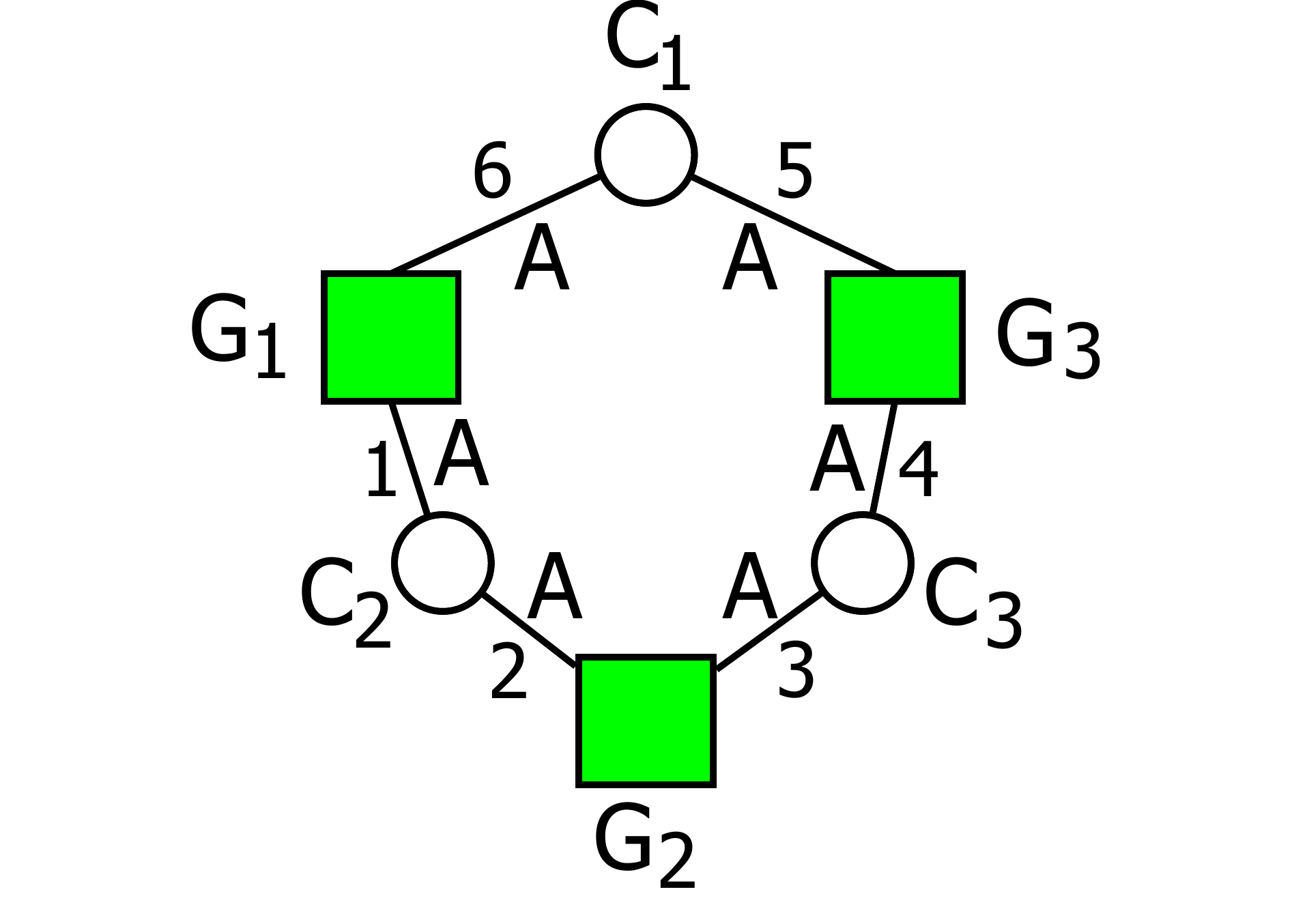}
\end{center}
\end{minipage}
\begin{minipage}{.80\linewidth}
\begin{align*}
(A \oplus A)(G_2 \otimes G_1 \otimes G_3) &= \text{sPf}(\underbrace{\Xi_2  \oplus_{\sigma} \Xi_1 \oplus_{\sigma} \Xi_3}_{\Xi})~,\\
(A^{-1} \oplus A^{-1})C_1 \otimes C_2 \otimes C_3 &= \beta^3\text{sPf}^{*}(\underbrace{\Theta_1  \oplus_{\sigma} \Theta_2 \oplus_{\sigma} \Theta_3}_{\Theta})~.
\end{align*}
\end{minipage}

Therefore,
\vspace{-10pt}
{\footnotesize{
\[
\overbrace{
\bordermatrix{\text{}&2&3\cr
                2& 0 &  -1\cr
                3& 1  & 0}}^{\Xi^G_2}\quad \oplus_{\sigma} \quad
\overbrace{
\bordermatrix{\text{}&1&6\cr
                1& 0 &  -1\cr
                6& 1  & 0}}^{\Xi^G_1}\quad \oplus_{\sigma} \quad
\overbrace{
 \bordermatrix{\text{}&4&5\cr
                4& 0 &  -1\cr
                5& 1  & 0}}^{\Xi^G_3}\quad = \quad
\bordermatrix{\text{}&1&2&3 &4&5&6\cr
                1& 0 & 0 & 0 & 0 & 0 & -1\cr
                2& 0 & 0 & -1 & 0 & 0 & 0\cr
                3& 0 & 1 & 0 & 0 & 0 & 0\cr
                4& 0 & 0 & 0 & 0 & -1 & 0\cr
		5& 0 & 0 & 0 & 1  & 0 & 0\cr
		6& 1 & 0 & 0 & 0  & 0 & 0} \quad \text{and~,}
\]

\vspace{-10pt}

\[
\overbrace{
\bordermatrix{\text{}&1&2\cr
                1& 0 &  k\cr
                2& -k   & 0}}^{\Theta_1} \quad \oplus_{\sigma} \quad
\overbrace{
\bordermatrix{\text{}&3&4\cr
                3& 0 &  k\cr
                4& - k   & 0}}^{\Theta_2} \quad \oplus_{\sigma} \quad
\overbrace{
 \bordermatrix{\text{}&5&6\cr
                5& 0 &   k \cr
                6& -k   & 0}}^{\Theta_3} \quad = \quad
\underbrace{\bordermatrix{\text{}&1&2&3&4&5&6\cr
                1& 0 & k & 0 & 0 & 0 & 0\cr
                2&  -k  & 0 & 0 & 0 & 0 & 0\cr
                3& 0 & 0 & 0 & k & 0 & 0\cr
                4& 0 & 0 &  - k  & 0 & 0 & 0\cr
		5& 0 & 0 & 0 & 0  & 0 & k\cr
		6& 0 & 0 & 0 & 0  &  -k  & 0}}_{\Theta} \quad \phantom{\text{and~}}.
\]
}} \normalsize
Finally, we see

\vspace{-10pt}
\begin{minipage}{.25\linewidth}
{\footnotesize{
\begin{align*}
\widetilde{\Theta} + \Xi &= 
\bordermatrix{\text{}&1&2&3&4&5&6\cr
                1& 0 & k & 0 & 0 & 0 & -1\cr
                2&  -k  & 0 & -1 & 0 & 0 & 0\cr
                3& 0 & 1 & 0 & k & 0 & 0\cr
                4& 0 & 0 &  -k  & 0 & -1 & 0\cr
		5& 0 & 0 & 0 & 1  & 0 &  k\cr
		6& 1 & 0 & 0 & 0  &  -k & 0} \Longrightarrow 
\end{align*}}} \normalsize
\end{minipage}
\hspace{-55pt}
\begin{minipage}{.75\linewidth}
\begin{align*}
 \text{Pf}\big( \widetilde{\Theta} + \Xi \big) &= 8 + 4 \sqrt{5}~,\\
\beta^3  \text{Pf}\big( \widetilde{\Theta} + \Xi \big) &= \bigg( \frac{\sqrt{5}}{2} - \frac{1}{2} \bigg)^3 \big(  8 + 4 \sqrt{5} \big)~,\\
&= \mathbf{4}~.
\end{align*}
\end{minipage}

Thus, $val(\Gamma)$, as calculated by the Pfaffian circuit evaluation theorem, is equal to the number of satisfying solutions to the underlying combinatorial problem. We conclude by mentioning that if the number of edges in $\Gamma$ is odd, then $\text{Pf}(\widetilde{\Theta} + \Xi) = 0$, and there is no relation between the Pfaffian and $val(\Gamma)$. This explains why the Pfaffian circuit evaluation theorem only applies to planar, Pfaffian circuits with an \emph{even} number of edges (see \cite{morton_pfaff} for more detail).

\section{Constructing Pfaffian Circuits via Algebraic Methods} \label{sec_alg_meth} A gate is \emph{Pfaffian} (after a heterogeneous change of basis) if there exists an $n \times n$ skew-symmetric matrix $\Xi$, an $\alpha \in \mathbb{C}$, and matrices $A_1,\ldots,A_n \in \mathbb{C}^{2 \times 2}$ such that $(A_1 \otimes \cdots \otimes A_n) G = \alpha\hspace{1pt}\text{sPf}(\Xi)$ (Def. \ref{def_pfaff_gcg}). In this section, we present an algebraic model of this property. In other words, given a gate $G$ (or cogate $C$), we present a system of polynomial equations such that the changes of bases $A_1,\ldots, A_n$ under which $G$ is Pfaffian are in bijection to the solutions of this system. If there are \emph{no} solutions, then there do not exist \emph{any} matrices $A_1,\ldots, A_n$ such that $(A_1 \otimes \cdots \otimes A_n) G = \alpha\hspace{1pt}\text{sPf}(\Xi)$, and $G$ is \emph{not} Pfaffian. We comment 
that there is a wealth of literature on pursuing combinatorial problems via systems of polynomial equations (see \cite{alonsurvey, alontarsi, susan_cpc, lovasz1, onn} and references therein).

Throughout the rest of this paper, in order to prove that a given gate/cogate is Pfaffian, we simply demonstrate the $n \times n$ skew-symmetric matrix $\Xi$, the scalar $\alpha$, and the change of basis matrices $A_1,\ldots,A_n$ such that $G$ is Pfaffian, without further explanation. But we observe in advance that \emph{all} such ``Pfaffian certificates" (i.e., those displayed in Ex. \ref{ex_pfaff_gcg}) are found via solving the system of equations presented below. When we state that a given gate/cogate is \emph{not} Pfaffian, we specifically mean that the Gr\"obner basis of the system of polynomial equations described below (found via computation with \textsc{Singular}) consists of the single polynomial 1. Thus, the variety associated with the ideal is empty, and there are \emph{no} solutions. 

Before presenting the system of polynomial equations, we introduce the following notation. Let $[n]$ denote the set of integers $\{1,\ldots,n\}$. Given an integer $i \in [n]$ and a set $I \subseteq [n]$, the notation $I_i \in \{0,1\}$ denotes the $i$-th bit of the bitstring representation of $I$. For example, given $I = \{1,3\} \subseteq [4]$, then $|I\rangle$ is equivalent to $|1010\rangle$, with $I_1 = 1, I_2 = 0, I_3 = 1$ and  $I_4 = 0$. 
Given an $n$-arity gate $G = \sum_{I \subseteq [n]}G_I |I\rangle$, let $G' = (A_1 \otimes \cdots \otimes A_n)G$, and then the following formula defines the coefficient $G'_{I'}$ where $I' \subseteq [n]$:
\vspace{-5pt}
\begin{align}
G'_{I'} &= \sum_{I \subseteq [n]}G_{I}\prod_{i = 1}^nA_i[I'_i, I_i]~. \label{eq_g_cob}
\end{align}
\vspace{-10pt}
\begin{example} \label{ex_gcob_hom} Let $A$ be the matrix below and consider a generic 3-arity gate $G$:

\vspace{-5pt}
\begin{minipage}{.20\linewidth}
\[
A = \left[ \begin{array}{cc}
a_{00} & a_{01}\\[-1.0ex]
a_{10} & a_{11}
\end{array} 
\right]~, \quad \quad
\]
\end{minipage}
\begin{minipage}{.70\linewidth}
\begin{align*}
G &= G_{000}|000\rangle + G_{010}|010\rangle + G_{001}|001\rangle + G_{011}|011\rangle + G_{100}|100\rangle\\
&\phantom{=~} +G_{110}|110\rangle+ G_{101}|101\rangle +G_{111}|111\rangle~.
\end{align*}
\end{minipage}

\noindent Then, given $G' = A^{\otimes 3}G$, the coefficient of the ket $|100\rangle$ in $G'$, denoted $G'_{100}$, is
\vspace{-5pt}
\begin{align*}
G'_{100} &= G_{000}a_{10}a_{00}^2+G_{100}a_{11}a_{00}^2+G_{010}a_{10}a_{01}a_{00}+G_{001}a_{10}a_{00}a_{01} +G_{110}a_{11}a_{01}a_{00}\\
&\phantom{=}+G_{101}a_{11}a_{00}a_{01}+G_{011}a_{10}a_{01}^2+G_{111}a_{11}a_{01}^2~.
\end{align*} 

\vspace{-20pt}
\hfill $\Box$
\end{example}

Finally, given $I \subseteq [n]$ with $|I|$ even and $|I| \geq 4$, let $I_{\min}$ be the smallest integer in $I$.
\begin{theorem} \label{thm_pfaffian_gates_cob} Let $G  = \sum_{I \subseteq [n]}G_{I} |I\rangle$ be an $n$-arity gate,  and let $G' = (A_1 \otimes \cdots \otimes A_n)G$ with $G'_{I'}$ defined by Eq. \ref{eq_g_cob}. Then, the solutions to the following system of polynomial equations are in bijection to the changes of bases $A_1,\ldots,A_n$ such that $G$ is Pfaffian.
\begin{align}
\big(G'_{\emptyset}\big)^{-1}{G'}_{I'} &= 0, \hspace{148pt}\text{~$\forall$ $I' \subseteq [n]$ with $|I'|$ odd}~, \label{eq_parity}
\end{align}

\vspace{-20pt}
\begin{align}
\big(G'_{\emptyset}\big)^{-1}{G'}_{I'} &= \sum_{i \in {I'} \setminus {I'}_{\min}} (-1)^{{I'}_{\min}+ i + 1} \Big(\big(G'_{\emptyset}\big)^{-1} \Big)^2G'_{{I'}_{\min}, i}{G'}_{{I'} \setminus \{ {I'}_{\min}, i \}}~, \nonumber\\
& \hspace{105pt}\text{~$\forall$ $I' \subseteq [n]$ with $|I'|$ even and $|I'| \geq 4$}~, \label{eq_consist}
\end{align}

\vspace{-20pt}
\begin{align}
\sum_{I \subseteq [n]}G_{I}\prod_{i = 1}^nA_i[0, I_i] - G'_{\emptyset} &=0~, \label{eq_em}\\ G'_{\emptyset}\big(G'_{\emptyset}\big)^{-1} - 1 = 0~, \label{eq_eminv}\\
\det(A_i)\big(\det(A_i)\big)^{-1} - 1 &=0, \text{~for~$1 \leq i \leq n$}~, \label{eq_ainv}
\end{align}
\end{theorem}

\vspace{-3pt}
We refer to Eqs.  \ref{eq_eminv} and \ref{eq_ainv} as the ``inversion" equations, since Eq. \ref{eq_ainv} insures that each of the matrices $A_1,\ldots,A_n$ are invertible, and Eq. \ref{eq_eminv} insures that the coefficient for $G'_{\emptyset}$ can be scaled to equal one (since the Pfaffian of the empty matrix is equal to one by definition). This scaling factor appears again in Eqs. \ref{eq_parity} and \ref{eq_consist}, taking into account that each coefficient $G'_{I'}$ is scaled by $\alpha$ since $(A_1\otimes \cdots \otimes A_n)G = \alpha \hspace{1pt}\text{sPf}(\Xi)$. Eq. \ref{eq_em} is simply the equation for $G'_{\emptyset}$ according to the change of basis formula given in Eq. \ref{eq_g_cob}.
We refer to Eqs. \ref{eq_parity} and \ref{eq_consist} as the ``parity" and ``consistency" equations, respectively. The parity equations capture the condition that the Pfaffian of an $n \times n$ matrix with $n$ odd is zero. To understand the ``consistency" equations, recall that the Pfaffian of a given $n \times n$ matrix with $n$ even can be recursively written in terms the Pfaffians of each of the $(n - 1) \times (n-1)$ submatrices. For example, the Pfaffian of $6 \times 6$ matrix $A$ can be expressed as $a_{12}\text{Pf}(A|_{3456}) - a_{13}\text{Pf}(A|_{2456}) + a_{14}\text{Pf}(A|_{2356}) - a_{15}\text{Pf}(A|_{2346}) + a_{16}\text{Pf}(A|_{2345})$. In this case, observe that $I' = \{1,2,3,4,5,6\}$, and $I'_{\min} = 1$. 
Though Theorem \ref{thm_pfaffian_gates_cob} refers gates,  there is an analogous theorem for cogates. 
\begin{example} \label{ex_g0111_gb} Let $G = \text{OR} = |10\rangle + |01\rangle + |11\rangle$, and let $G' = (A \otimes B)G$. Then, the system of equations associated with Theorem \ref{thm_pfaffian_gates_cob} is as follows:
\begin{align*}
\big({G'}_{\emptyset}\big)^{-1}(a_{11}b_{00}+a_{10}b_{01}+a_{11}b_{01})=
\big({G'}_{\emptyset}\big)^{-1}(a_{01}b_{10}+a_{00}b_{11}+a_{01}b_{11})=0~,\\
a_{01}b_{00}+a_{00}b_{01}+a_{01}b_{01}-{G'}_{\emptyset} =
a_{00}a_{11}-a_{01}a_{10}-\det(A) =
b_{00}b_{11}-b_{01}b_{10}-\det(B)=0~,\\
{G'}_{\emptyset}\big({G'}_{\emptyset}\big)^{-1}-1=0~, \quad
\det(A)\big(\det(A)\big)^{-1}-1=0~, \quad
\det(B)\big(\det(B)\big)^{-1}-1=0~.
\end{align*}
Note that the change of basis in Ex. \ref{ex_pfaff_gcg} is indeed a solution to this system of  equations.
\end{example}

The next two sections detail the results of computing with these algebraic models.

\section{Boolean Variables and Tensor Contraction Networks} \label{sec_bool_ten} In the field of combinatorial optimization, there are an almost countless number of problems (independent set, dominating set, bin packing, partition, satisfiability, etc.) that can be modeled as a series of equations in which the variables are restricted to 0/1 values (also called \emph{binary} or \emph{Boolean variables}).  In a generalized counting constraint satisfaction setting, we cannot assume that we are allowed to swap variables or copy them (indeed the latter is impossible in the quantum setting).  Thus to demonstrate the applicability of Pfaffian circuits to 0/1 \#CSP  problems, we must first demonstrate a planar, Pfaffian tensor representation of a variable that is either zero or one in every constraint where it appears.   Such a tensor may only be avaiable for restricted arities; hence the popularity of restrictions such as ``read-twice'' in the literature on holographic algorithms.  In this section, we explore two 
different representations of Boolean variables, the first utilizing a homogeneous change of basis, and the second possible under a heterogeneous change of basis only. We compare and contrast these two constructions, focusing in particular on categorizing the gates/cogates that {\em pair with} these representations, since these collections represent a class of 0/1 \#CSP problems solvable in polynomial time.

In order to model a 0/1 \#CSP problem as a tensor contraction network, we first recall the standard variable/constraint graph (see Fig. \ref{fig_var_con}.a), which is a bipartite graph with variables above and constraints below, and an edge between a variable and constraint if a given variable appears in a particular constraint. In order to translate this graph to a Pfaffian circuit, we simply model variables as cogates and constraints as gates (see Fig. \ref{fig_var_con}.b), or vice versa.
\begin{figure}[!h]
\vspace{-18pt} 
\hspace{50pt}\begin{minipage}{.4\linewidth}
\begin{center}
\includegraphics[scale=.22,clip=true,trim=60 502 140 55]{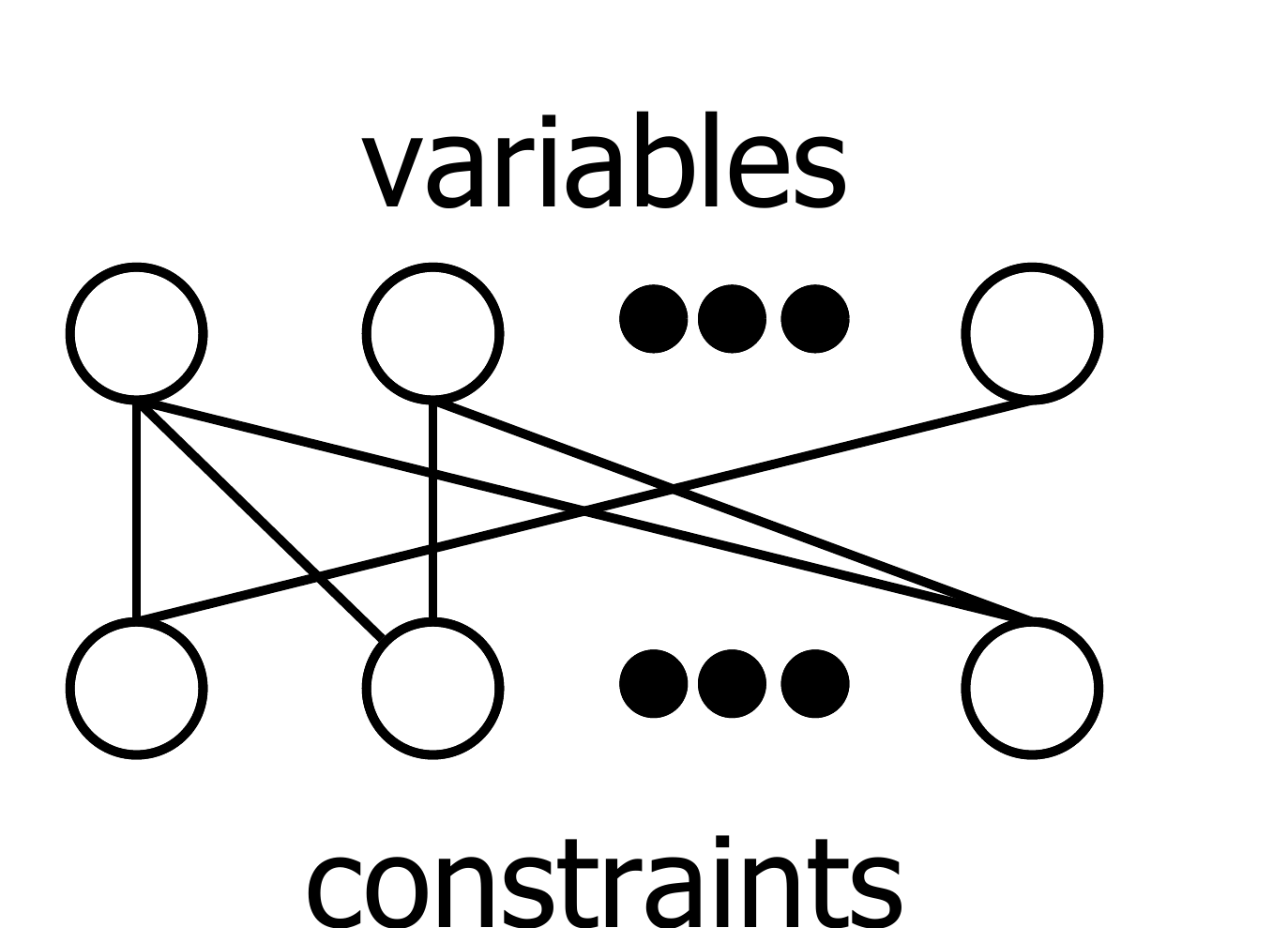}
\end{center}
\vspace{-1pt}
\hspace{70pt}(a)
\end{minipage}
\begin{minipage}{.4\linewidth}
\begin{center}
\includegraphics[scale=.22,clip=true,trim=60 502 140 55]{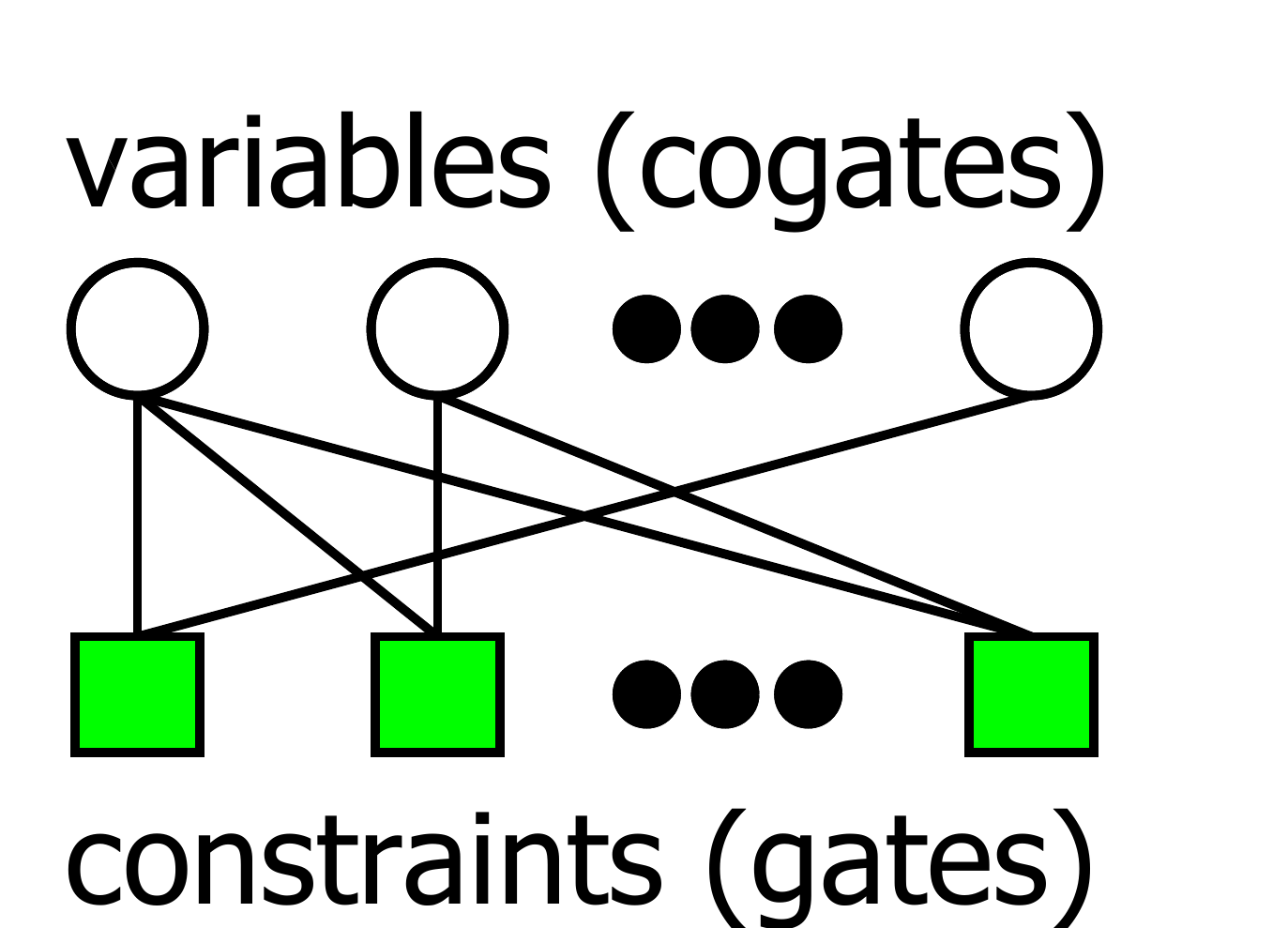}
\end{center}
\vspace{-1pt}
\hspace{65pt}(b)
\end{minipage}
\vspace{-5pt}
\caption{Variable/constraint graph as a tensor contraction network.} \label{fig_var_con}
\end{figure}

Following the graphs presented in Fig. \ref{fig_var_con}, we need a tensor that models the property that if a variable is set to zero, it is simultaneously zero in every constraint in which it appears (and similarly for one). The obvious suggestion for modeling a Boolean variable as a tensor is the following gate or cogate, commonly denoted as the \textsc{equal} gate/cogate:
\vspace{-15pt}
\begin{align*}
|00\cdots 0 \rangle + |11\cdots 1 \rangle~, \quad \quad \langle 00\cdots 0| + \langle 11\cdots 1 |~.
\end{align*}
\vspace{-15pt}
\begin{center}
\includegraphics[scale=.10,page=1,clip=true,trim=0 485 0 81]{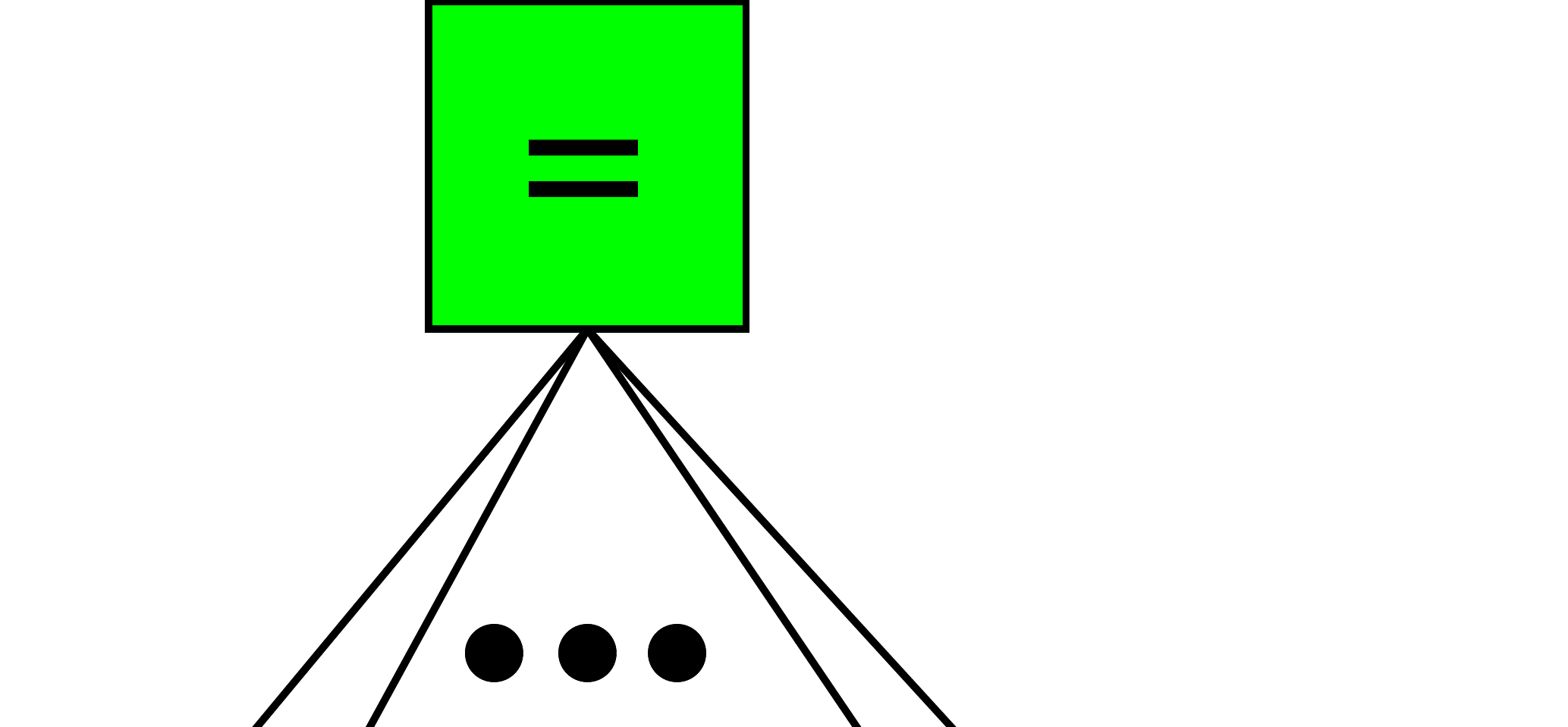}
\includegraphics[scale=.10,page=2,clip=true,trim=0 485 0 81]{bool_var.pdf}
\includegraphics[scale=.10,page=3,clip=true,trim=0 485 0 81]{bool_var.pdf}
\end{center}
Consider the gate $|0\cdots 0\rangle + |1\cdots 1 \rangle$. If any edge is set to $|0\rangle$, then every edge is set to $|0\rangle$, and similarly for $|1\rangle$ (see Ex. \ref{ex_bool_var_con}).  Observe that gate $(|0\rangle + |1\rangle)^{\otimes n}$ does not satisfy this property.
\begin{example} \label{ex_bool_var_con} Consider the following tensor contraction network fragment, containing gate $G_1$ and cogates $C_1$ and $C_2$, where edges $1,\ldots,6$ are labeled, and $7,\ldots,m$ are unlabeled.

\begin{minipage}{.30\linewidth}
\begin{center}
\includegraphics[scale=.22,page=1,clip=true,trim=0 475 0 85]{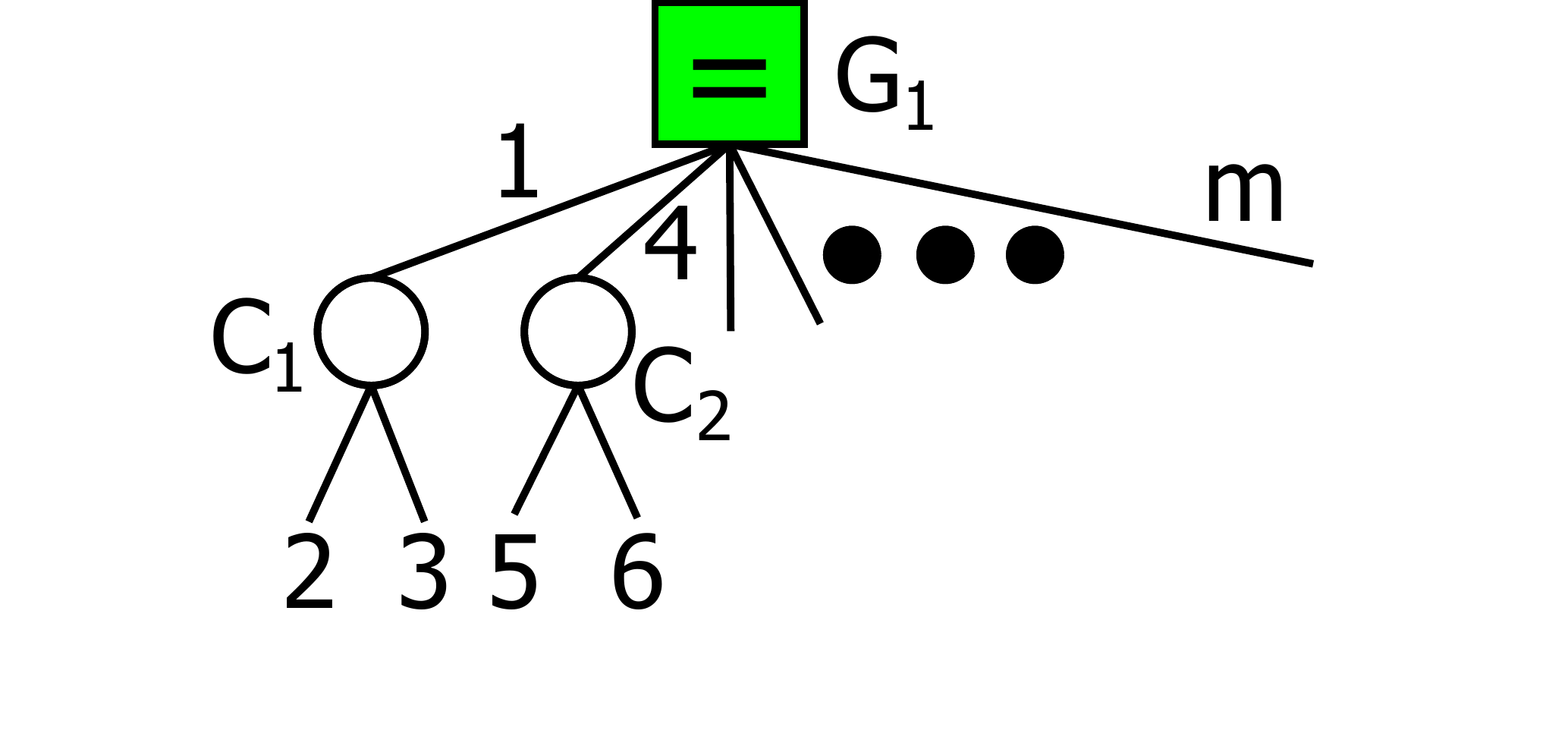}
\end{center}
\end{minipage}
\hspace{-10pt}\begin{minipage}{.70\linewidth}
\vspace{-15pt}
\begin{align*}
G_1 &= |0_10_4 \cdots 0_m \rangle + |1_11_4 \cdots 1_m\rangle~,\\
C_1 &= \langle 0_10_21_3| + \langle 0_11_20_3| + \langle 0_11_21_3|~,\\
C_2 &= \langle 0_41_51_6| + \langle 1_40_50_6|~.
\end{align*}
\end{minipage}
\vspace{-5pt}
While computing $val(\Gamma)$, the contraction $C_1 \otimes G_1$ is performed:
\begin{align*}
\Big \langle \big(\langle 0_10_21_3| + \langle 0_11_20_3| + \langle 0_11_21_3| \big) \Big | \big(|0_10_4 \cdots 0_m \rangle + |1_11_4 \cdots 1_m\rangle \big) \Big \rangle~.
\end{align*}
Since $\langle 0_1 | 1_1 \rangle =0$, the partial contraction
$\Big \langle \langle 0_10_21_3| + \langle 0_11_20_3| + \langle 0_11_21_3| \Big | |1_11_4 \cdots 1_m\rangle \Big \rangle~=~0$. Therefore, the only non-zero terms arise from $\Big \langle \big(\langle 0_10_21_3| + \langle 0_11_20_3| + \langle 0_11_21_3| \big) \Big | \big( |0_10_4 \cdots 0_m \rangle \big) \Big \rangle$, since $\langle 0_1 | 0_1\rangle = 1$. Thus, even though cogate $C_2$ contains bra $\langle 1_40_50_6|$, that combination is not counted as a solution since $\langle 1_40_50_6|0_10_4 \cdots 0_m\rangle = 0$. Thus, for \emph{any} tensor contraction network containing $G_1$ (or a similar boolean tensor), if any edge of $G_1$ is set to $|0\rangle$ (or $|1\rangle$), all the edges of $G_1$ are set to $|0\rangle$ (or $|1\rangle$, respectively). \hfill $\Box$
\end{example}

Throughout this section, we refer to $|00\cdots 0\rangle + |11\cdots 1 \rangle$ as the \emph{n-arity  \textsc{equal} gate} (or cogate, respectively).
The following is a variant of the Hadamard basis of \cite{valiant_qc}. 

\begin{proposition}  \label{prop_bool_gcg} The $n$-arity \textsc{equal} gate (or cogate, respectively) is Pfaffian under the homogeneous change of basis $(A \otimes \cdots \otimes A)$ (or $(B^{-1} \otimes \cdots \otimes B^{-1})$, respectively).

\vspace{-5pt}
\begin{minipage}{0.10\linewidth}
\[
A = \left[\begin{array}{cc}
1 & 1\\[-1ex]
1 & -1
\end{array} \right]
\]
\end{minipage}
\begin{minipage}{0.25\linewidth}
\begin{center}
\includegraphics[scale=.15,clip=true,trim=75 480 200 55]{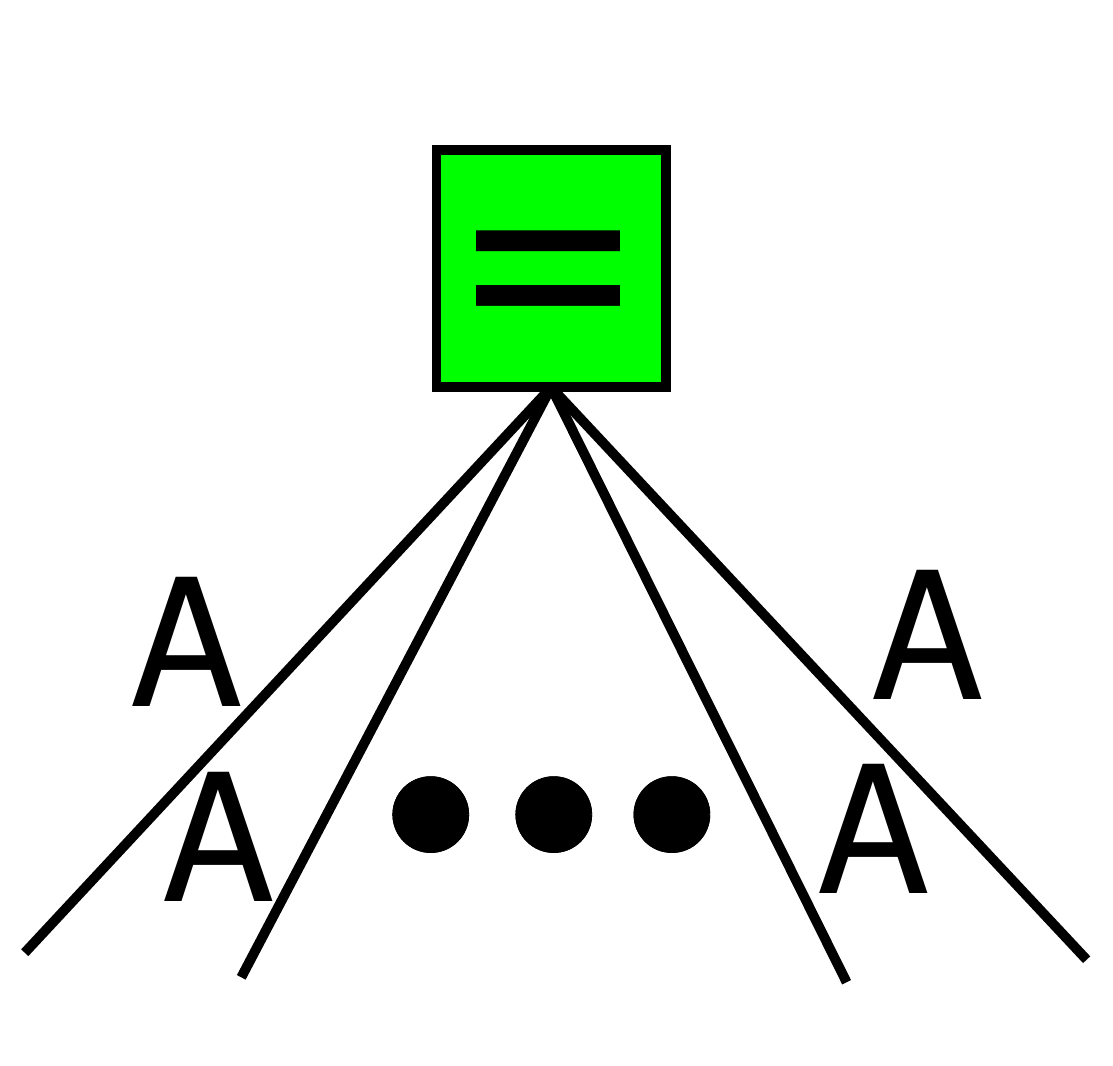}
\end{center}
\end{minipage}
\hspace{-35pt}\begin{minipage}{0.50\linewidth}
\[
,~\text{and}~B = \left[\begin{array}{cc}
1 & -1\\[-1ex]
1/2 & 1/2
\end{array} \right]~,~B^{-1} = \left[\begin{array}{cc}
1/2 & 1\\[-1ex]
-1/2 & 1
\end{array} \right]~,
\]
\end{minipage}
\begin{minipage}{0.15\linewidth}
\begin{center}
\includegraphics[scale=.15,clip=true,trim=0 480 0 75]{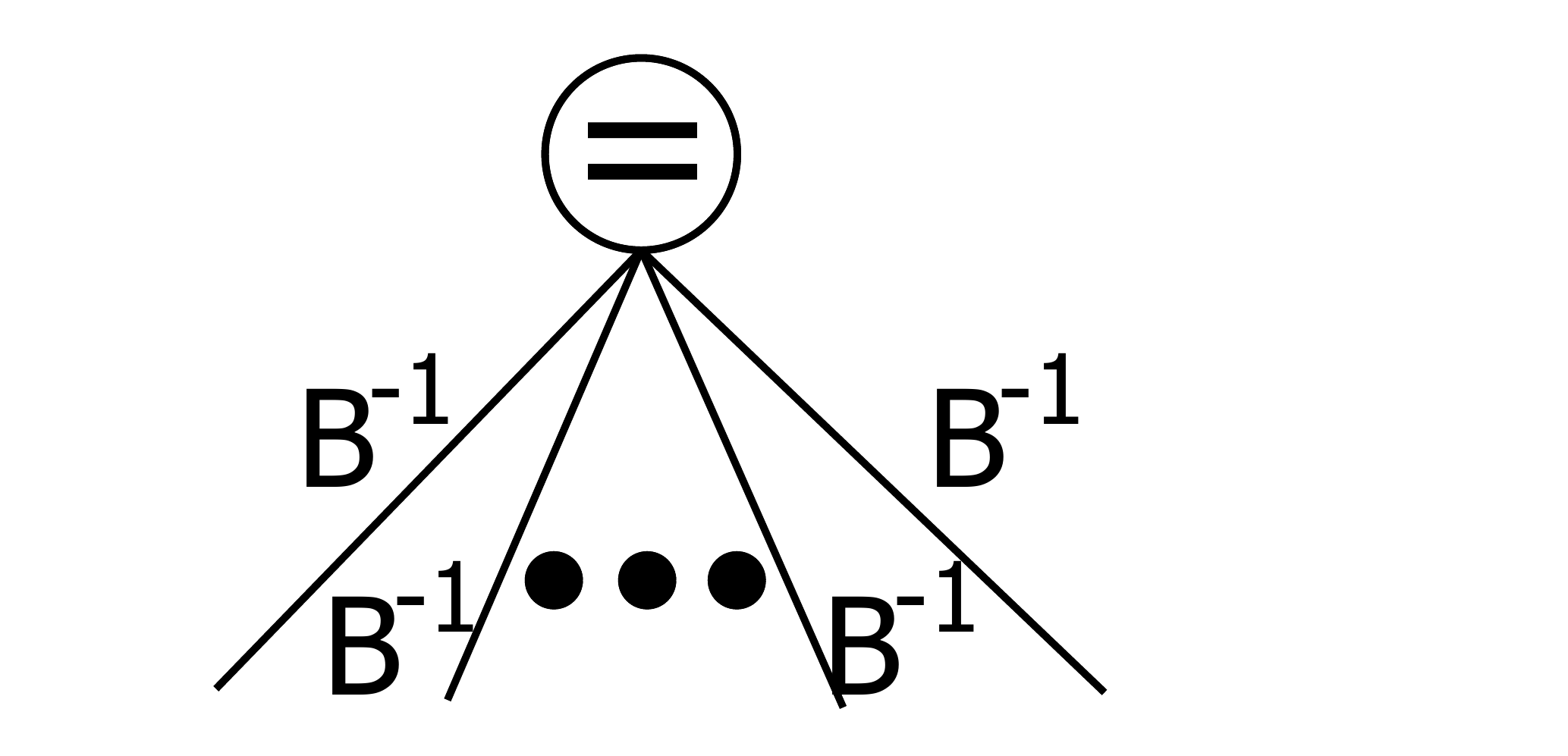}
\end{center}
\end{minipage}

\end{proposition}

%
%
\begin{proof} We provide a generalized form for the matrices $\Xi$ and $\Theta$ in the ``Pfaffian certificates" $(A\otimes \cdots \otimes A)\big(|00\cdots 0\rangle + |11 \cdots 1\rangle\big) = 2\hspace{1pt}\text{sPf}(\Xi)$ and $(B^{-1}\otimes \cdots \otimes B^{-1})\big(\langle 00 \cdots 0| + \langle 11 \cdots 1|\big) = 2\hspace{1pt}\text{sPf}^{*}(\Theta)$.
\vspace{-12pt}
\begin{align*}
(A\otimes \cdots \otimes A)\big(|00\cdots 0\rangle + |11 \cdots 1\rangle\big) &= 2\hspace{1pt}\text{sPf}\left(\left[ {\footnotesize \begin{array}{ccccc}
 0 & 1 & \cdots & \cdots & 1\\[-1.5ex]
-1 & 0 & 1 & \cdots & 1\\[-1.5ex]
\vdots & -1 &\ddots & \ddots & \vdots \\[-1.5ex]
\vdots&  & \ddots & 0 & 1\\[-1.5ex]
-1 & -1 & \cdots & -1 & 0
\end{array}}\right]\right)~, \\
(B^{-1}\otimes \cdots \otimes B^{-1})\big(\langle 00 \cdots 0| + \langle 11 \cdots 1|\big) &= 2\hspace{1pt}\text{sPf}^{*}\left(\left[ {\footnotesize\begin{array}{ccccc}
 0 & 1/4 & \cdots & \cdots & 1/4\\[-1.5ex]
-1/4 & 0 & 1/4 & \cdots & 1/4\\[-1.5ex]
\vdots & -1/4 &\ddots & \ddots & \vdots \\[-1.5ex]
\vdots&  & \ddots & 0 & 1/4\\[-1.5ex]
-1/4 & -1/4 & \cdots & -1/4 & 0
\end{array}}\right]\right)~.
\end{align*}
\end{proof}

\vspace{-8pt}
We now introduce a tensor contraction network fragment that is equivalent to the $n$-arity \textsc{equal} gate/cogate, in that if any edge in the fragment is fixed to $|0\rangle$ or $|1\rangle$, then every edge in the fragment is automatically fixed to the same state. In the following theorem, every vertex is associated with a \textsc{equal} tensor (i.e., the 2-arity cogate is $\langle 00| + \langle 11|$, and the 3-arity gate is $|000\rangle + |111\rangle$, etc.).
\begin{theorem} \label{thm_bool_tree}
$\phantom{xxx}$
\begin{enumerate}
	\item There does \textbf{not} exist a matrix $A \in \mathbb{C}^{2 \times 2}$ such that the  collection of \textsc{equal} gates/cogates in Fig. \ref{fig_thm_bt}.a are Pfaffian under the change of basis $A$.
 	\item There \textbf{do} exist matrices $A, B \in \mathbb{C}^{2 \times 2}$ such that the collection of \textsc{equal} gates/cogates in Fig. \ref{fig_thm_bt}.b are Pfaffian under the changes of bases $A$ and $B$.
\end{enumerate}
\begin{figure}[!h]
\hspace{50pt}\begin{minipage}{.4\linewidth}
\begin{center}
\includegraphics[scale=.21,clip=true,trim=0 425 0 81]{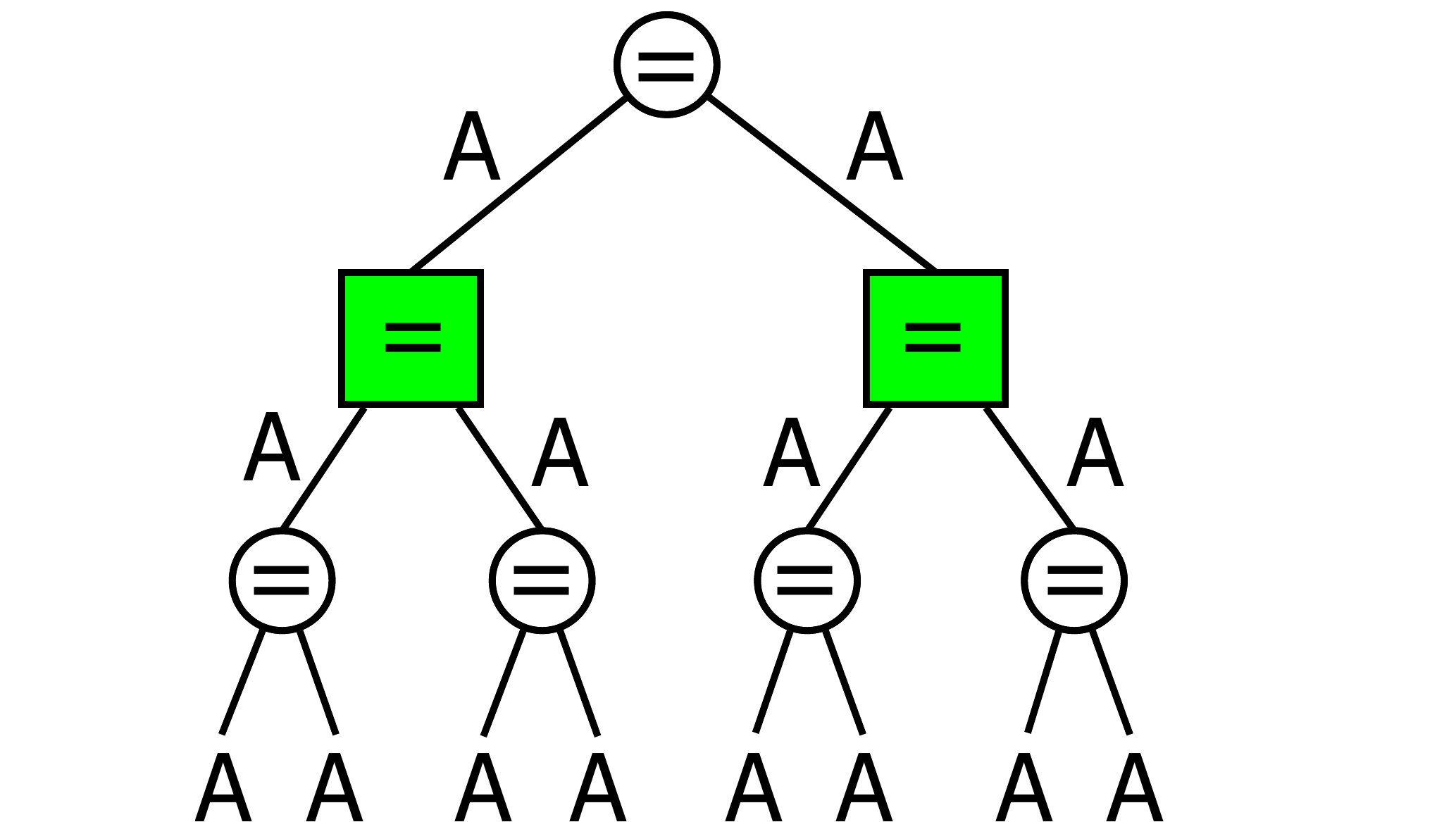}
\end{center}
\hspace{65pt}(a)
\end{minipage}
\vspace{-12pt} 
\hspace{30pt}\begin{minipage}{.4\linewidth}
\begin{center}
\includegraphics[scale=.21,clip=true,trim=0 375 0 80]{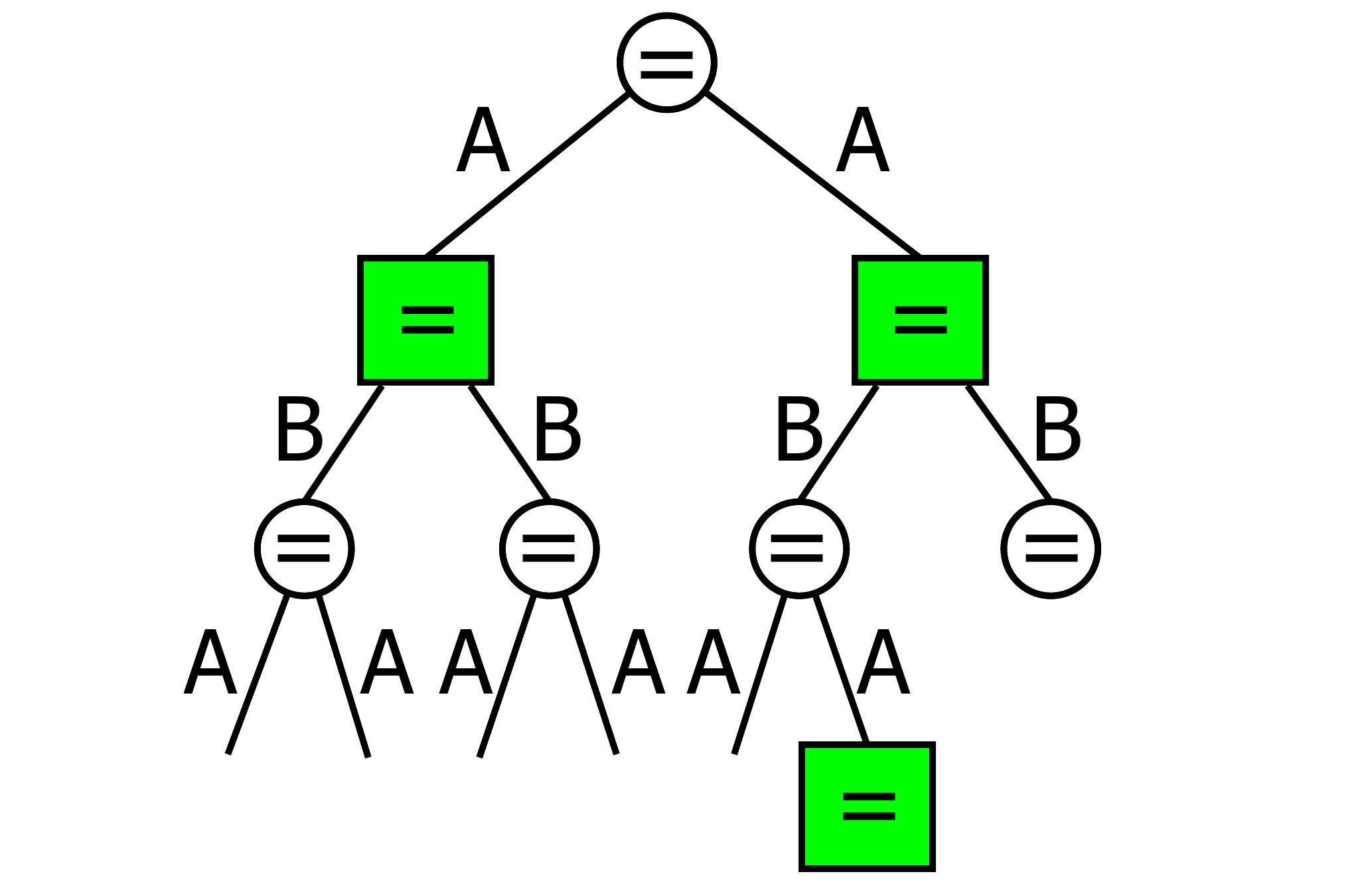}
\end{center}
\vspace{-5pt}
\hspace{70pt}(b)
\end{minipage}
\caption{Boolean trees (homogeneous and heterogeneous)} \label{fig_thm_bt}
\end{figure}
\end{theorem}
\vspace{-10pt}
\begin{proof} The proof of Theorem \ref{thm_bool_tree}.1 is by a Gr\"obner basis computation that runs in under a second. For Theorem \ref{thm_bool_tree}.2, we present the ``Pfaffian certificates" for each of the \textsc{equal} gates/cogates present in the Boolean tree. Let $A, B \in \mathbb{C}^{2 \times 2}$ be as follows:
\[
A = \left[\begin{array}{cc}
1 & 1\\[-1.5ex]
1 & -1
\end{array} \right]~, \quad
A^{-1} = -\frac{1}{2}\left[\begin{array}{cc}
-1 & -1\\[-1.5ex]
-1 & 1
\end{array} \right]~, \quad
B = \left[\begin{array}{cc}
1 & -1\\[-1.5ex]
1/2 & 1/2
\end{array} \right]~, \quad B^{-1} = \left[\begin{array}{cc}
1/2 & 1\\[-1.5ex]
-1/2 & 1
\end{array} \right]~,
\]

\begin{minipage}{.85\linewidth}
\vspace{-12pt}
\begin{align*}
\big(A^{-1} \otimes A^{-1}\big)\big(\langle 00| + \langle 11|\big) &= \frac{1}{2}\hspace{1pt}\text{sPf}^*\left(\left[ \begin{array}{cc}
0 & 1\\[-1.5ex]
-1 & 0
\end{array}\right]\right)~,\\
(A \otimes B \otimes B)\big(| 000\rangle + |111\rangle\big) &= 2\hspace{1pt}\text{sPf}\left(\left[ \begin{array}{ccc}
0 & 1/2 & 1/2\\[-1.5ex]
-1/2 & 0 & 1/4\\[-1.5ex]
-1/2 & -1/4 & 0
\end{array}\right]\right)~,\\
\big(B^{-1} \otimes A^{-1} \otimes A^{-1}\big)\big(\langle 000| + \langle 111|\big) &= \frac{1}{2}\hspace{1pt}\text{sPf}^{*}\left(\left[ \begin{array}{ccc}
0 & 1/2 & 1/2\\[-1.5ex]
-1/2 & 0 & 1\\[-1.5ex]
-1/2 & -1 & 0
\end{array}\right]\right)~,\\
B^{-1}\big(\langle 0| + \langle 1|\big) = 2\hspace{1pt}\text{sPf}^{*}\left(\left[ \begin{array}{c}
0
\end{array}\right]\right)~, &\quad
A\big(|0\rangle + |1\rangle\big) = 2\hspace{1pt}\text{sPf}\left(\left[ \begin{array}{c}
0
\end{array}\right]\right)~.
\end{align*}
\end{minipage}
\begin{minipage}{.15\linewidth}
\begin{center}
\vspace{-5pt}
\hspace{-3pt}\includegraphics[scale=.22,clip=true,trim=90 659 0 85]{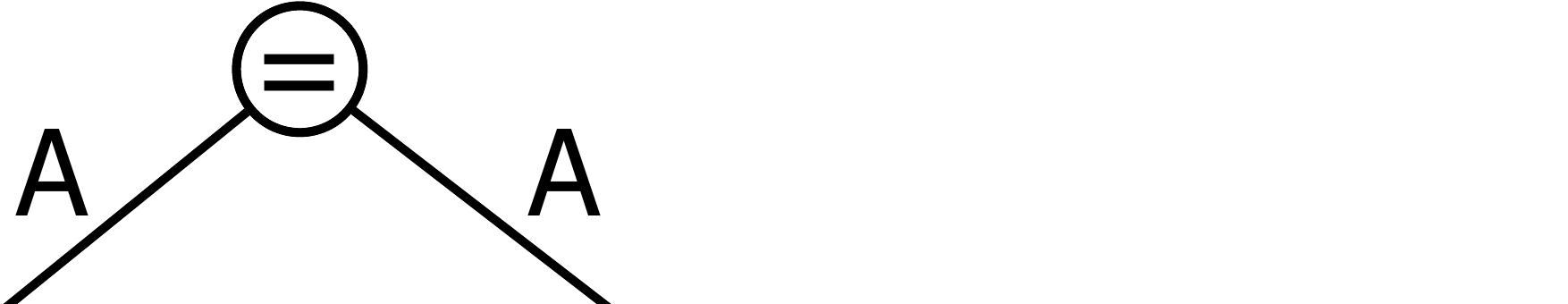}\\
\vspace{-10pt}
\hspace{-9pt}\includegraphics[scale=.25,clip=true,trim=80 580 320 30]{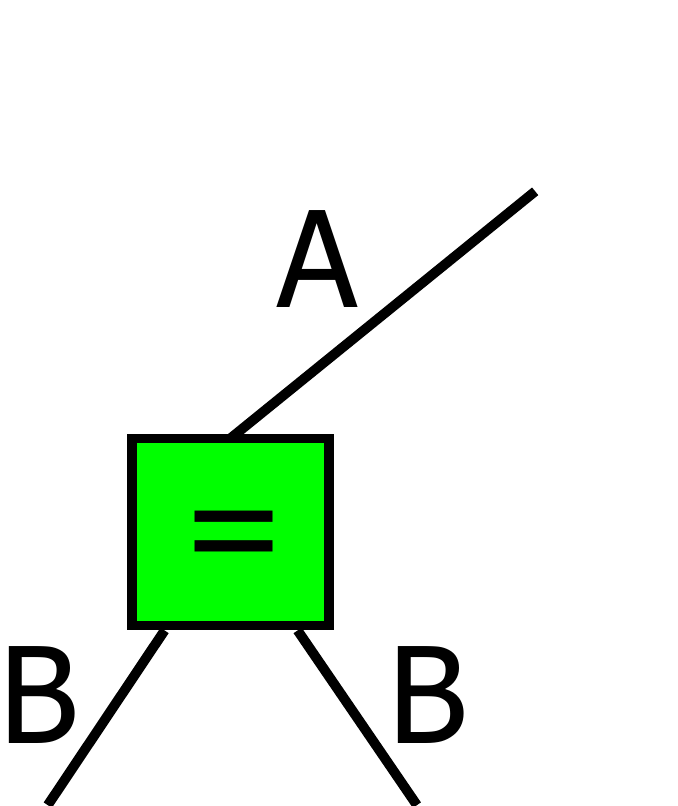}\\
\vspace{-9pt}
\hspace{-9pt}\includegraphics[scale=.25,clip=true,trim=80 590 350 40]{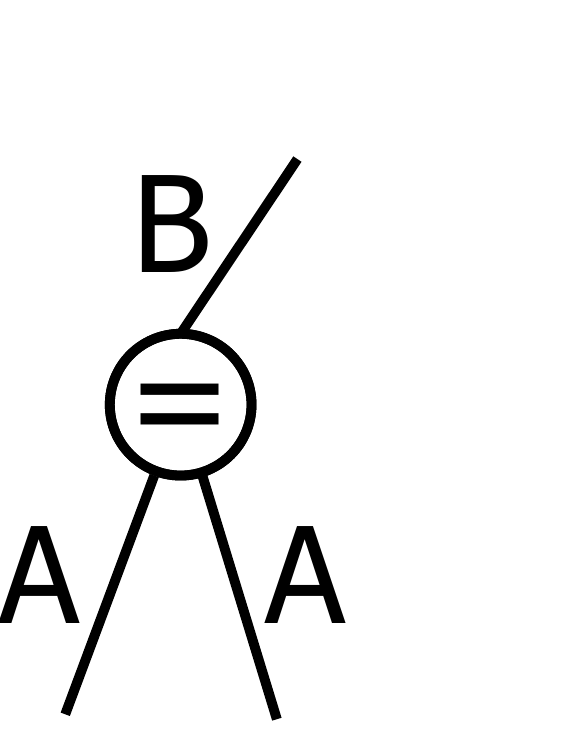}\\
\hspace{-25pt}\includegraphics[scale=.25,clip=true,trim=80 640 350 50]{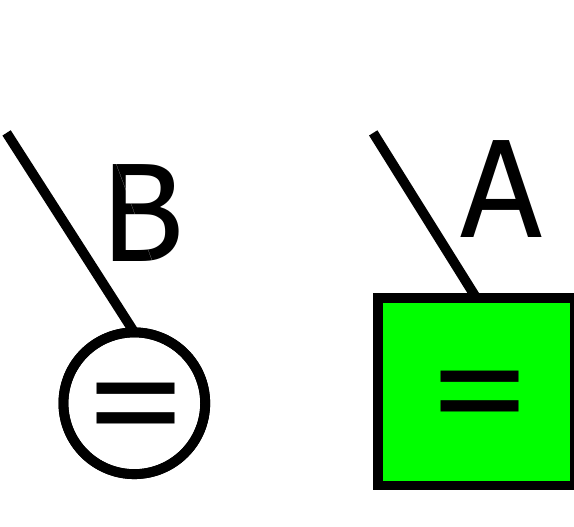}
\end{center}
\end{minipage}

\hfill 
\end{proof}

We will refer to the trees defined by Theorem \ref{thm_bool_tree}.2 as \emph{Boolean trees}, since they are meant to represent 0/1 variables or $n$-arity \textsc{equal} gates/cogates. We first observe that these trees can be extended to an arbitrary height. Since the gates ($\Box$) are Pfaffian under the change of basis $A,BB$, and the cogates ($\circ$) are Pfaffian under the change of basis $B,AA$, the gates can always be connected to the cogates, and vice versa, allowing the tree to be indefinitely ``grown". We next observe that any gate can be ``sealed off" with the 1-arity Boolean cogate $\langle 0| + \langle 1|$ (which is Pfaffian under change of basis $B$), and that any cogate can likewise be ``sealed off" with the 1-arity Boolean gate $|0\rangle + |1\rangle$ (which is Pfaffian under the change of basis $A$). Therefore, for any integer $n$, there exists a Boolean tree that is equivalent to an $n$-arity Boolean gate/cogate. For example, Fig. \ref{fig_thm_bt}.b is equivalent to 5-arity Boolean cogate. 

One metric of comparison for Boolean trees and $n$-arity Boolean gates/cogates is the complexity of their respective Gr\"obner bases. For example, the following 18 equations in degrees one, two and three is the Gr\"obner basis for a Boolean tree of arbitrary size:
\begin{align*}
a_{10}+a_{11} = 
a_{00}-a_{01} = 
2a_{11}^2-1 = 
2a_{01}a_{11}-\det(A) =
2a_{01}^2-\det(A)^2 = 0~ ,\\
a_{01}\det(A)^{-1}-a_{11} =
a_{11}\det(A)-a_{01} =
\det(A)\det(A)^{-1}-1 = 0~,\\
b_{10}-b_{11} = 
b_{00}+b_{01} = 
2b_{11}-1 =
b_{01}+\det(B) = 
\det(B)\det(B)^{-1}-1=0~,\\
\big(\det(B)^{-1}\big)^2-2a_{01} = 
2a_{11}\det(B)-\det(A)^{-1}\det(B)^{-1} = 
2a_{01}\det(B)-\det(B)^{-1} = 0~,\\
a_{11}\det(A)^{-1}-\det(B)^2 = 
2\det(B)^3-\big(\det(A)^{-1}\big)^2\det(B)^{-1} = 0~.
\end{align*}
By comparison, the Gr\"obner basis of the 9-arity Boolean cogate contains 82 polynomials, ranging from degrees two to ten. Furthermore, the complexity of the Gr\"obner basis of the $n$-arity Boolean gate/cogate obviously grows with respect to $n$, while the complexity of the ``$n$-arity" Boolean tree remains constant. Since the algebraic approach to exploring Pfaffian circuits relies on the complexity of computing the Gr\"obner basis, the question of categorizing classes of compatible gates/cogates (representing constraints) is obviously more efficiently pursued with Boolean trees.

Despite the algebraic advantages of using a Boolean tree (as opposed to an $n$-arity \textsc{equal} gate/cogate), at first glance it may seem that introducing two distinct bases $A$ and $B$ may complicate the question of identifying compatible constraint gates/cogates. The two obstacles to overcome are 1) that the overall tensor contraction network must remain bipartite, and 2) that each of the gates/cogates must be simultaneously Pfaffian. However, the following proposition demonstrates that it is possible to treat the bases $A$ and $B$ as virtually interchangeable, since a ``bridge" can always be built from gate to cogate, and from basis $A$ to basis $B$ (and vice versa).

\begin{proposition} \label{prop_join} Given matrices $A, B$ as in Prop. \ref{prop_bool_gcg},
the following 2-arity gates/cogates are Pfaffian under both $(A \otimes A)$ and $(B \otimes B)$: $\langle 00| + \langle 11|, |00\rangle +  |11\rangle$ (\textsc{Equal}), and $\langle 10| + \langle 01|, |10\rangle +  |01\rangle$ (\textsc{not})~.
\end{proposition}

Furthermore, the following is a direct corollary of Theorem \ref{thm_bool_tree}.

\begin{corollary} \label{cor_br} Given matrices $A, B$ as in Prop. \ref{prop_bool_gcg},
the following two gate/cogate combinations are Pfaffian under the indicated changes of bases.
\begin{center}
\hspace{100pt}\includegraphics[scale=.21,clip=true,page=1,trim=80 466 70 90]{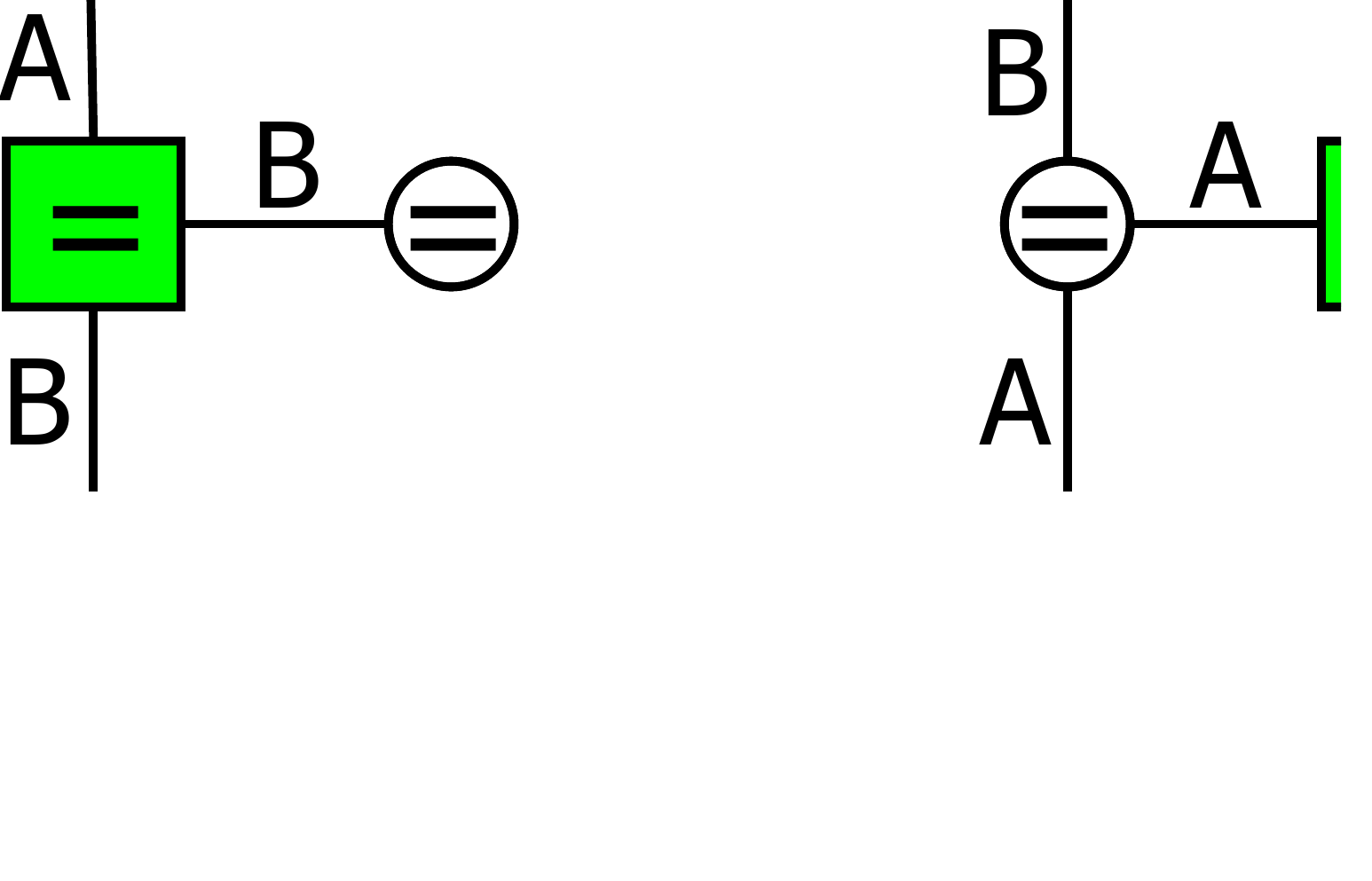}
\hspace{-7pt}\includegraphics[scale=.21,clip=true,page=2,trim=80 466 0 90]{bridge_AB.pdf}
\end{center}
\vspace{-40pt}
\hfill $\Box$
\end{corollary}

The significance of Prop. \ref{prop_join} is that a gate can be arbitrarily changed into a cogate, without changing the state on the outgoing edge (\textsc{equal}), or by flipping the state  (\textsc{not}) from $|0\rangle$ to $|1\rangle$ (or $\langle 0|$ and $\langle 1|$, respectively). The significance of Cor. \ref{cor_br} is that not only can a gate be changed to a cogate, but the \emph{basis on the outgoing edge itself can be changed} (from $A$ to $B$, or vice versa) without altering the state on the outgoing edge. We will now investigate how to categorize the gates/cogates that successfully pair with a Boolean tree. Towards that end, we introduce the following definition.

\begin{defi}
A gate $G$ is \emph{invariant under complement}, if for every ket $|I\rangle$ in the tensor $G$, the complement ket $\overline{|I\rangle}$ also appears in the tensor.
\end{defi}

\begin{example} \label{ex_g3_com_prop} Here are two 3-arity gates that are invariant under complement:
\vspace{-5pt}
\begin{align*}
|010\rangle  + |001\rangle  + |110\rangle  + |101\rangle ~,& \quad 
|100\rangle  + |001\rangle  + |110\rangle  + |011\rangle ~.
\end{align*}
\vspace{-25pt}
\end{example}

\begin{theorem} \label{thm_bt_pair_hom} Given matrices $A, B$ as in Prop. \ref{prop_bool_gcg},
then the following one 1-arity gate/cogate, three 2-arity gates/cogates, 15 3-arity gates/cogates and 117 4-arity gates/cogates are Pfaffian under the change of basis $(A \otimes \cdots \otimes A)$ and $\big(B^{-1} \otimes \cdots \otimes B^{-1}\big)$, respectively. For convenience, we only list the gates, and the 117 4-arity gates are listed in the appendix.

\noindent Here is the one 1-arity gate: $|0\rangle + |1\rangle$~.\\
Here are the three 2-arity gates: 
$|00\rangle + |11\rangle, 
|10\rangle + |01\rangle, |00\rangle + |10\rangle + |01\rangle + |11\rangle~.$\\
Here are the 15 3-arity gates:
{\footnotesize{\begin{align*}
|000\rangle + |111\rangle~, \quad
|100\rangle + |011\rangle~,&\quad
|010\rangle + |101\rangle~, \quad
|001\rangle + |110\rangle~,\\
|000\rangle + |100\rangle + |011\rangle + |111\rangle~, & \quad
|000\rangle + |010\rangle + |101\rangle + |111\rangle~,\\
|000\rangle + |001\rangle + |110\rangle + |111\rangle~,& \quad
|100\rangle + |010\rangle + |101\rangle + |011\rangle~,\\
|100\rangle + |001\rangle + |110\rangle + |011\rangle~,& \quad
|010\rangle + |001\rangle + |110\rangle + |101\rangle~,\\
|000\rangle + |100\rangle + |010\rangle + |101\rangle + |011\rangle + |111\rangle~,& \quad
|000\rangle + |100\rangle + |001\rangle + |110\rangle + |011\rangle + |111\rangle~,
\end{align*}
\vspace{-20pt}
\begin{align*}
|000\rangle + |010\rangle + |001\rangle + |110\rangle + |101\rangle + |111\rangle~,& \quad
|100\rangle + |010\rangle + |001\rangle + |110\rangle + |101\rangle + |011\rangle~,\\
|000\rangle + |100\rangle + |010\rangle + |001\rangle & + |110\rangle + |101\rangle + |011\rangle + |111\rangle~.
\end{align*}}}
The 117 4-arity gates are listed in  Appendix \ref{app_thm_bt_pair_hom_4}.
\end{theorem}

We observe that \emph{every} gate/cogate that pairs with the Boolean trees is invariant under complement. However, \emph{not} every gate/cogate invariant under complement is Pfaffian, let along pairing properly with the Boolean trees.
 
\begin{example} \label{ex_g_comp_prop_fail} We observe that not every gate invariant under complement is Pfaffian under some change of basis. For example, consider the following 4-arity gate:
\begin{align*}
|0000\rangle + |1000\rangle + |0100\rangle + |0010\rangle + |0111\rangle + |1011\rangle + |1101\rangle + |1111\rangle
\end{align*}
By an algebraic computation, we know that there do \textbf{not} exist any matrices $A_1,\ldots, A_4 \in \mathbb{C}^{2 \times 2}$ such that the gate is Pfaffian under the change of basis $(A_1 \otimes \cdots \otimes A_4)$. \hfill $\Box$
\end{example}
\vspace{-5pt}
It remains an open question to determine which gates/cogates invariant under complement are Pfaffian under \emph{some} heterogeneous change of basis.


We recall that Boolean trees are Pfaffian under a heterogeneous change of basis only. Furthermore, Prop. \ref{prop_join} and Cor. \ref{cor_br} together allow ``bridges" to be built between change of basis $A$ and change of basis $B$. In Theorem \ref{thm_bt_pair_hom} we explored a categorization of gates/cogates that pair with the Boolean trees under a homogeneous change of basis. However, since the Boolean trees allow two distinct changes of bases ($A$ and $B$), it is logical to investigate the types of gates that are Pfaffian under this heterogeneous change of basis.

\begin{theorem} \label{thm_bt_pair_het} Given matrices $A, B$ as in Prop. \ref{prop_bool_gcg},
there exists a matrix $C \in \mathbb{C}^{2 \times 2}$ such that \textbf{each} of the following 3-arity gates/cogates are Pfaffian under the heterogeneous change of basis $(A \otimes B \otimes C)$ (or $A^{-1} \otimes B^{-1} \otimes C^{-1}$, respectively).
\begin{center}
\includegraphics[scale=.20,page=1,clip=true,trim=0 620 0 75]{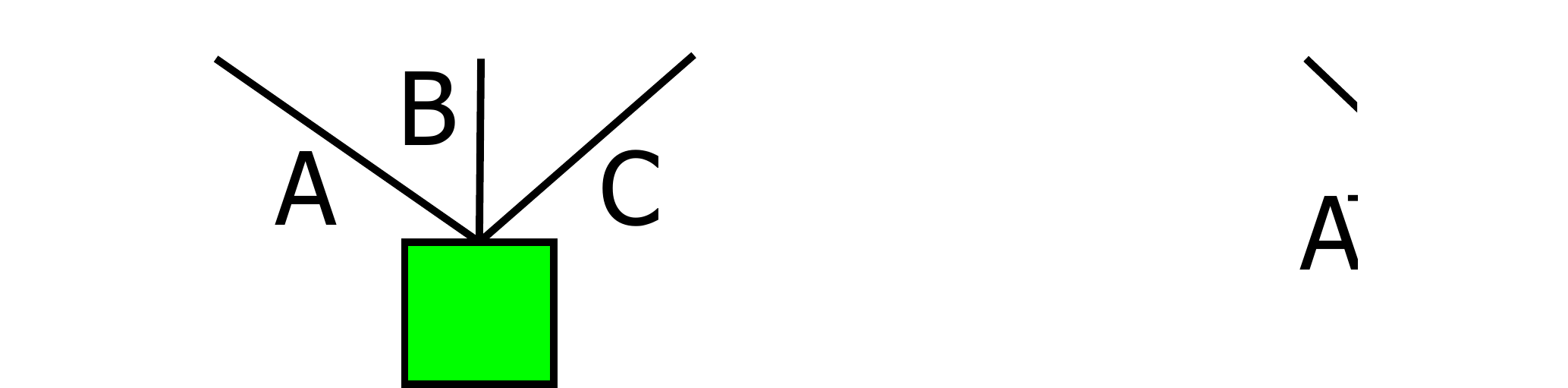}
\hspace{-35.5pt}\includegraphics[scale=.20,page=2,clip=true,trim=0 620 55 75]{bool_pair_con_gcg_het.pdf}
\end{center}
Here are the 3-arity gates:
{\footnotesize{\begin{align*}
|000\rangle + |110\rangle + |111\rangle~,\quad 
|000\rangle + |001\rangle + |111\rangle~, \quad &
|000\rangle + |001\rangle + |110\rangle~, \quad
|001\rangle + |110\rangle + |111\rangle~,\\
|010\rangle + |101\rangle + |011\rangle~,\quad
|100\rangle + |101\rangle + |011\rangle~, \quad&
|100\rangle + |010\rangle + |011\rangle~,\quad 
|100\rangle + |010\rangle + |101\rangle~,\\
|010\rangle + |001\rangle + |110\rangle + |101\rangle + |011\rangle + |111\rangle~, \quad & 
|000\rangle + |100\rangle + |001\rangle + |101\rangle + |011\rangle + |111\rangle~,\\
|000\rangle + |100\rangle + |010\rangle + |110\rangle + |011\rangle + |111\rangle~,\quad&
|000\rangle + |100\rangle + |010\rangle + |001\rangle + |110\rangle + |101\rangle~.
\end{align*}}}
Here are the 3-arity cogates:
{\footnotesize{
\begin{align*}
\langle 000| + \langle 110| + \langle 111|~, \quad
\langle 000| + \langle 001| + \langle 111|~, \quad&
\langle 000| + \langle 001| + \langle 110|~,\quad
\langle 001| + \langle 110| + \langle 111|~,\\
\langle 010| + \langle 101| + \langle 011|~, \quad
\langle 100| + \langle 101| + \langle 011|~, \quad&
\langle 100| + \langle 010| + \langle 011|~, \quad
\langle 100| + \langle 010| + \langle 101|~,\\
\langle 100| + \langle 001| + \langle 110| + \langle 101| + \langle 011| + \langle 111|~, \quad &
\langle 000| + \langle 010| + \langle 001| + \langle 101| + \langle 011| + \langle 111|~,\\
\langle 000| + \langle 100| + \langle 010| + \langle 110| + \langle 101| + \langle 111|~, \quad &
\langle 000| + \langle 100| + \langle 010| + \langle 001| + \langle 110| + \langle 011|~.
\end{align*}}}
\end{theorem}

We conclude this section by observing that Theorem \ref{thm_bt_pair_het} on its own is only a partial result. However, just as we were able to use Prop. \ref{prop_join} and Cor. \ref{cor_br} to build bridges from change of basis $A$ to change of basis $B$, we strongly suspect that on a case-by-case basis, it will be possible to build a bridge from change of basis $C$ to changes of bases $A$ and $B$, respectively. In that case, each of the gates/cogates listed in Theorem \ref{thm_bt_pair_het} would pair with the Boolean trees, and the set of building blocks for 0/1 \#CSP solvable in polynomial-time will have increased.

We also observe that the gates/cogates listed in Theorem \ref{thm_bt_pair_het} are not the same. For example, gate $|010\rangle + |001\rangle + |110\rangle + |101\rangle + |011\rangle + |111\rangle$ is Pfaffian, but the corresponding cogate is not. However, cogate $\langle 100| + \langle 001| + \langle 110| + \langle 101| + \langle 011| + \langle 111|$ is Pfaffian, and we observe that the only difference between the two is in their first term; the gate contains ket $|010 \rangle$ and the cogate contains bra $\langle  100 |$, which is a difference of a single set element. This is the defining characteristic of the different gates and cogates, but formulating the theoretical principle remains elusive. However, we highlight that since the gates/cogates in Theorem \ref{thm_bt_pair_het} are different, by positing the existence of a third matrix $C$, we have actually expanded the number of allowable constraints. Therefore, by considering heterogeneous changes of bases, we have increased the computing 
power of Pfaffian circuits. Finally, we observe that the Boolean trees are an example of a gate/cogate (see Prop. \ref{prop_bool_gcg}) that can be modeled as a collection of gates/cogates under a heterogeneous change of basis only. This observation plays directly into the next section, and is an example of a fundamental question posed in this paper: if a given gate/cogate is \emph{not} Pfaffian, does there exists a combinatorial structure of gates and cogates that \emph{is} Pfaffian under some heterogeneous change of basis?

\section{Planarity and Tensor Contraction Networks} \label{sec_swap} The Pfaffian circuit evaluation theorem (Theorem \ref{thm_pfaff_eval}) applies only to \emph{planar}, Pfaffian circuits with an even number of edges. Since planarity is a key hypothesis of Theorem \ref{thm_pfaff_eval} , it is well-worth discussing the barriers to transforming a non-planar circuit into a planar one. We observe that there are no inherent complexity theoretic blocks to non-planar Pfaffian circuits, since the \emph{Pfaffian} gates/cogates under that construction would be an extremely specific set. In this section, we reduce the notion of a swap gate to a tensor contraction network fragment with two extra ``dangling'' edges (the dangling end is not connected to any vertex), one of which must be  fixed to the $\langle 1|$ state and the other of which must be fixed to the $\langle 0|$ state.

The \textsc{swap} gate is a 4-arity gate (2 input, 2 output) that simply swaps the order of the input.

\begin{minipage}{.35\linewidth}
\begin{center}
\hspace{-10pt}\includegraphics[scale=.22,clip=true,page=1,trim=80 560 80 80]{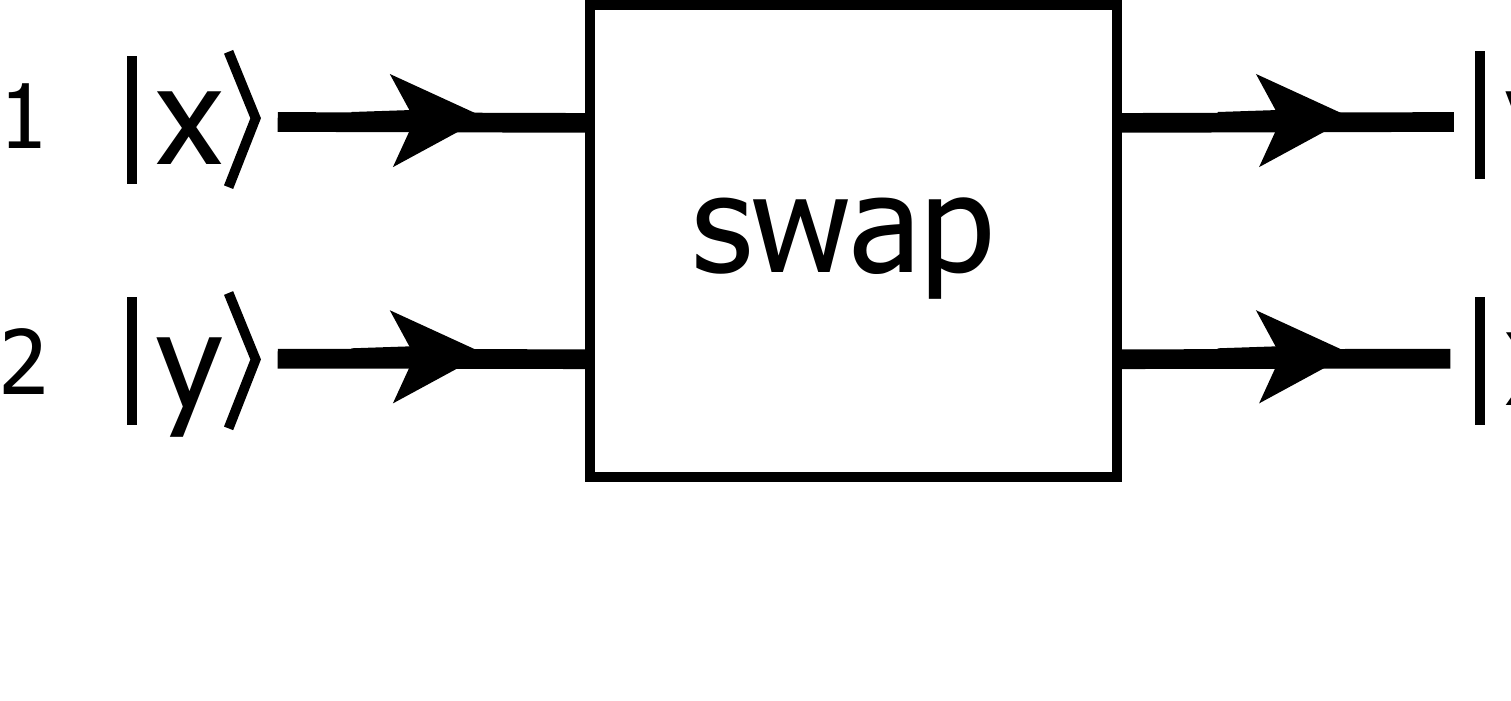}
\hspace{-4pt}\includegraphics[scale=.22,clip=true,page=2,trim=80 560 455 80]{swap_xy_lbl.pdf}
\end{center}
\end{minipage}
\hspace{-10pt}\begin{minipage}{.65\linewidth}
\vspace{-8pt}
\begin{align*}
\textsc{swap} = |\text{xy,xy}\rangle &= |0_10_2,0_30_4\rangle + |0_11_2,0_31_4\rangle\\
&\phantom{=}\hspace{9pt} + |1_10_2,1_30_4\rangle + |1_11_2,1_31_4\rangle~. 
\end{align*}
\vspace{1pt}
\end{minipage}

\vspace{-10pt}
The \textsc{swap} gate/cogate can be used to untangle an edge crossing, and turn a \emph{non-planar} tensor contraction network into a \emph{planar} tensor contraction network.
\begin{example} \label{ex_cross} In this example, we demonstrate how an edge crossing in a tensor contraction network can be untangled by a \textsc{swap} gate/cogate.

\begin{figure}[h!]
\vspace{-5pt}
\begin{minipage}{.50\linewidth}
\begin{center}
\includegraphics[scale=.20,clip=true,trim=0 590 120 80]{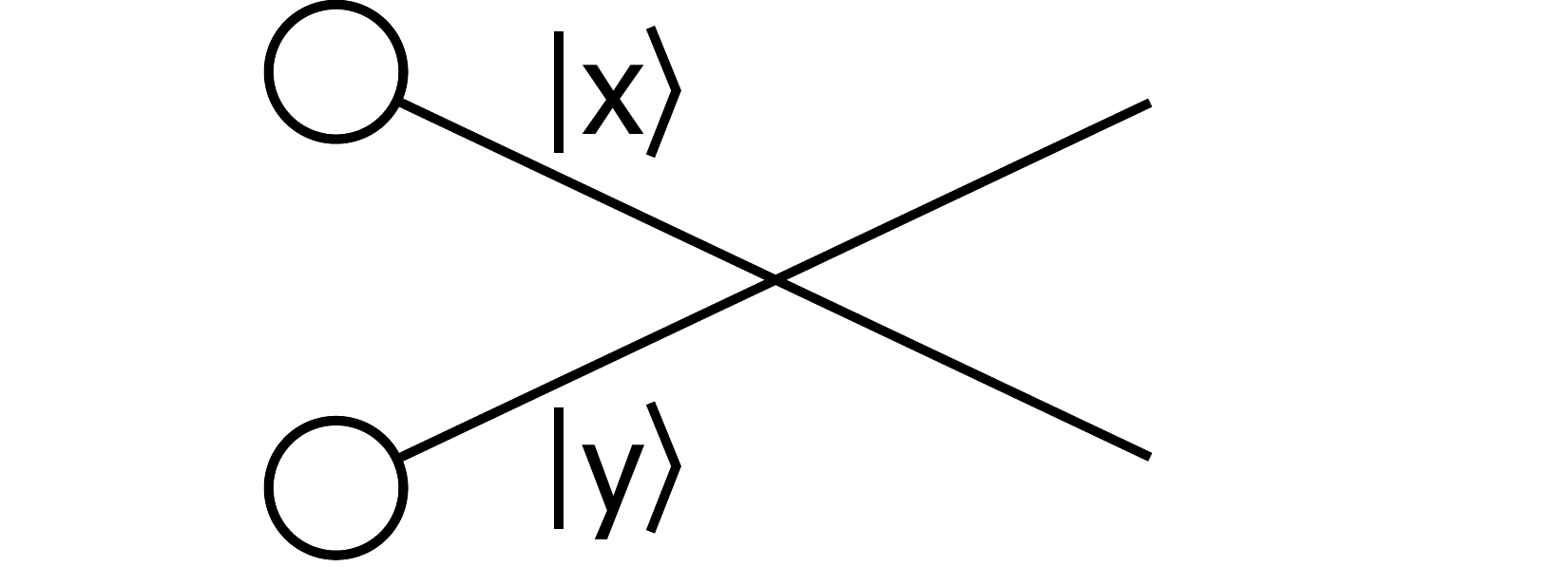}
\end{center}
\vspace{-1pt}
\hspace{87pt}(a)
\end{minipage}
\hspace{-20pt}
\begin{minipage}{.50\linewidth}
\begin{center}
\includegraphics[scale=.20,clip=true,trim=0 578 50 80]{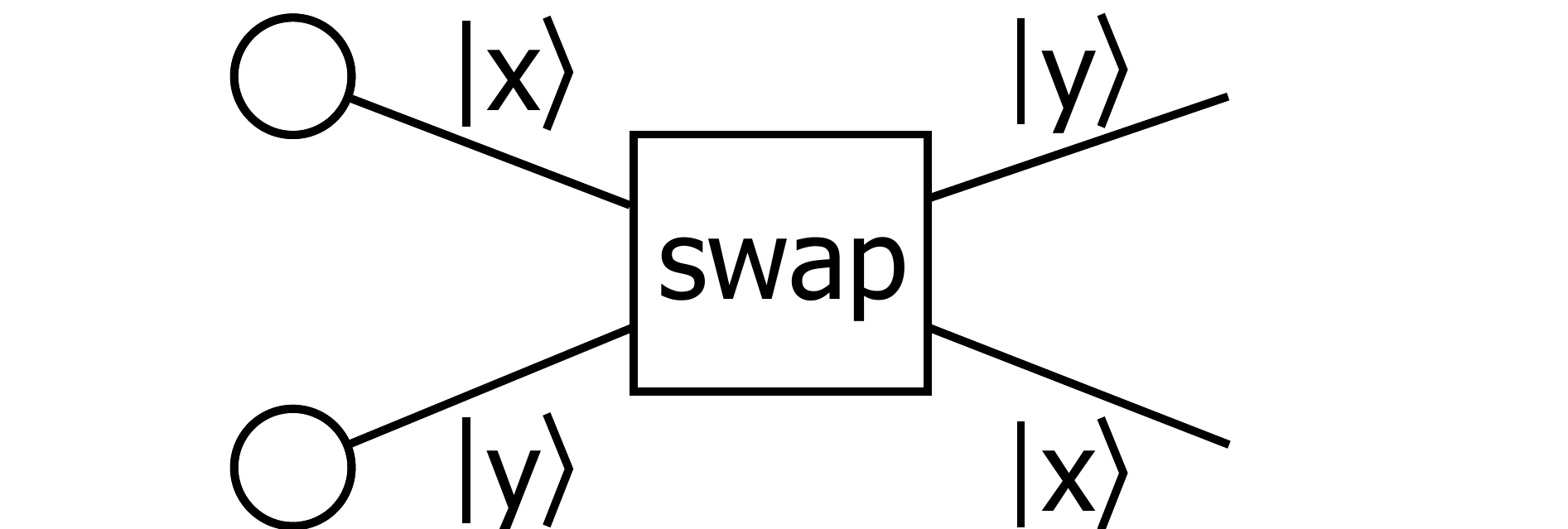}
\end{center}
\vspace{-1pt}
\hspace{90pt}(b)
\end{minipage}
\end{figure}

\vspace{-10pt}
In (a), we see a crossing where the edges are in state $|x\rangle$ and $|y\rangle$. In (b), we see a crossing replaced with a \textsc{swap} gate. We observe that when a crossing is replaced with the a \textsc{swap} gate (or cogate), the three properties of a Pfaffian circuit must be preserved: (1) bipartite, (2) Pfaffian and (3) contains an even number of edges. \hfill $\Box$
\end{example}
\newpage
\begin{theorem} \label{thm_swap_het}
$\phantom{xxx}$
\begin{enumerate}
	\item There do \textbf{not} exist matrices $A,B,C,D \in \mathbb{C}^{2 \times 2}$ such that the \textsc{swap} gate is Pfaffian under the change of basis $(A \otimes B \otimes C \otimes D)$.

\vspace{-2pt}
\begin{minipage}{.35\linewidth}
\begin{center}
\hspace{-15pt}\includegraphics[scale=.20,page=1,clip=true,trim=50 560 50 90]{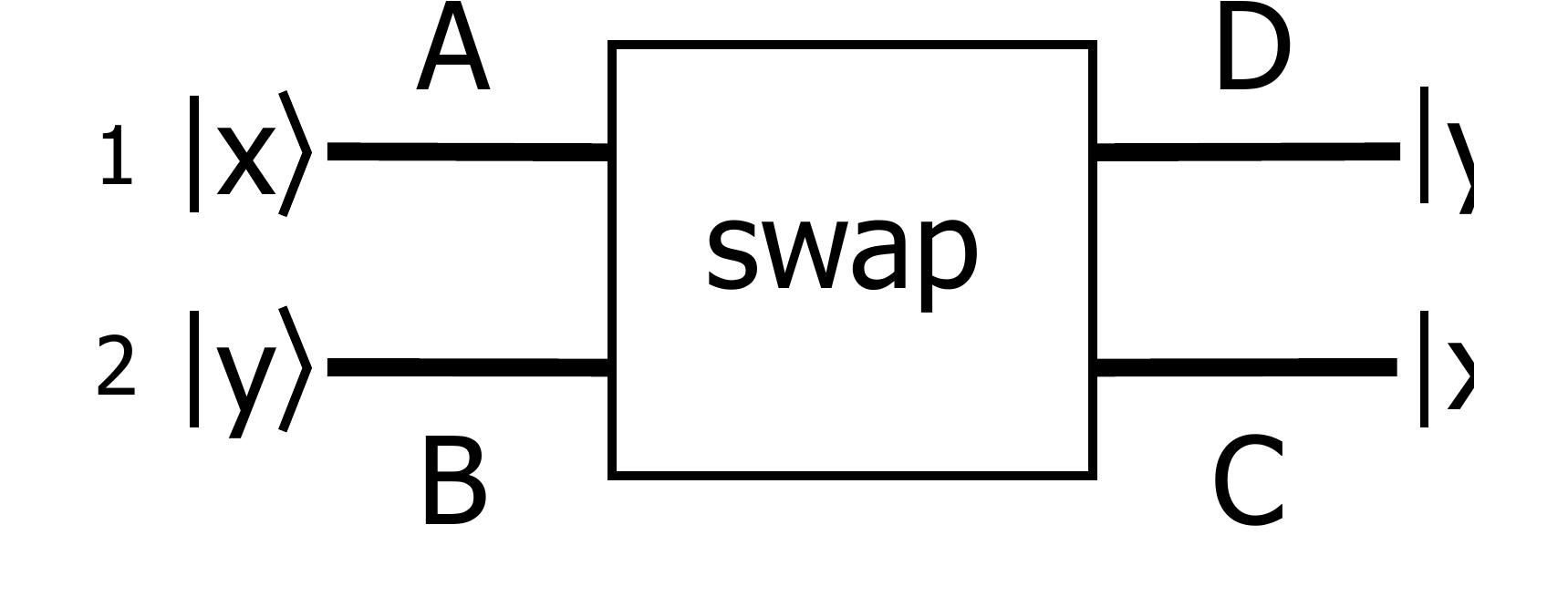}
\hspace{-16pt}\includegraphics[scale=.20,page=2,clip=true,trim=50 560 465 90]{swap_xy_het_nolbl.pdf}
\end{center}
\end{minipage}
\begin{minipage}{.65\linewidth}
\begin{align*}
\textsc{swap}=|\text{xy,xy}\rangle &= |0_10_2,0_30_4\rangle + |0_11_2,0_31_4\rangle\\
&\phantom{=}\hspace{9pt} + |1_10_2,1_30_4\rangle + |1_11_2,1_31_4\rangle~. 
\end{align*}
\vspace{1pt}
\end{minipage}
		\item There do \textbf{not} exist a matrices $A,B,C,D \in \mathbb{C}^{2 \times 2}$ such that the \textsc{swap} cogate is Pfaffian under the change of basis $(A^{-1} \otimes B^{-1} \otimes C^{-1} \otimes D^{-1})$. 

\vspace{-2pt}
\begin{minipage}{.35\linewidth}
\begin{center}
\hspace{-13pt}\includegraphics[scale=.20,page=1,clip=true,trim=83 560 59 87]{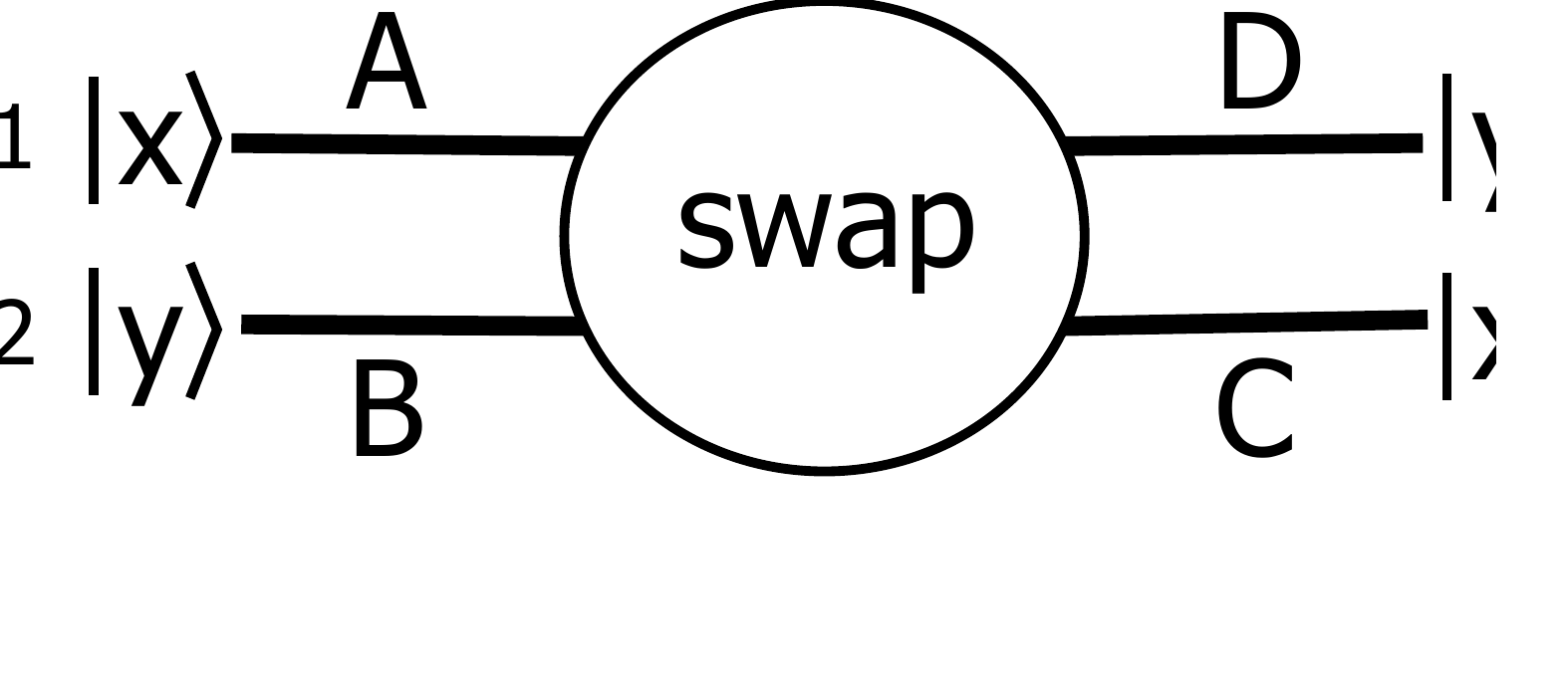}
\hspace{-8pt}\includegraphics[scale=.20,page=2,clip=true,trim=80 560 467 87]{swap_cg_xy_het_nolbl.pdf}
\end{center}
\end{minipage}
\begin{minipage}{.65\linewidth}
\begin{align*}
\textsc{swap}=\langle\text{xy,xy}| &= \langle 0_10_2,0_30_4| + \langle 0_11_2,0_31_4|\\
&\phantom{=}\hspace{9pt} + \langle 1_10_2,1_30_4| + \langle 1_11_2,1_31_4|~. 
\end{align*}
\vspace{1pt}
\end{minipage}
\end{enumerate}
\end{theorem}

\vspace{-25pt}
\begin{proof} The proof of Theorems \ref{thm_swap_het}.1 and 2 are by Gr\"obner basis computations that run in 1 sec and 26 sec, respectively.
\end{proof}
\vspace{-6pt}
We pause to note that, subsequent to this computation, a proof of Theorem \ref{thm_swap_het} via invariant theory (and computation) was found in \cite{turner2013some}. Having demonstrated that the \textsc{swap} gate/cogate is not Pfaffian, the question then arises as to whether or not the behavior of the \textsc{swap} tensor can be mimicked by a larger collection of Pfaffian gates/cogates. For example, it is certainly well-known that the behavior of the \textsc{swap} operation can be mimicked by chaining three \textsc{cnot}-gates together, but in order to embed that structure in a Pfaffian circuit, each of the three \textsc{cnot} gates/cogates must be simultaneously Pfaffian.

The \textsc{cnot} gate is a 4-arity gate (2 input, 2 output), where one qubit acts as a ``control" (commonly denoted as $\bullet$) and the second qubit is flipped based on whether or not the control qubit is one or zero (commonly denoted as $\oplus$). Since the orientation of the gate in the plane defines the order of the tensor, we define two gates, \textsc{cnot1} and \textsc{cnot2}, depending on the location of the control bit. 

\begin{defi} \label{def_cnot}
$\phantom{xxx}$
\begin{enumerate}
	\item Let \textsc{cnot1} be the following gate (the control bit is on the first wire):

\begin{minipage}{.45\linewidth}
\begin{center}
\includegraphics[scale=.22,page=1,clip=true,trim=50 659 80 80]{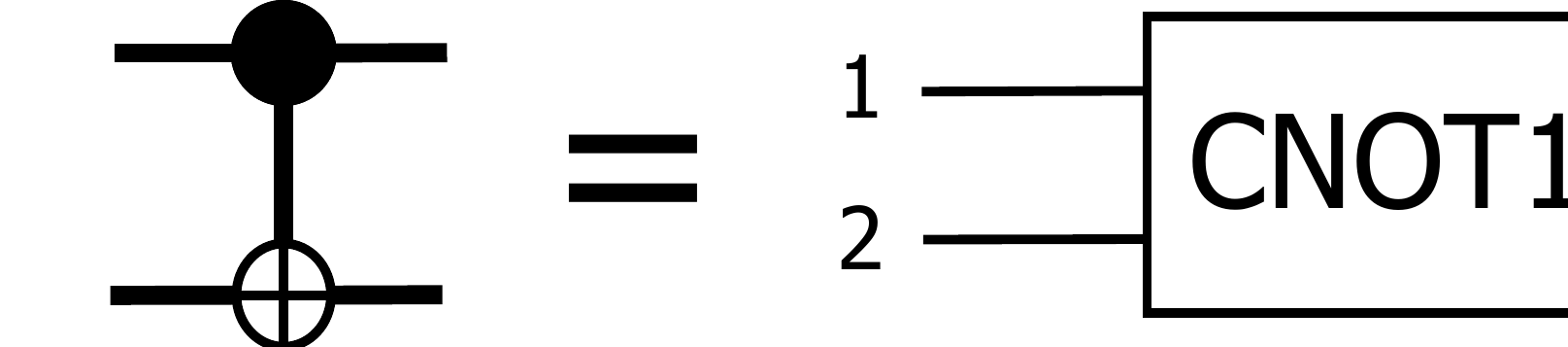}
\hspace{-4pt}\includegraphics[scale=.22,page=2,clip=true,trim=80 659 400 80]{CNOT1_g4.pdf}
\end{center}
\end{minipage}
\begin{minipage}{.55\linewidth}
\vspace{-10pt}
\begin{align*}
\text{\textsc{cnot1}} &= |0_10_2,0_30_4\rangle + |0_11_2,1_30_4\rangle\\
&\phantom{=}\hspace{9pt} + |1_10_2,1_31_4\rangle + |1_11_2,0_314\rangle~.
\end{align*}
\end{minipage}
	\item Let \textsc{cnot2} be the following gate (the control bit is on the second wire):

\begin{minipage}{.45\linewidth}
\begin{center}
\includegraphics[scale=.22,page=1,clip=true,trim=50 658 80 60]{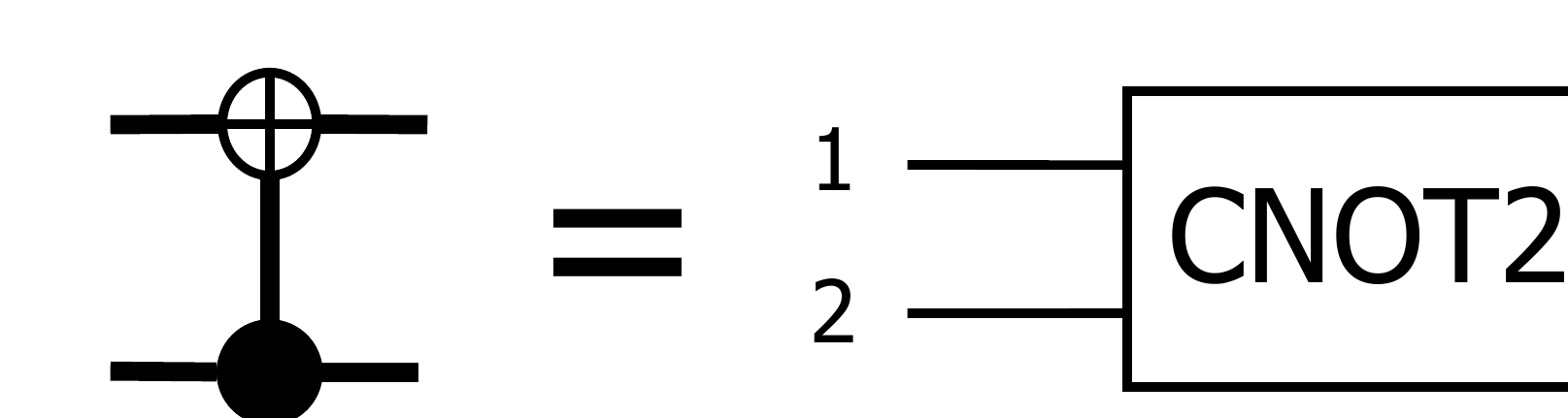}
\hspace{-4pt}\includegraphics[scale=.22,page=2,clip=true,trim=80 658 400 60]{CNOT2_g4.pdf}
\end{center}
\end{minipage}
\begin{minipage}{.55\linewidth}
\vspace{-10pt}
\begin{align*}
\text{\textsc{cnot2}} &= |0_10_2,0_30_4\rangle + |0_11_2,1_31_4\rangle\\
&\phantom{=}\hspace{9pt} + |1_10_2,0_31_4\rangle + |1_11_2,1_30_4\rangle~.
\end{align*}
\end{minipage}
\end{enumerate}
\end{defi}
\begin{example} We will now demonstrate how three \textsc{cnot} gates chained together can simulate a swap gate. In order to preserve the bipartite nature of a Pfaffian circuit, we link the \textsc{cnot} gates together via the \textsc{equal} cogate $\langle 00| + \langle 11 |$ (denoted by $\cgequal$).

\begin{figure}[h!]
\vspace{-5pt}
\begin{minipage}{.49\linewidth}
\begin{center}
\includegraphics[scale=.24,clip=true,trim=0 0 0 0]{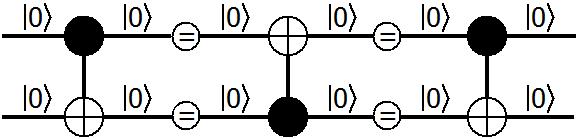}
\end{center}
\hspace{88pt}(a)
\end{minipage}
\begin{minipage}{.49\linewidth}
\begin{center}
\includegraphics[scale=.24,clip=true,trim=0 0 0 0]{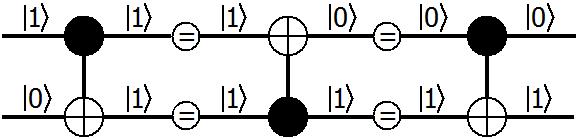}
\end{center}
\hspace{88pt}(b)
\end{minipage}\\
\begin{minipage}{.49\linewidth}
\vspace{5pt}
\begin{center}
\includegraphics[scale=.24,clip=true,trim=0 0 0 0]{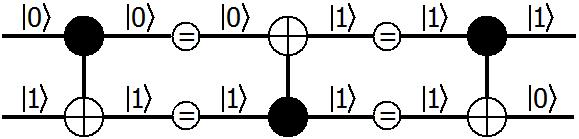}
\end{center}
\hspace{88pt}(c)
\end{minipage}
\begin{minipage}{.49\linewidth}
\vspace{5pt}
\begin{center}
\includegraphics[scale=.24,clip=true,trim=0 0 0 0]{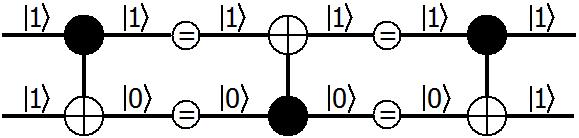}
\end{center}
\hspace{88pt}(d)
\end{minipage}
\end{figure}

\vspace{-20pt}
Considering only the final input and output, we see (a) represents $|00,00\rangle$, (b) represents $|10,10\rangle$, (c), represents $|01,01\rangle$ and finally $(d)$ represents $|11,11\rangle$. Thus, the chain \textsc{cnot1-cnot2-cnot1} is equivalent in behavior to the \textsc{swap} tensor. \hfill $\Box$
\end{example}

\begin{theorem} \label{thm_cnot}
$\phantom{xxx}$
\begin{enumerate}
	\item There does \textbf{not} exist any matrix $A \in \mathbb{C}^{2 \times 2}$ such that the \textsc{cnot1} gate is Pfaffian under the homogeneous change of basis $A$.

\begin{minipage}{.35\linewidth}
\begin{center}
\includegraphics[scale=.20,clip=true,trim=50 572 80 92]{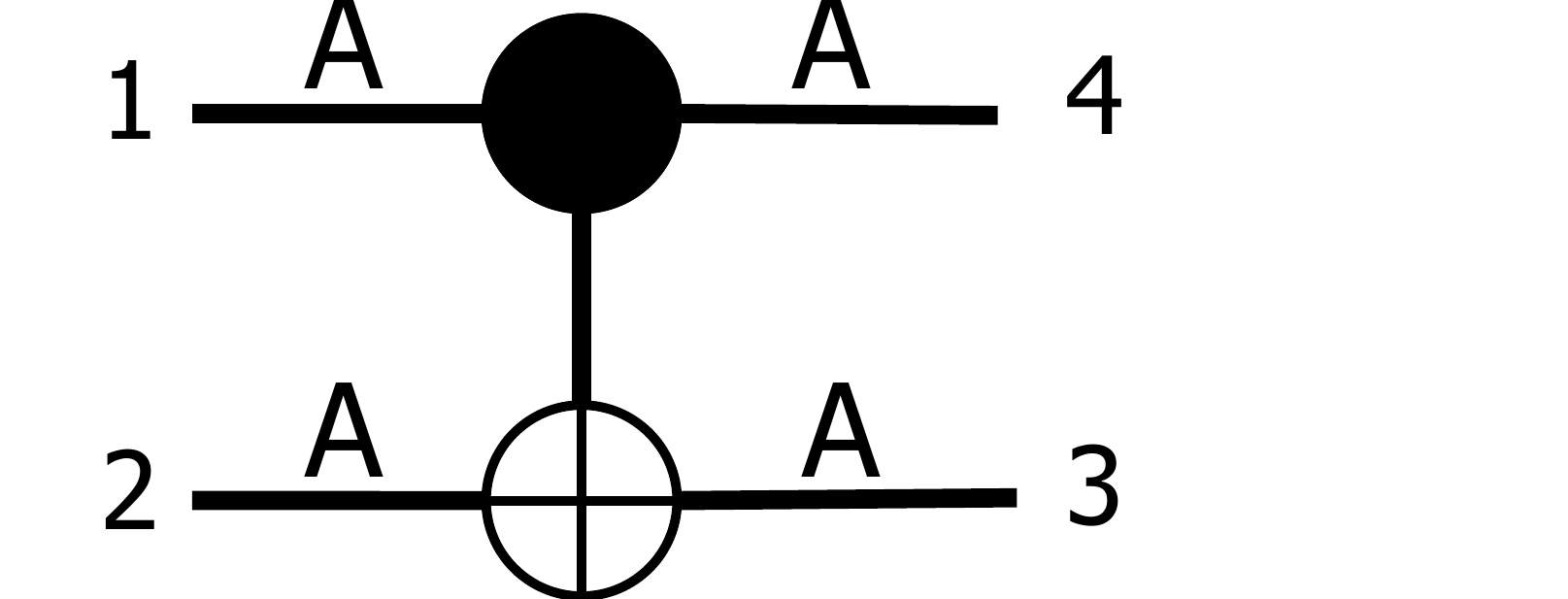}
\end{center}
\end{minipage}
\hspace{-30pt}\begin{minipage}{.65\linewidth}
\begin{align*}
, \quad \quad \text{\textsc{cnot1}} &= |0_10_2,0_30_4\rangle + |0_11_2,1_30_4\rangle\\
&\phantom{=}\hspace{9pt} + |1_10_2,1_31_4\rangle + |1_11_2,0_31_4\rangle~.
\end{align*}
\vspace{1pt}
\end{minipage}
	\item There \textbf{do} exist matrices $A,B,C,D \in \mathbb{C}^{2 \times 2}$ such that the \textsc{cnot1} gate is Pfaffian under the heterogeneous change of basis $A,B,C$ and $D$.

\begin{minipage}{.35\linewidth}
\begin{center}
\includegraphics[scale=.20,clip=true,trim=50 572 80 92]{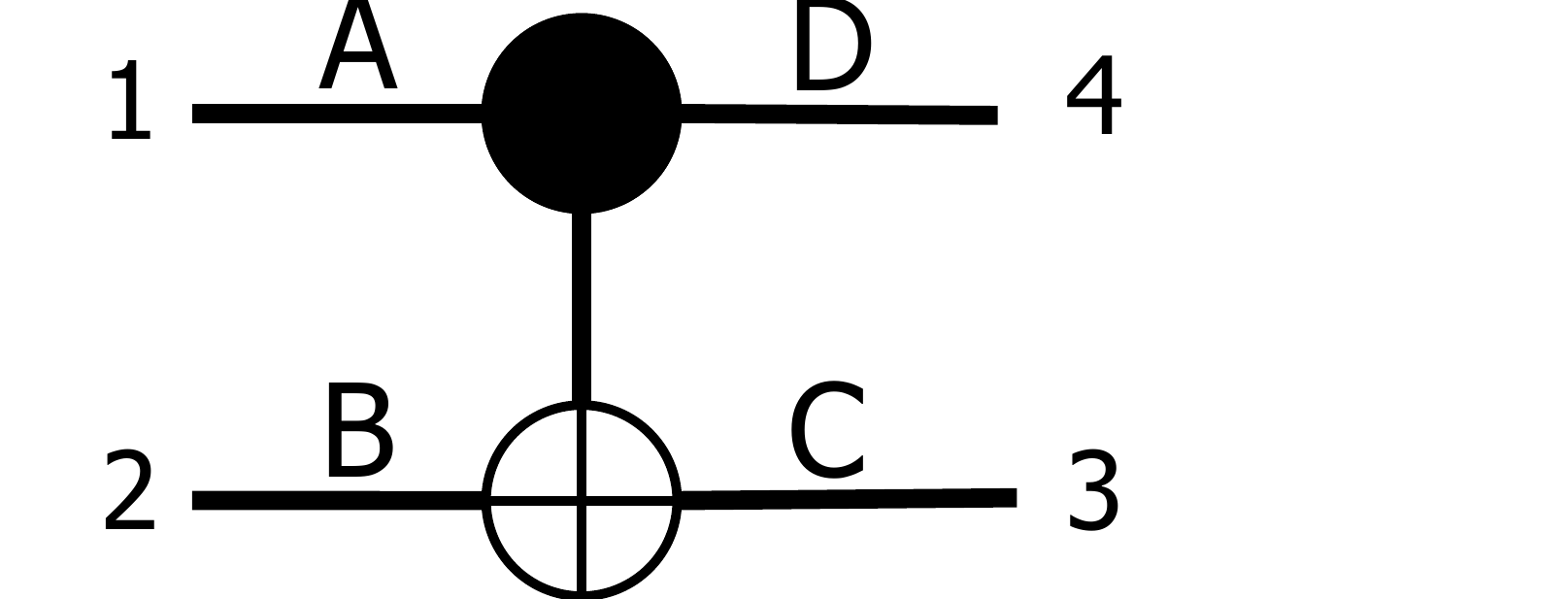}
\end{center}
\end{minipage}
\hspace{-30pt}\begin{minipage}{.65\linewidth}
\begin{align*}
, \quad \quad \text{\textsc{cnot1}} &= |0_10_2,0_30_4\rangle + |0_11_2,1_30_4\rangle\\
&\phantom{=}\hspace{9pt} + |1_10_2,1_31_4\rangle + |1_11_2,0_31_4\rangle~.
\end{align*}
\vspace{1pt}
\end{minipage}
\end{enumerate}
\end{theorem}
\vspace{-15pt}
\begin{proof} The proof of Theorem \ref{thm_cnot}.1 is by a Gr\"obner basis computation that runs in under a minute. For the second part, we present a ``Pfaffian certificate". Consider the following:
\vspace{-5pt}
\[
A = \left[\begin{array}{cc}
1 & 1\\[-1.5ex]
-\frac{1}{2} & \frac{1}{2}
\end{array}\right]~, \quad
B = \left[\begin{array}{cc}
0 & 1\\[-1.5ex]
-1 & 0
\end{array}\right]~, \quad
C = \left[\begin{array}{cc}
\frac{1}{2} & -\frac{\mathbf{i}}{2}\\[-1.5ex]
-\mathbf{i} & 1
\end{array}\right]~, \quad
D = \left[\begin{array}{cc}
\mathbf{i} & 1\\[-1.5ex]
-\frac{1}{2} & -\frac{\mathbf{i}}{2}
\end{array}\right]~.
\]
\vspace{-10pt}
\begin{align*}
(A\otimes B \otimes C \otimes D)\textsc{cnot1} &=  |0000\rangle + \frac{\mathbf{i}}{2}|1100\rangle -\mathbf{i}|1010\rangle -\frac{\mathbf{i}}{4}|1001\rangle -2|0110\rangle + \frac{1}{2}|0101\rangle\\
&\phantom{=}  -|0011\rangle + \frac{\mathbf{i}}{2}|1111\rangle~,\\
&= \text{sPf}\left(\left[ \begin{array}{cccc}
 0   &   \frac{\mathbf{i}}{2} &   -\mathbf{i}   & -\frac{\mathbf{i}}{4}\\[-1.5ex]
-\frac{\mathbf{i}}{2}  &    0  &    -2  &   \frac{1}{2}\\[-1.5ex]
 \mathbf{i} &        2   &   0  &     -1\\[-1.5ex]
\frac{\mathbf{i}}{4}  &   -\frac{1}{2} &    1    &   0   
\end{array}\right]\right)~.
\end{align*}  \hfill 
\vspace{-10pt}
%
\end{proof}

Theorem \ref{thm_cnot} demonstrates that the individual \textsc{cnot1} gate is Pfaffian. However, despite this hopeful combinatorial property, we are unable to chain together \textsc{cnot1-cnot2-cnot1} such that all three gates are simultaneously Pfaffian.
\begin{theorem}\label{thm_cnot_swap}
There do \textbf{not} exist matrices $A,\ldots, K \in \mathbb{C}^{2 \times 2}$ such that \textsc{cnot1-cnot2-cnot1} (linked by \textsc{equal} $\cgequal$  cogates) is Pfaffian under the change of basis $(A \otimes \cdots \otimes K)$.
\begin{center}
\includegraphics[scale=.20,page=1,clip=true,trim=0 570 0 85]{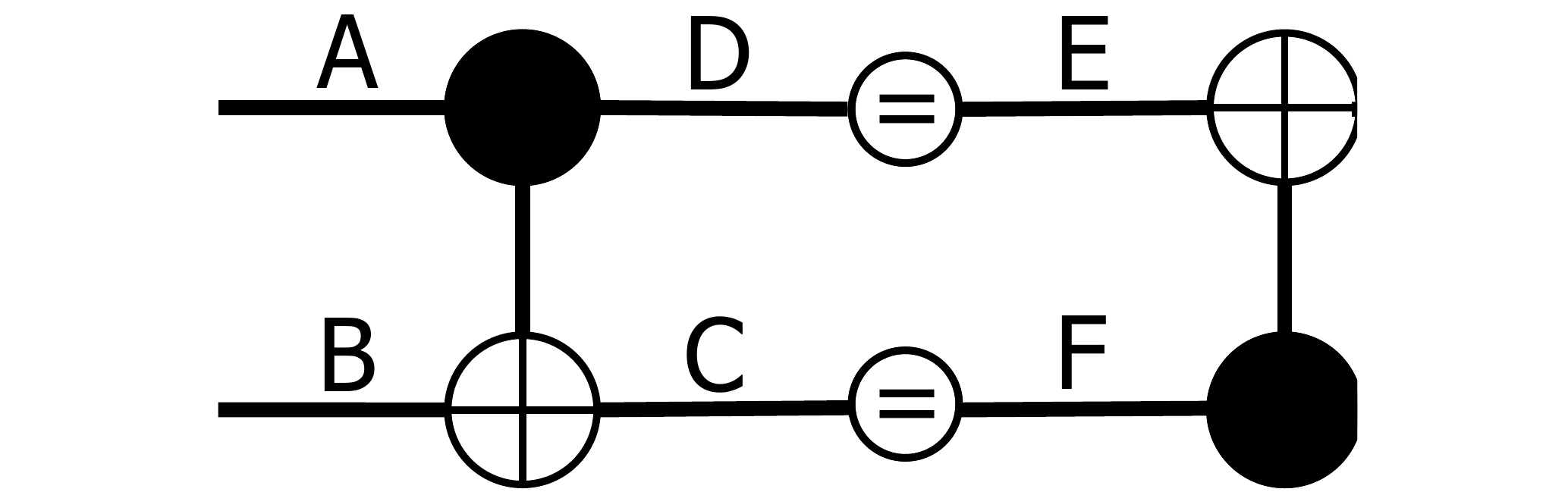}
\hspace{-36pt}\includegraphics[scale=.20,page=2,clip=true,trim=0 570 0 85]{CNOT1_CNOT2_CNOT1_het_A_to_K.pdf}
\end{center}
\end{theorem}
\vspace{-15pt}
\begin{proof} The Gr\"obner basis of the associated ideal is 1. This computation was extremely difficult to run, and many straightforward attempts ran for over 96 hours without terminating. We eventually succeeded by breaking the computation into three separate pieces. We paired \textsc{cnot1} with the corresponding \textsc{equal} cogates (utilizing changes of basis $A,\ldots,F$) and found a Gr\"obner bases consisting of 20,438 polynomials with degrees ranging from 2 to 12 after 2 hours and 44 minutes. Next, we paired \textsc{cnot3} with the corresponding \textsc{equal} cogates (utilizing changes of basis $G,\ldots,L$) and found a Gr\"obner bases consisting of 18,801 polynomials with degrees ranging from 2 to 12. Finally, we combined those two ideals with the ideal for \textsc{cnot2} (on changes of bases $E,\ldots,H$), and found a Gr\"obner basis of 1. This last computation ran in 5 hours and 55 minutes. All computations were run on the Cyberstar high-performance cluster, and allocated 2 GB of RAM.
%
%
%
\end{proof}

The construction demonstrated in Theorem \ref{thm_cnot_swap} is an example of a series of gates and cogates that mimic the behavior of \textsc{swap} when linked together. However, as we see from Theorem \ref{thm_cnot_swap}, there do not exist \emph{any} heterogeneous changes of bases such that each of the gates/cogates are simultaneously Pfaffian in this construction. However, not only is the chaining together of three \textsc{cnot} gates just one particular way of simulating \textsc{swap}, but the particular method of chaining the gates together is just one specific chaining construction. Thus, although \textsc{cnot1-cnot2-cnot1} (with \textsc{equal} $\cgequal$ cogates) is not Pfaffian, this algebraic observation does not shed definitive light on the question of whether it is possible to mimic \textsc{swap} using some other construction. We will now demonstrate that it is possible to mimic the behavior of three \textsc{cnot} gates chained together (and therefore, the behavior of \textsc{swap}) with only two 
additional ``dangling" edges. 

\begin{defi} \label{def_cnot12} Let \textsc{cnot12} denote the tensor representing the partial chain \textsc{cnot1-cnot2}.

\begin{minipage}{.45\linewidth}
\begin{center}
\hspace{-20pt}\includegraphics[scale=.25,page=1,clip=true,trim=0 650 0 80]{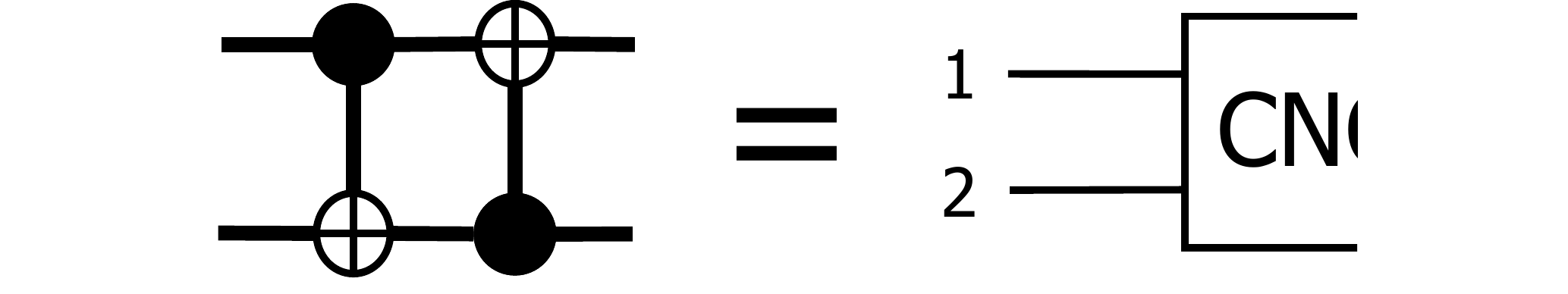}
\hspace{-25pt}\includegraphics[scale=.25,page=2,clip=true,trim=77 650 300 80]{CNOT12_g4.pdf}
\end{center}
\end{minipage}
\begin{minipage}{.45\linewidth}
\vspace{-10pt}
\begin{align*}
\textsc{cnot12} &= |0_10_2,0_30_4\rangle + |1_10_2,1_30_4\rangle\\
&\phantom{=}\hspace{9pt}  + |1_11_2,0_31_4\rangle + |0_11_2,1_31_4\rangle~.
\end{align*}
\end{minipage}
\end{defi}

\begin{theorem}\label{thm_cnot12_br}
There \textbf{do} exist matrices $A, B,\ldots, J \in \mathbb{C}^{2 \times 2}$ such that \textsc{cnot12-cnot1} (linked with two specific 3-arity \textsc{bridge} cogates, denoted $\cgtop$ and $\cgbot$, respectively) is Pfaffian under the heterogeneous change of basis $(A \otimes \cdots \otimes J)$.
\begin{center}
\hspace{80pt}\includegraphics[scale=.22,page=1,clip=true,trim=0 590 0 75]{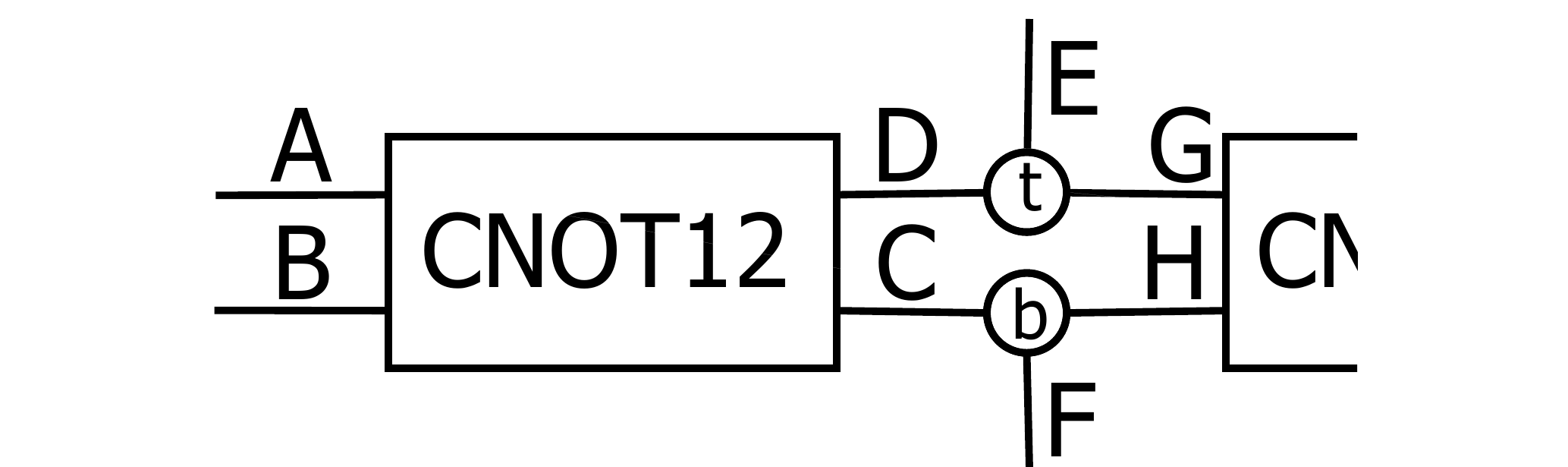}
\hspace{-39pt}\includegraphics[scale=.22,page=2,clip=true,trim=0 590 0 75]{CNOT12_1_br.pdf}
\end{center}
where $\cgtop$ and $\cgbot$ are the following cogates:

\begin{minipage}{.50\linewidth}
\begin{center}
\includegraphics[scale=.21,clip=true,trim=75 635 300 75]{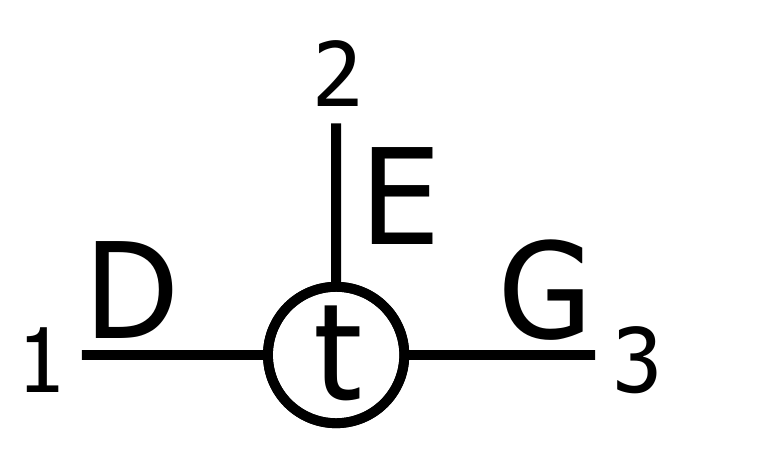}
\end{center}
\vspace{-7pt}
\begin{align*}
\cgtop &=\langle 1_10_20_3| + \langle 0_1\mathbf{1_2}0_3| + \langle 0_10_21_3| + \langle 1_1\mathbf{1_2}1_3|~,
\end{align*}
\end{minipage}
\begin{minipage}{.50\linewidth}
\begin{center}
\includegraphics[scale=.22,clip=true,trim=75 635 300 80]{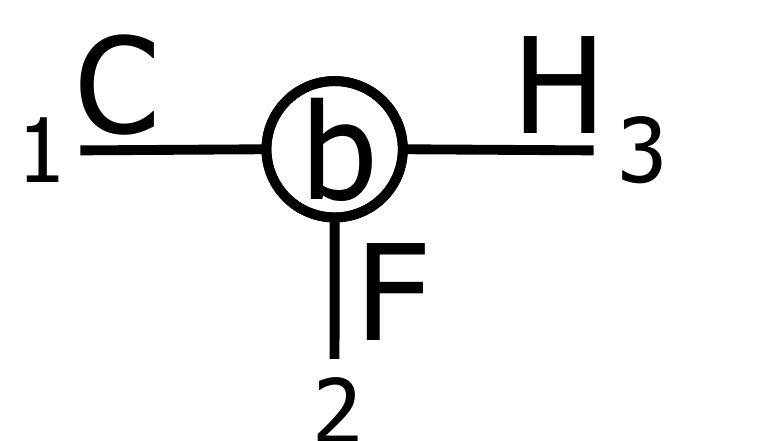}
\end{center}
\vspace{-7pt}
\begin{align*}
\cgbot &= \langle 0_1\mathbf{0_2}0_3| + \langle 0_11_20_3| + \langle 1_1\mathbf{0_2}1_3|~.
\end{align*}
\end{minipage}
\end{theorem}

Observe that if the second wire of $\cgtop$ is fixed to state $\langle 1|$, and the second wire of $\cgbot$  is fixed to state $\langle 0|$, then both $\cgtop$ and $\cgbot$ act exactly like the \textsc{equal} $\cgequal$ cogate.

\begin{proof}
Consider matrices $A,\ldots,J \in \mathbb{C}^{2 \times 2}$ displayed  below:
\[
A = \left[\begin{array}{cc}
1 & 1\\[-1.5ex]
-\frac{1}{2} & \frac{1}{2}
\end{array}\right], \hspace{3pt}
B = \left[\begin{array}{cc}
0 & 1\\[-1.5ex]
-1 & 0
\end{array}\right], \hspace{3pt}
C = \left[\begin{array}{cc}
1 & 1\\[-1.5ex]
-\frac{1}{2} & \frac{1}{2}
\end{array}\right], \hspace{3pt}
D = \left[\begin{array}{cc}
0 & \frac{1}{2}\\[-1.5ex]
-2 & 0
\end{array}\right], \hspace{3pt}
E = \left[\begin{array}{cc}
-2 & 2\mathbf{i}\\[-1.5ex]
\frac{\mathbf{i}}{4} & -\frac{1}{4}
\end{array}\right],
\] \[
F = \left[\begin{array}{cc}
\frac{-1+\mathbf{i}}{2} & -\frac{1}{2}\\[-1.5ex]
1-\mathbf{i} & -\mathbf{i}
\end{array}\right], \hspace{3pt}
G = \left[\begin{array}{cc}
-\frac{\mathbf{i}}{2} &  -\frac{1}{2}\\[-1.5ex]
1 & \mathbf{i}
\end{array}\right], \hspace{3pt}
H = \left[\begin{array}{cc}
1 & -\mathbf{i}\\[-1.5ex]
-\frac{\mathbf{i}}{2} & \frac{1}{2}
\end{array}\right], \hspace{3pt}
I = \left[\begin{array}{cc}
\mathbf{i} &0\\[-1.5ex]
0 & -\mathbf{i}
\end{array}\right], \hspace{3pt}
J = \left[\begin{array}{cc}
1 & -1\\[-1.5ex]
\frac{1}{2} & \frac{1}{2}
\end{array}\right]~.
\]
We must show that gates $(A \otimes B \otimes C \otimes D)\textsc{cnot12}, (G \otimes H \otimes I \otimes J)\textsc{cnot1}$ and cogates $(D \otimes E \otimes G)\cgtop, (C \otimes F \otimes H)\cgbot$ are Pfaffian. Consider the following ``Pfaffian certificates":
\begin{align*}
(A \otimes B \otimes C \otimes D)\textsc{cnot12} &= |0000\rangle -\frac{1}{4}|1010\rangle + 4|0101\rangle + |1111\rangle~
= \text{sPf}\left(\left[ {\footnotesize\begin{array}{cccc}
 0   &   0 &   -\frac{1}{4}  & 0\\[-1.5ex]
0  &    0  &   0  &  4\\[-1.5ex]
\frac{1}{4} & 0   &   0  &   0\\[-1.5ex]
0  & -4 & 0   &   0   
\end{array}}\right]\right)~,\\
(G \otimes H \otimes I \otimes J)\textsc{cnot1} &= |0000\rangle + |1100\rangle -2|1010\rangle + \mathbf{i}|1001\rangle  -\frac{1}{2}|0110\rangle\\
&\hspace{28pt} -\frac{\mathbf{i}}{4}|0101\rangle
+\frac{\mathbf{i}}{2}|0011\rangle - \frac{\mathbf{i}}{2}|1111\rangle~
= \text{sPf}\left(\left[ \begin{array}{cccc}
 0   &   1 &  -2 & \mathbf{i}\\[-1.5ex]
-1  &    0  &  - \frac{1}{2}  &   -\frac{\mathbf{i}}{4}\\[-1.5ex]
2 &  \frac{1}{2}   &   0  &   \frac{\mathbf{i}}{2}\\[-1.5ex]
-\mathbf{i}  & \frac{\mathbf{i}}{4} & -\frac{\mathbf{i}}{2}   &   0   
\end{array}\right]\right)~,
\end{align*} \vspace{-17pt}
\begin{align*}
(D \otimes E \otimes G)\big(\langle 100| + \langle 010| + \langle 001| + \langle 111|\big) &=
 - \frac{1}{4}\langle 100| + 8\langle 010|\\
&\hspace{25pt}  - \frac{1}{2}\langle 001| + \langle 111|~
= \text{sPf}^{*}\left(\left[ \begin{array}{ccc}
 0  &  - \frac{1}{2} & 8\\[-1.5ex]
\frac{1}{2} & 0 & - \frac{1}{4}\\[-1.5ex]
-8  &  \frac{1}{4} & 0
\end{array} \right] \right)~,
\end{align*}
\begin{align*}
(C \otimes F \otimes H)\big(\langle 000| + \langle 010| + \langle 101| \big)&= \langle 100| +  \frac{\mathbf{i}}{4}\langle 010| - \mathbf{i}\langle 001| + \langle 111|~
= \text{sPf}^{*}\left(\left[ {\small \begin{array}{ccc}
 0  &  -\mathbf{i} & \frac{\mathbf{i}}{4}\\[-1.5ex]
\mathbf{i}  &  0 & 1\\[-1.5ex]
-\frac{\mathbf{i}}{4} &-1 & 0
\end{array}} \right] \right)~.
\end{align*} \hfill
\vspace{-5pt}
\end{proof}

We observe that the construction outlined in Theorem \ref{thm_cnot12_br} is a model of \textsc{swap} \emph{only} if the dangling edges from cogates $\cgtop$ and $\cgbot$ can be fixed to $\langle 1|$ and $\langle 0 |$, respectively. However, this partial result reduces the question of modeling \textsc{swap} \textbf{from} the more complicated question of mimicking a wire crossing as a bipartite, Pfaffian tensor contraction network fragment \textbf{to} the simplified question of assigning a constant state to a given wire. As example of the myriad of ways of fixing a constant state to a given wire, we present the following example.

\begin{example} \label{ex_const} Consider the following tensor contraction network fragment:

\begin{minipage}{.20\linewidth}
\begin{center}
\includegraphics[scale=.22,page=1,clip=true,trim=80 430 250 75]{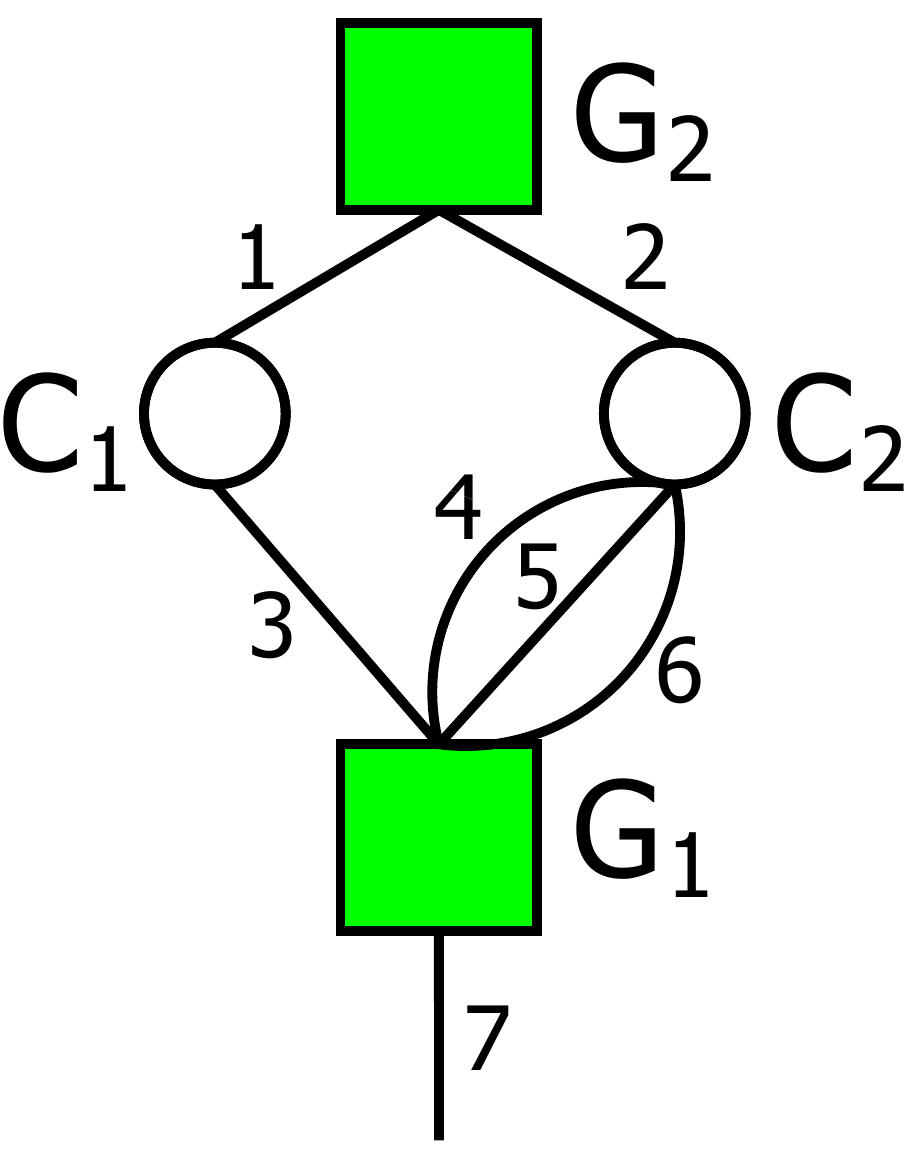}
\end{center}
\end{minipage}
\begin{minipage}{.15\linewidth}
where
\end{minipage}
\begin{minipage}{.55\linewidth}
\vspace{-13pt}
\begin{align*}
G_1 &= |1_71_30_41_51_6 \rangle  + |0_70_31_41_51_6 \rangle + |0_71_30_41_51_6 \rangle~,\\
G_2 &= |0_10_2\rangle + |1_11_2\rangle~,\\
C_1 &= \langle 0_11_3| + \langle 1_10_3|~,\\
C_2 &= \langle 0_2 0_41_51_6| + \langle 1_21_40_51_6| + \langle 1_2 1_41_50_6|~.
\end{align*}
\end{minipage}

Observe that $\langle C_2|G_1\rangle = \langle 0_2|1_71_3 \rangle~,
\big\langle \langle 0_2|1_71_3 \rangle  \big| G_2 \big \rangle = |0_11_71_3 \rangle~,$ and $C_1 |0_11_71_3 \rangle = 1_7~.$ Therefore, this tensor contraction network fragment fixes edge 7 in state $|1\rangle$. There are infinitely many such constructions, when considering gates/cogates of unbounded arity. However, the primary question is whether there exists a construction that fixes edge 7 in state $|1\rangle$ such that all gates/cogates are simultaneously Pfaffian. This example (which is not Pfaffian) highlights the computational difficulty of finding such a construction, or proving that such a construction does not exist. \hfill $\Box$
\end{example}

It is very important to observe that there is no contradiction between widely held complexity class assumptions (such as $\#\text{P} \neq \text{P}$), and the definitive existence of a Pfaffian \textsc{swap}. Even if a Pfaffian \textsc{swap} could be explicitly constructed, the question of which particular gates/cogates were simultaneously Pfaffian (and therefore, precisely which non-planar \#CSP problems were then solvable in polynomial-time) would still remain an open question. For example, if the only gates/cogates compatible with the Pfaffian \textsc{swap} were tensors modeling non-planar, \textbf{non}-\#P-complete \#CSP problems, there would be no collapse of \#P into P. 


	To conclude, we consider the problem of modeling \textsc{swap} as a Pfaffian tensor contraction network fragment under a heterogeneous change of basis to be one of the most interesting open problems in this area. Not only does overcoming the barrier of planarity for a certain class of problems have wide-spread appeal, but constructing a Pfaffian \textsc{swap} would also provide an answer to another, equally interesting mathematical question: given an arbitrary tensor that is \emph{not} Pfaffian, is it possible to design an equivalent tensor with gates/cogates that are simultaneously Pfaffian under some heterogeneous change of basis?


\section{Decomposable Gates/Cogates}  \label{sec_decomp} In the previous section, we demonstrated that the \textsc{swap} gate was not Pfaffian, and also explored various directions for designing a Pfaffian tensor contraction network fragment with input/output behavior equivalent to \textsc{swap}. In Sec. \ref{sec_bool_ten}, we designed a Pfaffian tensor contraction network fragment called a \emph{Boolean tree}, which has behavior equivalent to the Pfaffian $n$-arity \textsc{equal} gate/cogate. The question of whether the behavior of a larger arity gate/cogate can be modeled by a construction of lower arity gates/cogates, or whether a non-Pfaffian gate/cogate can be modeled by a tensor contraction network fragment \emph{at all} is a question of both computational and theoretical interest. In this section, we formalize the notion of a \emph{decomposable gate} (or cogate), and provide an explicit example.

\begin{defi} \label{def_decomp_gcg} Let $G$ be an $n$-arity Pfaffian gate such that $(A_1\otimes \cdots \otimes A_{n})G = \alpha_G\hspace{1pt}\text{\emph{sPf}}(\Xi_G)$. Then $G$ is \emph{decomposable} if there exists a planar, bipartite, Pfaffian tensor contraction network fragment with spanning tree edge order  $\sigma$ constructed from gates $G_1,\ldots,G_i$ and cogates $C_1,\ldots, C_j$ (with $\Xi_{\sigma} =  \Xi_{1} \oplus_{\sigma} \cdots \oplus_{\sigma} \Xi_{i}, \alpha_{\sigma} = \alpha_{1}\cdots \alpha_{i}, \Theta_{\sigma} =  \Theta_{1} \oplus_{\sigma} \cdots \oplus_{\sigma} \Theta_{j}$ and $\beta_{\sigma} = \beta_{1}\cdots \beta_{j}$) such that
\begin{align*}
(A_1\otimes \cdots \otimes A_{n})G = \alpha_G\hspace{1pt}\text{\emph{sPf}}(\Xi_G) = \Big\langle \beta_{\sigma}\hspace{1pt} \text{\emph{sPf}}^{*}(\widetilde{\Theta}_{\sigma})~|~ \alpha_{\sigma}\hspace{1pt}\text{\emph{sPf}} (\widetilde{\Xi}_{\sigma} ) \Big\rangle~,
\end{align*}
where $\widetilde{\Theta},\widetilde{\Xi}$ are defined in Theorem \ref{thm_pfaff_eval}.
\end{defi}

We comment that we do not force the gates $G_1,\ldots, G_i$ and cogates $C_1,\ldots, C_j$ to be lower-arity then gate $G$. Observe that the integer 210 can be equivalently expressed as $7\cdot 5 \cdot 3 \cdot 2$ or $35 \cdot 6$ or $7770/37$ or, of course, $210\cdot 1$. These expressions of the integer 210 are obviously of varying degrees of interest, depending on the circumstances. However, since it is known that every 3-arity gate/cogate is Pfaffian under some heterogeneous change of basis (\cite{morton_pfaff}), then 5-arity gates/cogates are the first odd-arity non-trivial case. Therefore, we do not wish to exclude the possibility of ``decomposing" a 4-arity gate (such as \textsc{swap}) into a construction that includes 5-arity gates/cogates. Previously in Ex. \ref{ex_const}, we highlighted how a 5-arity gate could be used in constructing a Pfaffian \textsc{swap}. We will now demonstrate a concrete example of a decomposable gate.


\begin{example} \label{ex_decomp} Let
\[
A = \left[\begin{array}{cc}
\mathbf{i}\big(\frac{2^{3/4}}{2}\big) & \mathbf{i}\big(\frac{2^{3/4}}{2}\big)\\
\mathbf{i}\big(\frac{2^{1/4}}{2}\big) & -\mathbf{i}\big(\frac{2^{1/4}}{2}\big)
\end{array}\right]~, \quad
B = \left[\begin{array}{cc}
-\big(\frac{2^{3/4}}{2}\big) & \frac{2^{3/4}}{2}\\
-\big(\frac{2^{1/4}}{2}\big) & -\big(\frac{2^{1/4}}{2}\big)
\end{array}\right]~, \quad
C = \left[\begin{array}{cc}
-\big(\frac{2^{3/4}}{2}\big)  & -\mathbf{i}\big(\frac{2^{3/4}}{2}\big) \\
-\mathbf{i}\big(\frac{2^{1/4}}{2}\big) & -\big(\frac{2^{1/4}}{2}\big) 
\end{array}\right]~,
\]
\[
D = \left[\begin{array}{cc}
\frac{2^{3/4}}{2} & -\mathbf{i}\big(\frac{2^{3/4}}{2}\big)\\
-\mathbf{i}\big(\frac{2^{1/4}}{2}\big)  & \frac{2^{1/4}}{2}
\end{array}\right]~, \quad
E = \left[\begin{array}{cc}
\frac{2^{3/4}}{2}  & -\mathbf{i}\big(\frac{2^{3/4}}{2}\big)\\
-\mathbf{i}\big(\frac{2^{1/4}}{2}\big)& \frac{2^{1/4}}{2}
\end{array}\right]~, \quad
F = \left[\begin{array}{cc}
-\big(\frac{2^{3/4}}{2}\big) & -\mathbf{i}\big(\frac{2^{3/4}}{2}\big)\\
-\mathbf{i}\big(\frac{2^{1/4}}{2}\big) & -\big(\frac{2^{1/4}}{2}\big)
\end{array}\right]~.
\]
We claim
\begin{minipage}{.25\linewidth}
\begin{center}
\includegraphics[scale=.22,page=1,clip=true,trim=0 550 350 120]{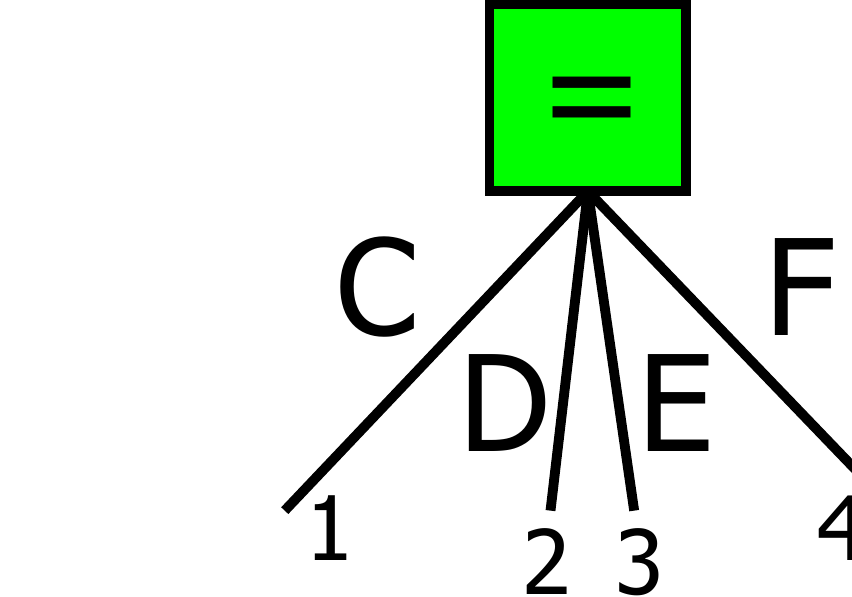}
\end{center}
\end{minipage}
\begin{minipage}{.25\linewidth}
\vspace{-10pt}
\begin{center}
is decomposable into 
\end{center}
\end{minipage}
\begin{minipage}{.43\linewidth}
\begin{center}
\includegraphics[scale=.22,page=1,clip=true,trim=400 510 75 75]{bool_tree_het_decomp_lbl_1234.pdf}
\hspace{-6.5pt}\includegraphics[scale=.22,page=2,clip=true,trim=75 510 0 80]{bool_tree_het_decomp_lbl_1234.pdf}
\end{center}
\end{minipage}

\noindent In order to prove this claim, let $\Xi_{=}$ be such that $(A\otimes B \otimes B)\big(|000\rangle + |111\rangle\big) = \alpha\hspace{1pt}\text{\emph{sPf}}(\Xi_{=})$ and $\Theta_{=}$ be such that $(A^{-1}\otimes A^{-1})\big(\langle 00| + \langle 11|\big) = \beta \hspace{1pt} \text{\emph{sPf}}^{*}(\Theta_{=})$. We must 1) choose a spanning tree order $\sigma$ of the 4-arity Boolean tree fragment (above right), and 2) construct the matrices $\widetilde{\Xi} = \Xi_{=} \oplus_{\sigma} \Xi_{=}$ and $\widetilde{\Xi}_{\sigma}$ (note that $\widetilde{\Theta}, \widetilde{\Xi}$ are defined in Theorem \ref{thm_pfaff_eval} and note that $\widetilde{\Theta} = \Theta_{=})$, and 3) demonstrate
\begin{align*}
 (C\otimes D \otimes E \otimes F) (|0000\rangle + |1111\rangle) &= \beta \alpha^2 \big \langle   \hspace{1pt} \text{\emph{sPf}}^{*}(\widetilde{\Theta}_{=}) | \text{\emph{sPf}} \big(\widetilde{\Xi}_{\sigma}\big) \big \rangle ~.
\end{align*}
We will demonstrate each of these steps below. First, we choose the following arbitrary spanning tree order $\sigma = \{5,2,1,4,3,6\}$. Observe that the 2-arity \textsc{equal} cogate is in the same order as $\sigma$ ($\{5,6\}$), but both 3-arity \textsc{equal} gates are in the opposite order (for example, the left gate is in counter-clockwise order $\{5,1,2\}$, but the order appears in $\sigma$ as $\{5,2,1\}$):
\begin{center}
\includegraphics[scale=.22,clip=true,trim=81 450 75 82]{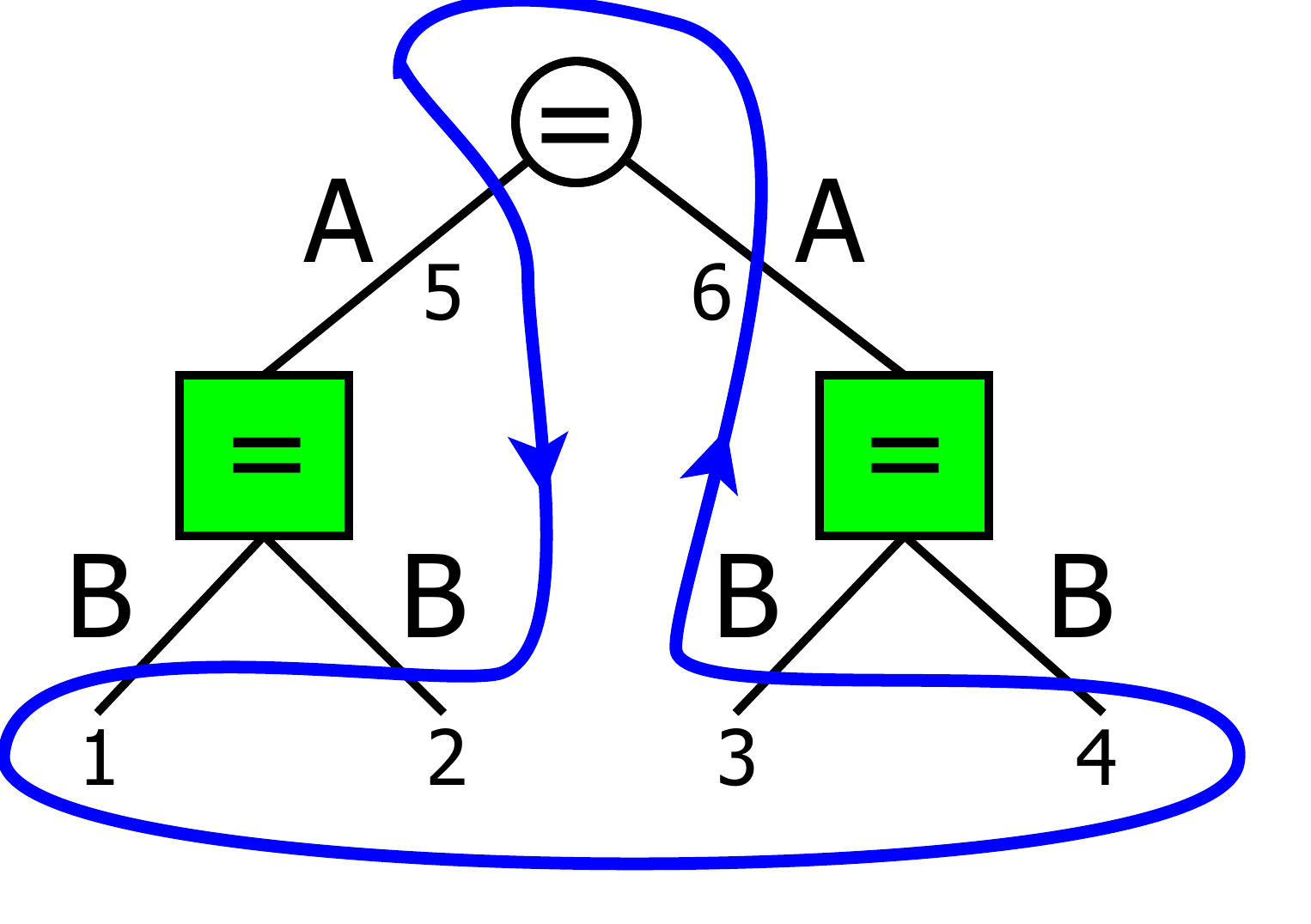}
\end{center}
Next, we present the following \emph{labeled} ``Pfaffian certificates".

\vspace{-15pt}
\begin{minipage}{.75\linewidth}
\begin{align*}
\hspace{44pt}\big(A^{-1} \otimes A^{-1}\big)\big(\langle 0_50_6| + \langle 1_51_6|\big) &= \underbrace{-\big(\sqrt{2}\big)}_{\beta}\hspace{1pt}\text{\emph{sPf}}^*
\overbrace{\left(
\bordermatrix{
&5&6\cr
5& 0 & \frac{1}{2} \cr
6& - \frac{1}{2}  & 0\
}\right)}^{\Theta_{=}}~,
\end{align*}
\end{minipage}
\begin{minipage}{.15\linewidth}
\begin{center}
\includegraphics[scale=.25,trim=80 635 375 85]{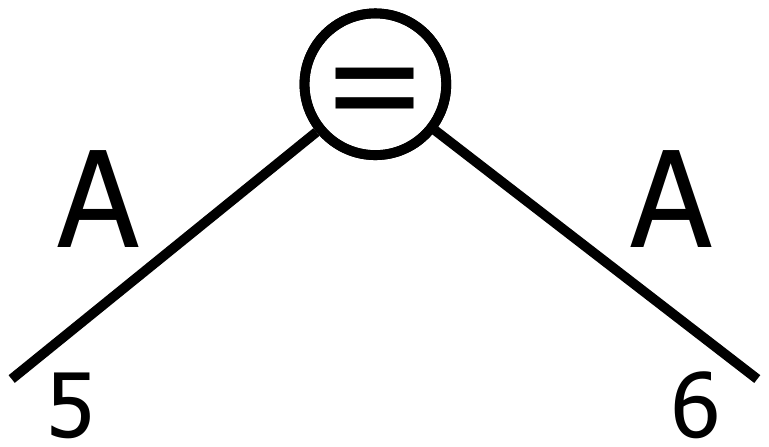}
\end{center}
\end{minipage}

\vspace{-10pt}
\begin{minipage}{.75\linewidth}
\begin{align*}
\hspace{26pt}(A \otimes B \otimes B)\big(|0_50_10_2\rangle + |1_51_11_2\rangle\big) &= \underbrace{\mathbf{i}\big(2^{1/4}\big)}_{\alpha}\hspace{1pt}\text{\emph{sPf}}
\overbrace{\left(
\bordermatrix{
&5&1&2\cr
5& 0 & \frac{1}{2}  & \frac{1}{2} \cr
1& - \frac{1}{2}  & 0 &  \frac{1}{2} \cr
2& - \frac{1}{2}   & - \frac{1}{2}   & 0
}\right)}^{\Xi_{=}}~,
\end{align*}
\end{minipage}
\begin{minipage}{.15\linewidth}
\begin{center}
\includegraphics[scale=.25,trim=80 570 340 30]{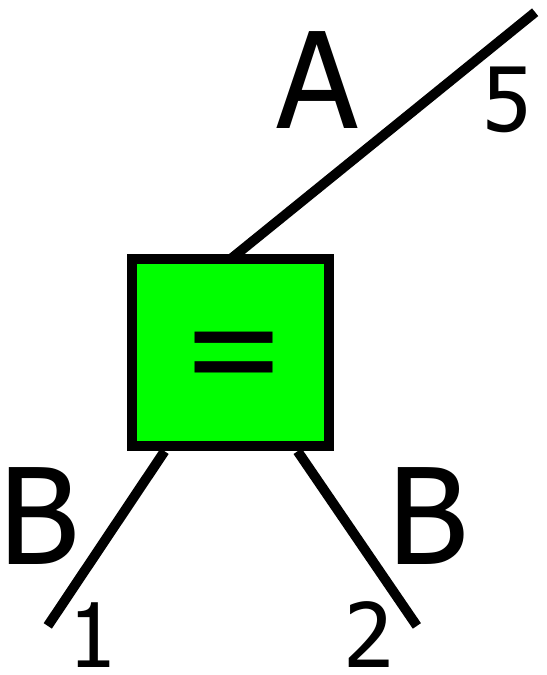}
\end{center}
\end{minipage}

\begin{minipage}{.75\linewidth}
\begin{align*}
\hspace{-13pt}\big(C \otimes D \otimes E \otimes F\big)\big(|0_10_20_30_4\rangle + |1_11_21_31_4\rangle\big) &= \text{\emph{sPf}} \left(
\bordermatrix{
&1&2&3&4\cr
1 & 0 &\frac{1}{2} & \frac{1}{2} & -\frac{1}{2}\cr
2 & -\frac{1}{2} & 0 & -\frac{1}{2} & \frac{1}{2}\cr
3 & -\frac{1}{2} & \frac{1}{2} & 0 & \frac{1}{2}\cr
4 &\frac{1}{2} & -\frac{1}{2} & -\frac{1}{2} & 0
}
\right)~.
\end{align*}
\end{minipage}
\begin{minipage}{.15\linewidth}
\begin{center}
\hspace{-20pt}\includegraphics[scale=.25,page=1,clip=true,trim=85 550 350 120]{bool_tree_het_decomp_lbl_1234.pdf}
\end{center}
\end{minipage}

We will now construct $\Xi_{\sigma} = \Xi_{=} \oplus_{\sigma} \Xi_{=}$ and $\widetilde{\Xi}_{\sigma}$. Recall that $\sigma = \{5,2,1,6,4,3\}$, but the two gates are in orders $\{5,1,2\}$ and $\{6,3,4\}$, respectively. Therefore, the wires of the gates appear in opposite order in $\sigma$, and thus, when the matrix $\widetilde{\Xi}_{\sigma}$ is calculated (via the definition given in Theorem \ref{thm_pfaff_eval}), the entries in both $\Xi_{=}$ blocks are flipped according to the $(-1)^{i+j+1}\xi_{ij}$ ``checkerboard" pattern:
{\footnotesize{
\[
\Xi_{\sigma} = 
\bordermatrix{
&5&1&2&6&3&4\cr
5& 0 & \frac{1}{2} & \frac{1}{2} & 0 & 0 & 0\cr
1& -\frac{1}{2}& 0 & \frac{1}{2} & 0 & 0 &0\cr
2& -\frac{1}{2} & -\frac{1}{2} & 0 & 0 & 0 & 0\cr
6& 0 & 0 & 0 & 0 & \frac{1}{2} & \frac{1}{2}\cr
3& 0 & 0 & 0 & -\frac{1}{2} & 0 & \frac{1}{2}\cr
4& 0 & 0 & 0 & -\frac{1}{2} & -\frac{1}{2} & 0
}~, \quad \widetilde{\Xi}_{\sigma} = 
\bordermatrix{
&5&1&2&6&3&4\cr
5& 0 & \frac{1}{2} & -\frac{1}{2} & 0 & 0 & 0\cr
1& -\frac{1}{2}& 0 & \frac{1}{2} & 0 & 0 &0\cr
2& \frac{1}{2} & -\frac{1}{2} & 0 & 0 & 0 & 0\cr
6& 0 & 0 & 0 & 0 & \frac{1}{2} & -\frac{1}{2}\cr
3& 0 & 0 & 0 & -\frac{1}{2} & 0 & \frac{1}{2}\cr
4& 0 & 0 & 0 & \frac{1}{2} & -\frac{1}{2} & 0
}~,
\]}} 
Now we compute $\big \langle   \hspace{1pt} \text{\emph{sPf}}^{*}(\widetilde{\Theta}_{=}) | \text{\emph{sPf}} \big(\widetilde{\Xi}_{\sigma}\big) \big \rangle$. For notational convenience, we display the wire labels of the first and last bra/kets only. Observe that
{\footnotesize{\begin{align*}
 \text{sPf} \big(\widetilde{\Xi}_{\sigma}\big) &= 
|0_50_10_20_60_30_4\rangle + \frac{1}{2}|110000\rangle -\frac{1}{2}|101000\rangle + \frac{1}{2}|011000\rangle + \frac{1}{2}|000110\rangle -\frac{1}{2}|000101\rangle\\
&\phantom{=}\hspace{9pt} + \frac{1}{2}|000011\rangle + \frac{1}{4}|110110\rangle -\frac{1}{4}|110101\rangle +\frac{1}{4}|110011\rangle-\frac{1}{4}|101110\rangle + \frac{1}{4}|101101\rangle \\
&\phantom{=}\hspace{9pt} -\frac{1}{4}|101011\rangle +\frac{1}{4}|011110\rangle -\frac{1}{4}|011101\rangle + \frac{1}{4}|0_51_11_20_61_31_4\rangle~,
%
\end{align*}}}
\vspace{-20pt}
\begin{align*}
\text{sPf}^{*}(\widetilde{\Theta}_{=}) &=\frac{1}{2} \langle 0_50_6| + \langle 1_51_6|~.
\end{align*}
Finally, 
\begin{align*}
 \big \langle   \hspace{1pt} \text{sPf}^{*}(\widetilde{\Theta}_{=})~|~\text{sPf} \big(\widetilde{\Xi}_{\sigma}\big) \big \rangle &= \frac{1}{2} \bigg( |0_10_20_30_4\rangle + \frac{1}{2}|1100\rangle + \frac{1}{2}|0011\rangle + \frac{1}{4}|1111\rangle\\
&\phantom{=}\hspace{12pt}  +\frac{1}{2}|1010\rangle - \frac{1}{2}|1001\rangle - \frac{1}{2}|0110\rangle +\frac{1}{2}|0_11_20_31_4\rangle\bigg)~,\\
%
\beta \alpha^2  \big \langle   \hspace{1pt} \text{sPf}^{*}(\widetilde{\Theta}_{=})~|~\text{sPf} \big(\widetilde{\Xi}_{\sigma}\big) \big \rangle &= -\big(\sqrt{2}\big) \Big(\mathbf{i}\big(2^{1/4}\big)\Big)^2 \Big(\frac{1}{2}(\cdots)\Big) = \text{sPf} \left(
\bordermatrix{
&1&2&3&4\cr
1 & 0 &\frac{1}{2} & \frac{1}{2} & -\frac{1}{2}\cr
2 & -\frac{1}{2} & 0 & -\frac{1}{2} & \frac{1}{2}\cr
3 & -\frac{1}{2} & \frac{1}{2} & 0 & \frac{1}{2}\cr
4 &\frac{1}{2} & -\frac{1}{2} & -\frac{1}{2} & 0
}
\right)~.\\
&= \big(C \otimes D \otimes E \otimes F\big)\big(|0_10_20_30_4\rangle + |1_11_21_31_4\rangle\big)~. 
\end{align*}
 \hfill $\Box$
\vspace{-10pt}
%
%
%
%
\end{example}

Ex. \ref{ex_decomp} above demonstrates the \emph{existence} of decomposable gates. This is important, since there are gates with arity as low as 4 where the Pfaffian property has not been determined. However, pairing gates with cogates often improves the computation time, since the lower degree ideals can act as ``cuts". For example, the following 4-arity gate:
\begin{align*}
|0010\rangle + |1100\rangle + |1001\rangle + |0110\rangle + |0101\rangle + |1110\rangle + |1011\rangle + |1111\rangle
\end{align*}
has an ideal that contains 19 polynomials of degrees ranging from 5 to 2, and the Gr\"obner basis algorithm fails to terminate after 96 hours of computation. However,  the Gr\"obner basis for the cogate $\langle 01|$ contains polynomials of degree one, and when this gate is paired with this cogate (which turns this pair into a 2-arity \textsc{equal} gate $|00\rangle + |11\rangle$), the computation terminates within seconds, indicating that this pair is \emph{not} Pfaffian. Thus, it may be much faster to check the Pfaffian property of a decomposition, then it would be to determine if a higher arity gate is Pfaffian.

We conclude this section by commenting on an interesting application of decomposable circuits. The problem of multiplying two integers together is often solved by a planar combinatorial multiplier. It is very easy to convert this multiplier to a planar tensor contraction network, and thus the question of factoring an $n$-bit integer is equivalent to solving $n$ \#CSP problems. However, the arity of the gate/cogate representing the ``adder" or ``half-adder" is eight, and thus the computation to determine if the gate/cogate is Pfaffian is not tractable with current Gr\"obner bases algorithms. However, it is very easy to ``guess and check" low-degree decompositions of this 8-arity gate/cogate. Since every new possibility for a polynomial-time factoring algorithm must be explored, Pfaffian decompositions of planar combinatorial multipliers must be an area of future work.  

\section{Conclusion} In this paper, we modeled Pfaffian gates/cogates as systems of polynomial equations, and applied algebraic computational methods to identify classes of 0/1 \#CSP problems that are solvable in polynomial-time. We discussed two different methods of modeling 0/1 variables, and identified one 1-arity gate, three 2-arity gates, 15 3-arity gates and 117 4-arity gates (and similar cogates) that describe the available types of constraints. We also open the possibility of 24 extra gates and cogates that are possible under a heterogeneous change of basis only, if we postulate the existence of a third matrix $C$ and a ``bridge" between $C$ and the previously used matrices $A$ and $B$. We also demonstrated that the gates/cogates yielded differing set of constraints in this scenario. We additionally discussed the barrier of planarity in terms of Pfaffian circuits, and constructed a partial combinatorial structure of gates/cogates that models a Pfaffian \textsc{swap} gate/cogate. We highlight again 
that this structure is not yet complete. However, we have identified a specific direction in algebraic computation for simulating a \textsc{swap} gate/cogate in polynomial-time. We also introduce the notion of a \emph{decomposable gate/cogate} and discuss the benefits of using decompositions, as opposed to higher-arity gates, in computation.

This project has indicated a wealth of questions for exploration, and we summarize a few of the more compelling questions here for future work.
\begin{enumerate}
	\item If a given $n$-arity gate is Pfaffian, is the corresponding $n$-arity cogate also Pfaffian (under some heterogeneous change of basis), regardless whether $n$ is even or odd?
	\item If a given gate/cogate is Pfaffian, is there always a lower-arity decomposition?
	\item If a given gate/cogate is \emph{not} Pfaffian (such as \textsc{swap}), does there exist a decomposition that is Pfaffian?
	\item What is the importance of gates/cogates that are invariant under complement? How does such a combinatorial property factor into the development of a dichotomy theory for Pfaffian circuits?
	\item Can we develop theoretical proofs in place of computational ones?
\end{enumerate}

\section*{Acknowledgments}
The authors would like to acknowledge the support of the Defense Advanced Research Projects Agency under Award No. N66001-10-1-4040.  We would also like to thank Tyson Williams and William Whistler for helpful comments.

\bibliographystyle{plain}
\bibliography{quantum}

\begin{thebibliography}{10}

\bibitem{alonsurvey}
N.~Alon.
\newblock {C}ombinatorial {N}ullstellensatz.
\newblock {\em Combinatorics, Probability and Computing}, 8:7--29, 1999.

\bibitem{alontarsi}
N.~Alon and M.~Tarsi.
\newblock Colorings and orientations of graphs.
\newblock {\em Combinatorica}, 12:125--134, 1992.

\bibitem{cite_newick}
J.~Archie, W.~H.~E. Day, J.~Felsenstein, W.~Maddison, C.~Meacham, F.~J. Rohlf,
  and D.~Swofford.
\newblock Newick tree format.
\newblock 1986.

\bibitem{bulatov_dic}
A.A. Bulatov.
\newblock The complexity of the counting constraint satisfaction problem.

\bibitem{bulatov2013complexity}
A.A. Bulatov.
\newblock The complexity of the counting constraint satisfaction problem.
\newblock {\em Journal of the ACM (JACM)}, 60(5):34, 2013.

\bibitem{bulatov_dalmau_maltsev}
A.A. Bulatov and V.~Dalmau.
\newblock A simple algorithm for {M}al'tsev constraints.
\newblock {\em SIAM Journal of Computing}, 36(1):16--27, 2006.

\bibitem{bulatov_dalmau_towdic}
A.A. Bulatov and V.~Dalmau.
\newblock Towards a dichotomy theorem for the counting constraint satisfaction
  problem.
\newblock {\em Information and Computation}, 205(5):651--678, 2007.

\bibitem{bulatov_grohe}
A.A. Bulatov and M.~Grohe.
\newblock The complexity of partition functions.
\newblock {\em Theoretical Computer Science}, 348(2--3):148--186, 2005.

\bibitem{bulatov_fa}
A.A. Bulatov, P.~Jeavons, and A.~Krokhin.
\newblock Classifying the complexity of constraints using finite algebras.
\newblock {\em SIAM Journal of Computing}, 34(3):720--742, 2005.

\bibitem{cai_mg_ten}
J.-Y. Cai and V.~Choudhary.
\newblock Valiant's holant theorem and matchgate tensors.
\newblock {\em Lecture Notes in Computer Science}, 3959:248--261, 2006.

\bibitem{cai_lu_xia:dichotomy}
J.-Y. Cai, P.~Lu, and M.~Xia.
\newblock Dichotomy for holant* problems of boolean domain.
\newblock In {\em SODA'11}, pages 1714--1728, 2011.

\bibitem{cai2012complexity}
Jin-Yi Cai and Xi~Chen.
\newblock Complexity of counting csp with complex weights.
\newblock In {\em Proc. of the 44th symposium on Theory of Computing}, pages
  909--920. ACM, 2012.

\bibitem{dyer2010effective}
Martin Dyer and David Richerby.
\newblock An effective dichotomy for the counting constraint satisfaction
  problem.
\newblock {\em arXiv preprint arXiv:1003.3879}, 2011.

\bibitem{lands_hol_wo_mg}
JM~Landsberg, J.~Morton, and S.~Norine.
\newblock Holographic algorithms without matchgates.
\newblock {\em Linear Algebra and its Applications}, 438(2):782--795, 2013.

\bibitem{susan_cpc}
J.A.~De Loera, J.~Lee, S.~Margulies, and S.~Onn.
\newblock Expressing combinatorial optimization problems by systems of
  polynomial equations and the {N}ullstellensatz.
\newblock {\em Combinatorics, Probability and Computing}, 18(4):551--582, 2009.

\bibitem{lovasz1}
L.~Lov\'asz.
\newblock Stable sets and polynomials.
\newblock {\em Discrete Mathematics}, 124:137--153, 1994.

\bibitem{morton_pfaff}
J.~Morton.
\newblock Pfaffian circuits.
\newblock {\em arXiv preprint arXiv:1101.0129}, 2011.

\bibitem{morton2013generalized}
J.~Morton and J.~Turner.
\newblock Generalized counting constraint satisfaction problems with
  determinantal circuits.
\newblock {\em arXiv preprint arXiv:1302.1932}, 2013.

\bibitem{turner2013some}
J.~Morton and J.~Turner.
\newblock Some effective results on the invariant theory of tensor networks.
\newblock {\em arXiv preprint arXiv:1310.0370}, 2013.

\bibitem{onn}
S.~Onn.
\newblock Nowhere-zero flow polynomials.
\newblock {\em Journal of Combinatorial Theory, Series A}, 108(2):205--215,
  2004.

\bibitem{valiant_qc}
L.~Valiant.
\newblock Quantum circuits that can be simulated classically in polynomial
  time.
\newblock {\em SIAM Journal on Computing}, 31:1229--1254, 2002.

\bibitem{valiant_ha}
L.~Valiant.
\newblock Holographic algorithms (extended abstract).
\newblock In {\em Proceedings of the 45th annual Symposium on Foundations of
  Computer Science}, pages 306--315, 2004.

\bibitem{singular}
G.~Pfister H.~Sch{\"o}nemann W.~Decker, G.M.~Greuel.
\newblock {\sc Singular} {3-1-6}:{~A} computer algebra system for polynomial
  computations.
\newblock 2012.
\newblock http://www.singular.uni-kl.de.

\end{thebibliography}

\appendix

\section{Listing the 117 4-arity gates of Thm. \ref{thm_bt_pair_hom}} \label{app_thm_bt_pair_hom_4}
\vspace{-25pt}
\scriptsize
\begin{align*}
|0000\rangle + |1111\rangle~, \quad
|1000\rangle + |0111\rangle~, &\quad
|0100\rangle + |1011\rangle~,  \quad
|0010\rangle + |1101\rangle~, \\
|0001\rangle + |1110\rangle~, \quad 
|1100\rangle + |0011\rangle~, &\quad
|1010\rangle + |0101\rangle~, \quad
|1001\rangle + |0110\rangle~,
\end{align*}
\vspace{-25pt}
\begin{align*}
|0000\rangle + |0010\rangle + |1101\rangle + |1111\rangle~, \quad
|0000\rangle + |0001\rangle + |1110\rangle + |1111\rangle~, & \quad
|0000\rangle + |1100\rangle + |0011\rangle + |1111\rangle~, \\
|0000\rangle + |1001\rangle + |0110\rangle + |1111\rangle~, \quad
|1000\rangle + |0100\rangle + |1011\rangle + |0111\rangle~, &\quad
|1000\rangle + |0001\rangle + |1110\rangle + |0111\rangle~, \\
|1000\rangle + |1100\rangle + |0011\rangle + |0111\rangle~, \quad
|1000\rangle + |1010\rangle + |0101\rangle + |0111\rangle~, & \quad
|1000\rangle + |1001\rangle + |0110\rangle + |0111\rangle~, \\
|0100\rangle + |0010\rangle + |1101\rangle + |1011\rangle~, \quad
|0100\rangle + |1100\rangle + |0011\rangle + |1011\rangle~, &\quad
|0100\rangle + |1010\rangle + |0101\rangle + |1011\rangle~, \\
|0100\rangle + |1001\rangle + |0110\rangle + |1011\rangle~, \quad
|0010\rangle + |0001\rangle + |1110\rangle + |1101\rangle~, &\quad
|0010\rangle + |1100\rangle + |0011\rangle + |1101\rangle~, \\
|0010\rangle + |1010\rangle + |0101\rangle + |1101\rangle~,  \quad
|0010\rangle + |1001\rangle + |0110\rangle + |1101\rangle~, &\quad
|0001\rangle + |1100\rangle + |0011\rangle + |1110\rangle~, \\
|0001\rangle + |1010\rangle + |0101\rangle + |1110\rangle~, \quad
|0001\rangle + |1001\rangle + |0110\rangle + |1110\rangle~,  & \quad
|1100\rangle + |1010\rangle + |0101\rangle + |0011\rangle~, \\
|1010\rangle + |1001\rangle + |0110\rangle + |0101\rangle~, \quad
|0000\rangle + |1000\rangle + |0111\rangle + |1111\rangle~, & \quad
|0000\rangle + |0100\rangle + |1011\rangle + |1111\rangle~,
\end{align*}
\vspace{-25pt}
\begin{align*}
|0000\rangle + |1000\rangle + |0001\rangle + |1110\rangle + |0111\rangle + |1111\rangle~, &\quad
|0000\rangle + |1000\rangle + |1100\rangle + |0011\rangle + |0111\rangle + |1111\rangle~, \\
|0000\rangle + |1000\rangle + |1001\rangle + |0110\rangle + |0111\rangle + |1111\rangle~, &\quad
|0000\rangle + |0100\rangle + |0010\rangle + |1101\rangle + |1011\rangle + |1111\rangle~, \\
|0000\rangle + |0100\rangle + |1100\rangle + |0011\rangle + |1011\rangle + |1111\rangle~, &\quad
|0000\rangle + |0100\rangle + |1001\rangle + |0110\rangle + |1011\rangle + |1111\rangle~, \\
|0000\rangle + |0010\rangle + |0001\rangle + |1110\rangle + |1101\rangle + |1111\rangle~, &\quad
|0000\rangle + |0010\rangle + |1100\rangle + |0011\rangle + |1101\rangle + |1111\rangle~, \\
|0000\rangle + |0010\rangle + |1001\rangle + |0110\rangle + |1101\rangle + |1111\rangle~, &\quad
|0000\rangle + |0001\rangle + |1100\rangle + |0011\rangle + |1110\rangle + |1111\rangle~, \\
|0000\rangle + |0001\rangle + |1001\rangle + |0110\rangle + |1110\rangle + |1111\rangle~, &\quad
|1000\rangle + |0100\rangle + |1100\rangle + |0011\rangle + |1011\rangle + |0111\rangle~, \\
|1000\rangle + |0100\rangle + |1010\rangle + |0101\rangle + |1011\rangle + |0111\rangle~, &\quad
|1000\rangle + |0100\rangle + |1001\rangle + |0110\rangle + |1011\rangle + |0111\rangle~, \\
|1000\rangle + |0001\rangle + |1100\rangle + |0011\rangle + |1110\rangle + |0111\rangle~, &\quad
|1000\rangle + |0001\rangle + |1010\rangle + |0101\rangle + |1110\rangle + |0111\rangle~, \\
|1000\rangle + |0001\rangle + |1001\rangle + |0110\rangle + |1110\rangle + |0111\rangle~, &\quad
|1000\rangle + |1100\rangle + |1010\rangle + |0101\rangle + |0011\rangle + |0111\rangle~, \\
|1000\rangle + |1010\rangle + |1001\rangle + |0110\rangle + |0101\rangle + |0111\rangle~, &\quad
|0100\rangle + |0010\rangle + |1100\rangle + |0011\rangle + |1101\rangle + |1011\rangle~, \\
|0100\rangle + |0010\rangle + |1010\rangle + |0101\rangle + |1101\rangle + |1011\rangle~, &\quad
|0100\rangle + |0010\rangle + |1001\rangle + |0110\rangle + |1101\rangle + |1011\rangle~, \\
|0100\rangle + |1100\rangle + |1010\rangle + |0101\rangle + |0011\rangle + |1011\rangle~, &\quad
|0100\rangle + |1010\rangle + |1001\rangle + |0110\rangle + |0101\rangle + |1011\rangle~, \\
|0010\rangle + |0001\rangle + |1100\rangle + |0011\rangle + |1110\rangle + |1101\rangle~, &\quad
|0010\rangle + |0001\rangle + |1010\rangle + |0101\rangle + |1110\rangle + |1101\rangle~, \\
|0010\rangle + |0001\rangle + |1001\rangle + |0110\rangle + |1110\rangle + |1101\rangle~, &\quad
|0010\rangle + |1100\rangle + |1010\rangle + |0101\rangle + |0011\rangle + |1101\rangle~, \\
|0010\rangle + |1010\rangle + |1001\rangle + |0110\rangle + |0101\rangle + |1101\rangle~, &\quad
|0001\rangle + |1100\rangle + |1010\rangle + |0101\rangle + |0011\rangle + |1110\rangle~, \\
|0001\rangle + |1010\rangle + |1001\rangle + |0110\rangle + |0101\rangle + |1110\rangle~, & \quad
|0000\rangle + |1000\rangle + |0100\rangle + |1011\rangle + |0111\rangle + |1111\rangle~,
\end{align*}
\vspace{-25pt}
\begin{align*}
|0000\rangle + |1000\rangle + |0100\rangle + |1001\rangle + |0110\rangle + |1011\rangle + |0111\rangle + |1111\rangle~, \quad 
|0000\rangle + |1000\rangle + |0010\rangle + |1010\rangle + \\ |0101\rangle + |1101\rangle + |0111\rangle + |1111\rangle~, \quad
|0000\rangle + |1000\rangle + |0001\rangle + |1100\rangle + |0011\rangle + |1110\rangle + |0111\rangle + |1111\rangle~, \\
|0000\rangle + |1000\rangle + |0001\rangle + |1001\rangle + |0110\rangle + |1110\rangle + |0111\rangle + |1111\rangle~,  \quad
|0000\rangle + |0100\rangle + |0010\rangle + |1100\rangle + \\ |0011\rangle + |1101\rangle + |1011\rangle + |1111\rangle~, \quad
|0000\rangle + |0100\rangle + |0010\rangle + |1001\rangle + |0110\rangle + |1101\rangle + |1011\rangle + |1111\rangle~,  \\
|0000\rangle + |0100\rangle + |0001\rangle + |1010\rangle + |0101\rangle + |1110\rangle + |1011\rangle + |1111\rangle~, \quad
|0000\rangle + |0010\rangle + |0001\rangle + |1100\rangle + \\ |0011\rangle + |1110\rangle + |1101\rangle + |1111\rangle~,  \quad
|0000\rangle + |0010\rangle + |0001\rangle + |1001\rangle + |0110\rangle + |1110\rangle + |1101\rangle + |1111\rangle~, \\
\end{align*}
\begin{align*}
|1000\rangle + |0100\rangle + |1100\rangle + |1010\rangle + |0101\rangle + |0011\rangle + |1011\rangle + |0111\rangle~,  \quad
|1000\rangle + |0100\rangle + |1010\rangle + |1001\rangle + \\ |0110\rangle + |0101\rangle + |1011\rangle + |0111\rangle~, \quad
|1000\rangle + |0010\rangle + |1100\rangle + |1001\rangle + |0110\rangle + |0011\rangle + |1101\rangle + |0111\rangle~,  \\
|1000\rangle + |0001\rangle + |1100\rangle + |1010\rangle + |0101\rangle + |0011\rangle + |1110\rangle + |0111\rangle~, \quad
|1000\rangle + |0001\rangle + |1010\rangle + |1001\rangle + \\ |0110\rangle + |0101\rangle + |1110\rangle + |0111\rangle~,  \quad
|0100\rangle + |0010\rangle + |1100\rangle + |1010\rangle + |0101\rangle + |0011\rangle + |1101\rangle + |1011\rangle~, \\
|0000\rangle + |1000\rangle + |0100\rangle + |1100\rangle + |0011\rangle + |1011\rangle + |0111\rangle + |1111\rangle~, \quad
|0100\rangle + |0010\rangle + |1010\rangle + |1001\rangle + \\|0110\rangle + |0101\rangle + |1101\rangle + |1011\rangle~,  \quad
|0100\rangle + |0001\rangle + |1100\rangle + |1001\rangle + |0110\rangle + |0011\rangle + |1110\rangle + |1011\rangle~, \\
|0010\rangle + |0001\rangle + |1100\rangle + |1010\rangle + |0101\rangle + |0011\rangle + |1110\rangle + |1101\rangle~,  \\
|0010\rangle + |0001\rangle + |1010\rangle + |1001\rangle + |0110\rangle + |0101\rangle + |1110\rangle + |1101\rangle~,
\end{align*}
\vspace{-20pt}
\begin{align*}
|0000\rangle + |1000\rangle + |0100\rangle + |0010\rangle + |1010\rangle + |0101\rangle + |1101\rangle + |1011\rangle + |0111\rangle + |1111\rangle~, \\
|0000\rangle + |1000\rangle + |0100\rangle + |0001\rangle + |1010\rangle + |0101\rangle + |1110\rangle + |1011\rangle + |0111\rangle + |1111\rangle~, \\
|0000\rangle + |1000\rangle + |0010\rangle + |0001\rangle + |1010\rangle + |0101\rangle + |1110\rangle + |1101\rangle + |0111\rangle + |1111\rangle~, \\
|0000\rangle + |1000\rangle + |0010\rangle + |1100\rangle + |1010\rangle + |0101\rangle + |0011\rangle + |1101\rangle + |0111\rangle + |1111\rangle~, \\
|0000\rangle + |1000\rangle + |0010\rangle + |1100\rangle + |1001\rangle + |0110\rangle + |0011\rangle + |1101\rangle + |0111\rangle + |1111\rangle~, \\
|0000\rangle + |1000\rangle + |0010\rangle + |1010\rangle + |1001\rangle + |0110\rangle + |0101\rangle + |1101\rangle + |0111\rangle + |1111\rangle~, \\
|0000\rangle + |0100\rangle + |0010\rangle + |0001\rangle + |1010\rangle + |0101\rangle + |1110\rangle + |1101\rangle + |1011\rangle + |1111\rangle~, \\
|0000\rangle + |0100\rangle + |0001\rangle + |1100\rangle + |1010\rangle + |0101\rangle + |0011\rangle + |1110\rangle + |1011\rangle + |1111\rangle~, \\
|0000\rangle + |0100\rangle + |0001\rangle + |1100\rangle + |1001\rangle + |0110\rangle + |0011\rangle + |1110\rangle + |1011\rangle + |1111\rangle~, \\
|0000\rangle + |0100\rangle + |0001\rangle + |1010\rangle + |1001\rangle + |0110\rangle + |0101\rangle + |1110\rangle + |1011\rangle + |1111\rangle~, \\
|1000\rangle + |0100\rangle + |0010\rangle + |1100\rangle + |1001\rangle + |0110\rangle + |0011\rangle + |1101\rangle + |1011\rangle + |0111\rangle~, \\
|1000\rangle + |0100\rangle + |0001\rangle + |1100\rangle + |1001\rangle + |0110\rangle + |0011\rangle + |1110\rangle + |1011\rangle + |0111\rangle~, \\
|1000\rangle + |0010\rangle + |0001\rangle + |1100\rangle + |1001\rangle + |0110\rangle + |0011\rangle + |1110\rangle + |1101\rangle + |0111\rangle~, \\
|1000\rangle + |0010\rangle + |1100\rangle + |1010\rangle + |1001\rangle + |0110\rangle + |0101\rangle + |0011\rangle + |1101\rangle + |0111\rangle~, \\
|0100\rangle + |0010\rangle + |0001\rangle + |1100\rangle + |1001\rangle + |0110\rangle + |0011\rangle + |1110\rangle + |1101\rangle + |1011\rangle~, \\
|0100\rangle + |0001\rangle + |1100\rangle + |1010\rangle + |1001\rangle + |0110\rangle + |0101\rangle + |0011\rangle + |1110\rangle + |1011\rangle~, \\
%
%
|0000\rangle + |1000\rangle + |0100\rangle + |0010\rangle + |1100\rangle + |1010\rangle 
+|0101\rangle + |0011\rangle + |1101\rangle + |1011\rangle + |0111\rangle + |1111\rangle~, \\
|0000\rangle + |1000\rangle + |0100\rangle + |0010\rangle + |1100\rangle + |1001\rangle 
+ |0110\rangle + |0011\rangle + |1101\rangle + |1011\rangle + |0111\rangle + |1111\rangle~, \\
|0000\rangle + |1000\rangle + |0100\rangle + |0010\rangle + |1010\rangle + |1001\rangle 
+ |0110\rangle + |0101\rangle + |1101\rangle +|1011\rangle + |0111\rangle + |1111\rangle~, \\
|0000\rangle + |1000\rangle + |0100\rangle + |0001\rangle + |1100\rangle + |1010\rangle 
+ |0101\rangle + |0011\rangle + |1110\rangle +|1011\rangle + |0111\rangle + |1111\rangle~, \\
|0000\rangle + |1000\rangle + |0100\rangle + |0001\rangle + |1100\rangle + |1001\rangle 
+ |0110\rangle + |0011\rangle + |1110\rangle +|1011\rangle + |0111\rangle + |1111\rangle~, \\
|0000\rangle + |1000\rangle + |0100\rangle + |0001\rangle + |1010\rangle + |1001\rangle 
+ |0110\rangle + |0101\rangle + |1110\rangle + |1011\rangle + |0111\rangle + |1111\rangle~, \\
|0000\rangle + |1000\rangle + |0010\rangle + |0001\rangle + |1100\rangle + |1010\rangle 
+ |0101\rangle + |0011\rangle + |1110\rangle + |1101\rangle + |0111\rangle + |1111\rangle~, \\
|0000\rangle + |1000\rangle + |0010\rangle + |0001\rangle + |1100\rangle + |1001\rangle 
+ |0110\rangle + |0011\rangle + |1110\rangle + |1101\rangle + |0111\rangle + |1111\rangle~, \\
|0000\rangle + |1000\rangle + |0010\rangle + |0001\rangle + |1010\rangle + |1001\rangle 
+ |0110\rangle + |0101\rangle + |1110\rangle + |1101\rangle + |0111\rangle + |1111\rangle~, \\
|0000\rangle + |0100\rangle + |0010\rangle + |0001\rangle + |1100\rangle + |1010\rangle 
+|0101\rangle + |0011\rangle + |1110\rangle + |1101\rangle + |1011\rangle + |1111\rangle~, \\
|0000\rangle + |0100\rangle + |0010\rangle + |0001\rangle + |1100\rangle + |1001\rangle 
+ |0110\rangle + |0011\rangle + |1110\rangle + |1101\rangle + |1011\rangle + |1111\rangle~, \\
|0000\rangle + |0100\rangle + |0010\rangle + |0001\rangle + |1010\rangle + |1001\rangle 
+ |0110\rangle + |0101\rangle + |1110\rangle + |1101\rangle + |1011\rangle + |1111\rangle~, \\
|1000\rangle + |0100\rangle + |0010\rangle + |1100\rangle + |1010\rangle + |1001\rangle 
+ |0110\rangle + |0101\rangle + |0011\rangle + |1101\rangle + |1011\rangle + |0111\rangle~, \\
|1000\rangle + |0100\rangle + |0001\rangle + |1100\rangle + |1010\rangle + |1001\rangle 
+ |0110\rangle + |0101\rangle + |0011\rangle + |1110\rangle + |1011\rangle + |0111\rangle~, \\
|1000\rangle + |0010\rangle + |0001\rangle + |1100\rangle + |1010\rangle + |1001\rangle 
+ |0110\rangle + |0101\rangle + |0011\rangle + |1110\rangle + |1101\rangle + |0111\rangle~, \\
|0100\rangle + |0010\rangle + |0001\rangle + |1100\rangle + |1010\rangle + |1001\rangle 
+ |0110\rangle + |0101\rangle + |0011\rangle + |1110\rangle + |1101\rangle + |1011\rangle~, \\
|0000\rangle + |1000\rangle + |0100\rangle + |0010\rangle + |0001\rangle + |1100\rangle + |1010\rangle + |1001\rangle + \\
|0110\rangle + |0101\rangle + |0011\rangle + |1110\rangle + |1101\rangle + |1011\rangle + |0111\rangle + |1111\rangle~.
\end{align*}
\normalsize

\end{document}